\documentclass[11pt,letterpaper,reqno]{amsart}
\usepackage[margin=1in]{geometry}
\usepackage{color}
\usepackage[latin1]{inputenc}
\usepackage{amsmath,amsfonts,amsthm,epsfig,graphicx,amssymb}
\usepackage{esint}
\usepackage{caption}
\usepackage{stmaryrd}
\usepackage[normalem]{ulem}
\usepackage[colorlinks=true, linkcolor=blue, citecolor=blue]{hyperref}
\usepackage{overpic}

\renewcommand{\H}{\mathcal{H}}

\newcommand{\V}{\mathcal{V}}

\newcommand{\N}{\mathbb{N}}

\newcommand{\R}{\mathbb{R}}
\renewcommand{\SS}{\mathbb{S}}

\newcommand{\Om}{\Omega}

\newcommand{\e}{\varepsilon}

\newcommand{\vphi}{\varphi}
\newcommand{\Lip}{{\rm Lip}}

\newcommand{\Div}{{\rm div}\,}
\newcommand{\Id}{{\rm Id}\,}
\newcommand{\dist}{{\rm dist}}

\newcommand{\loc}{{\rm loc}}
\newcommand{\diam}{{\rm diam}\,}

\newcommand{\spt}{{\rm spt}}

\newcommand{\weak}{\rightharpoonup}
\newcommand{\weakstar}{\stackrel{\scriptscriptstyle{*}}{\rightharpoonup}}

\newcommand{\pa}{\partial}

\newcommand{\cc}{\subset\!\subset}

\newcommand{\cl}{\mathrm{cl}\,}

\newcommand{\eps}{\e}

\newcommand{\dive}{\mathrm{div}\,}

\newcommand{\h}{\H^n}
\newcommand{\rn}{\mathbb{R}^{n+1}}

\newcommand{\ac}{\mathcal{E}}

\newcommand\restr[2]{{
		\left.\kern-\nulldelimiterspace 
		#1 
		\right|_{#2} 
}}

\newcommand{\one}{{\scriptscriptstyle{(1)}}}
\newcommand{\zero}{{\scriptscriptstyle{(0)}}}
\newcommand{\half}{{\scriptscriptstyle{(1/2)}}}

\newcommand{\toweak}{\rightharpoonup}

\newcommand{\Acal}{\mathcal{A}}
\newcommand{\Ccal}{\mathcal{C}}
\newcommand{\Hcal}{\mathcal{H}}
\newcommand{\Lcal}{\mathcal{L}}
\newcommand{\Rcal}{\mathcal{R}}
\newcommand{\Scal}{\mathcal{S}}
\newcommand{\Fcal}{\mathcal{F}}

\newcommand{\Wbf}{\mathbf{W}}
\newcommand{\Sbb}{\mathbb{S}}

\newcommand{\mres}{\mathbin{\vrule height 1.6ex depth 0pt width 
		0.13ex\vrule height 0.13ex depth 0pt width 1.3ex}}

\theoremstyle{plain}
\newtheorem{theorem}{Theorem}[section]
\newtheorem{lemma}[theorem]{Lemma}
\newtheorem{corollary}[theorem]{Corollary}
\newtheorem{proposition}[theorem]{Proposition}
\newtheorem*{theorem*}{Theorem}
\newtheorem*{corollary*}{Corollary}

\theoremstyle{definition}
\newtheorem{definition}{Definition}[section]

\newtheorem{remark}[theorem]{Remark}

\newtheorem*{notation*}{Notation}

\numberwithin{equation}{section}
\numberwithin{figure}{section}
\newcommand{\id}{{\rm id}\,}

\newcommand{\wire}{\mathbf{W}}

\allowdisplaybreaks

\usepackage{graphicx} 

\title{Free boundary regularity for semilinear variational problems with a topological constraint}
\author{Michael Novack}
\address{Department of Applied Mathematics, Illinois Institute of Technology, W 32nd St., E1 Room 208, 10 S Wabash Ave, Chicago, IL 60616, United States of America}
\email{mnovack@illinoistech.edu}
\author{Daniel Restrepo}
\address{Department of Mathematics, Johns Hopkins University, 3400 N. Charles Street, Baltimore, MD 21218, United States of America}
\email{drestre1@jh.edu}
\author{Anna Skorobogatova}
\address{Department of Mathematics, ETH Z\"{u}rich, R\"{a}mistrasse 101, 8092 Z\"{urich}, Switzerland}
\email{anna.skorobogatova@math.ethz.ch}

\begin{document}
	
	\begin{abstract}
		We study a class of semilinear free boundary problems in which admissible functions $u$ have a topological constraint, or {\it spanning condition}, on their 1-level set. This constraint forces $\{u=1\}$, which is the free boundary, to behave like a surface with some special types of singularities attached to a fixed boundary frame, in the spirit of the Plateau problem \cite{HP16}. 
        Two such free boundary problems are the minimization of capacity among surfaces sharing a common boundary and an Allen-Cahn formulation of the Plateau problem. We establish the existence of minimizers and study their regularity properties, obtaining the optimal Lipschitz regularity of minimizers and analytic regularity for the free boundaries away from a codimension two singular set. The singularity models for these problems are given by conical critical points of the minimal capacity problem, which are closely related to spectral optimal partition and segregation problems.
	\end{abstract}
	
	\maketitle

 \tableofcontents
    
	\section{Introduction}
	
	\subsection{Background}\label{subsec:overview}
	In this paper we study the regularity of solutions to elliptic variational problems with a topological constraint on the level sets of admissible functions. Given a compact set $\wire\subset\R^{n+1}$ with $\Om = \rn \setminus \wire$ and potentials $F,V:[0,1]\to[0,\infty)$ vanishing at $0$, consider the minimization problems
	\begin{eqnarray}
		\label{new model intro}
		\inf\left\{\int_{\Omega} |\nabla u|^2 +F(u)\,:\,\begin{aligned} &u\in C^0(\Om;[0,1]),\,{\nabla u \in L^2_\loc(\Omega),} \\ &\mbox{$u$ vanishes at infinity, $\{u = 1\}$ ``spans" $\wire$}\end{aligned}\ \right\}\,,
	\end{eqnarray}
	and
	\begin{eqnarray}
		\label{new model volume constrained}  
		\inf\left\{\int_{\Omega} |\nabla u|^2 +F(u)\,:\, \begin{aligned} &u\in C^0(\Om;[0,1]),\,{\nabla u \in L^2_\loc(\Omega),} \\ &\int_{\Omega}V(u)=1,\,\mbox{$\{u = 1\}$ ``spans" $\wire$} \end{aligned}\right\}\,.
	\end{eqnarray}
	Here the terminology ``$\{u=1\}$ spans $\wire$" means that for a homotopically closed family $\mathcal{C}$ of smooth embeddings of $\mathbb{S}^1$ in $\Omega$ (which is independent of $u$), called a spanning class,
	\begin{equation}\label{eq:HP spanning}
		\{u=1\} \cap \gamma\neq\varnothing\qquad \mbox{for every $\gamma \in \mathcal{C}$}\,.
	\end{equation}
	We will refer to any set satisfying \eqref{eq:HP spanning} as $\mathcal{C}$-{\bf spanning}. This type of condition originated in the study of the set-theoretic Plateau problem \cite{HP16} and at a heuristic level forces the spanning set to behave like a surface bounded by $\wire$; see Figure \ref{fig:spanning}.

	Several examples of these types of problems have appeared in the literature and motivate our work. An early version of the model \eqref{new model intro}--with a slightly different notion of spanning--and $F=0$ is the classical problem of finding \textit{surfaces of minimal capacity} spanning a closed curve. In the case when $\wire$ is a Jordan curve satisfying some additional restrictions, this problem was solved in $\R^3$ in \cite{evans1940minimal, evans1940minimum} using multivalued harmonic functions. A similar method was used in \cite{caffarelli1975surfaces} to address the case when $\wire$ is a knot, yielding however only local minimizers. Most relevant to our choice of spanning condition is the diffuse interface/Allen-Cahn approximation of the Plateau problem recently introduced in \cite{maggi2023hierarchy} and which is a prototypical example of \eqref{new model volume constrained}. In that case, $F=W/\e^2$ is a double-well potential (e.g. $F =u^2(u-1)^2/\e^2$) and $V$ is a particular volume potential related to $F$ (see \eqref{eq:v def}). It is shown in \cite{maggi2023hierarchy} that the rescaled problems 
	\begin{equation}\label{eq:MNR23a problem}
	\inf\Big\{\eps\int_{\Omega} |\nabla u|^2 +\frac{W(u)}{\eps^2}\, dx \,:  \int_{\Omega}V(u)=v,\,\mbox{$\{u = 1\}$ is $\mathcal{C}$-spanning} \Big\}\,,
	\end{equation}
	converge as $\e, v\to 0$, $\eps \ll v$, to the Plateau problem of Harrison-Pugh \cite{HP16}. Finally, although it is not immediately obvious from the statements of \eqref{new model intro}-\eqref{new model volume constrained}, they share some common features with optimal partition/segregation problems from the free boundary literature; see Remark \ref{remark:segregation}.
	
	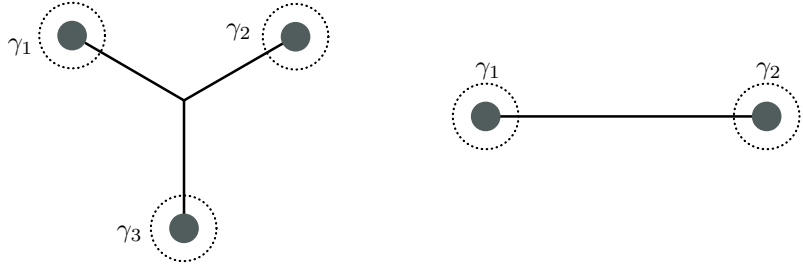
\begin{figure}
		\begin{overpic}[scale=0.6,unit=1mm]{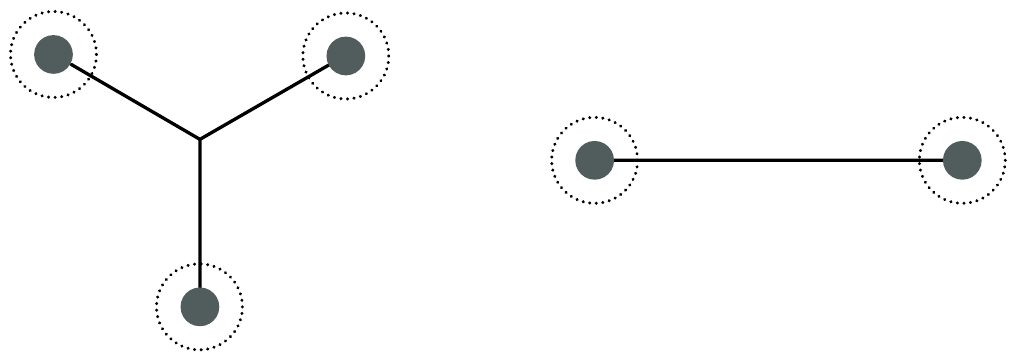}
			\put(57,25){\small{$\gamma_1$}}
			\put(-3,28){\small{$\gamma_1$}}
			\put(25,30){\small{$\gamma_2$}}
			\put(93,25){\small{$\gamma_2$}}
			\put(11,4){\small{$\gamma_3$}}
		\end{overpic}
		\caption{Shown above are two different configurations of $\wire\subset \mathbb{R}^2$, generators for an associated spanning class $\mathcal{C}$, and a spanning set. In both cases, $\wire$ is the union of the gray balls, $\mathcal{C}$ is the family of smooth loops homotopic to some $\gamma_i$, and the example spanning sets are composed of line segments.}\label{fig:spanning}
	\end{figure}
	
	As is to be expected, the spanning constraint significantly affects the behavior of minimizers and the corresponding analysis. For example, with respect to \eqref{new model intro}, in combination with the requirement that $u$ vanishes at infinity it eliminates from contention all constant functions and thus forces non-constant minimizers. Also, for \eqref{new model volume constrained}, it allows for the approximation of minimal surfaces with codimension 1 singularities, for example triple junctions in the plane, by energy minimizing solutions of an Allen-Cahn free boundary problem. This, however, is not possible when considering stationary, stable solutions to the classical Allen-Cahn equation \cite{TW12}. At the level of the Euler-Lagrange equation, the spanning constraint significantly changes even its derivation. This is due to the fact that when $\{u=1\}$ is $\mathcal{C}$-spanning and the test function $\varphi$ takes negative values, there is no reason that $\{u + t\varphi=1\}$ is $\mathcal{C}$-spanning and admissible for \eqref{new model intro}-\eqref{new model volume constrained}. Setting aside this consideration for the moment, \eqref{new model intro}-\eqref{new model volume constrained} formally lead to a free boundary problem with a transmission condition. If $u$ is minimizer and $u$ and the level set $\{u=1\}$ are sufficiently regular, then there exists a potential $\Phi$ such that $u$ solves the free boundary problem
	\begin{equation}\label{ace modified}
		\left\{
		\begin{aligned}
			2\,\Delta u&=\Phi'(u)\,,&\quad\mbox{on $\Om\cap\{u \neq 1\}$}\,,
			\\
			|\pa_\nu^+u |&=|\pa_\nu^-u|\,,&\quad\mbox{on $\Om\cap\{u = 1 \}$}\,,
			\\
			\mbox{such }&\mbox{that $\{u = 1 \}$ spans $\wire$}\,,
		\end{aligned}
		\right .
	\end{equation}
	where $\pa_\nu^\pm$ denote the one-sided directional derivative operators with respect to a unit normal $\nu$ to $\{u=1\}$ (cf.~ \cite[Prop. 1.4]{maggi2023hierarchy}). In the case of \eqref{new model intro}, $\Phi = F$, whereas in \eqref{new model volume constrained} $\Phi= F-\lambda V$, with $\lambda \in \R$ a Lagrange multiplier associated with the volume constraint $\int V(u)=1$. The interested reader may refer to \cite[Eq. 1-2]{evans1940minimal} regarding the significance of the transmission condition $|\pa_\nu^+u |=|\pa_\nu^-u|$ and its integral formulation in the simplest case $F=0$ with no volume constraint. 
	
	\subsection{Main results}\label{subsec:main results}
	Throughout the paper, we will assume that: $F$ and $V$ are potential functions satisfying the hypotheses
	\begin{align}\tag{H1}\label{eq:A1}
		&F,V \in C^{2}([0,1];[0,\infty))\,   \\ \tag{H2} \label{eq:A2}
		&0=F(0)=V(0)=F'(0)=V'(0)=V'(1)=F'(1)\,,\quad \textup{and} \\ \tag{H3}\label{eq:A3}
		&\mbox{$V$ is strictly increasing on (0,1)}\,.
	\end{align}
	For existence of minimizers in the presence of the volume constraint, we will also assume that
	\begin{equation}\tag{H4}\label{eq:A4}
		\lim_{t\to 0}\frac{V(t)}{F(t)}=0\,,
	\end{equation}
	which is mild and satisfied for example in the Allen-Cahn setting \cite{maggi2023hierarchy} where
	\begin{equation}\label{eq:v def}
		V(t) = {\mathcal{F}}(t)^{(n+1)/n}\,,\qquad {\mathcal{F}}(t) = \int_0^t \sqrt{F(s)}\,ds   \,.
	\end{equation}
	{Lastly, for the homotopically closed family of smooth loops $\mathcal{C}$ (the spanning class), we assume that
	\begin{equation}	
        \begin{aligned}\label{eq:spanning class assumption}
			&\mbox{no $\gamma\in \mathcal{C}$ is homotopic in $\Om$ to a point when $n\geq 2$, and}  \\
            &\mbox{no $\gamma\in \mathcal{C}$ is homotopic in $\Om$ to a point, or to $\partial B_R$ if $\wire \subset B_R$, when $n=1$\,.}
		\end{aligned}
        \end{equation}
		This assumption is sharp in the following sense: if $\mathcal{C}$ contains such a curve $\gamma$ homotopic (in $\Omega)$ to a point $x$, the problems \eqref{new model intro}-\eqref{new model volume constrained} are trivial on the connected component $\Om'$ of $\Om$ containing $x$, since any $u$ satisfying \eqref{eq:HP spanning} must be $1$ on $\Om'$. Also, in the plane, if there exists $\gamma\in \mathcal{C}$ homotopic to $\partial B_R$ with $\wire \subset B_R$, then the admissible class is empty in \eqref{new model intro} or the infimum is infinite in \eqref{new model volume constrained}. So there is nothing lost in assuming \eqref{eq:spanning class assumption}.}
	
 
	{Our main results are sharp regularity and existence of minimizers {(in the admissible class)} for \eqref{new model intro} and \eqref{new model volume constrained}.}
	
	\begin{theorem}[Regularity of minimizers]\label{thm:main regularity theorem}
		If $\wire = \mathbb{R}^{n+1}\setminus \Omega$ is compact, $\mathcal{C}$ is a spanning class for $\wire$ satisfying \eqref{eq:spanning class assumption}, $F$ and $V$ satisfy \eqref{eq:A1}-\eqref{eq:A3}, and {$u$ is a minimizer of \eqref{new model intro} or \eqref{new model volume constrained}}, then
		
		\smallskip
		
		\noindent{\bf (i):} $u$ is locally Lipschitz in $\Omega$, and
		
		\smallskip
		
		\noindent{\bf (ii):} on $\{\Om'\subset \Om:\mbox{$\Om'$ is a connected component of $\Om$ not contained in $\{u=1\}$}\}$, the free boundary $\{u=1\}$ {has empty interior in $\R^{n+1}$ and} decomposes as
		\begin{equation}\notag
			\{u=1\} =  \Rcal(u) \sqcup \mathcal{S}(u) \,,
		\end{equation}
		where $\mathcal{R}(u)$ is locally an {analytic} $n$-dimensional manifold and $\mathcal{S}(u)$ is a closed set with Hausdorff dimension at most $n-1$. If $n=1$, $\{u=1\}$ consists in a locally finite number of analytic curves meeting with equal angles at a discrete number of singular points.
	\end{theorem}

	\begin{theorem}[Existence of minimizers]\label{thm:main existence theorem}
		Suppose $\wire =\mathbb{R}^{n+1}\setminus \Omega$ is compact, $\mathcal{C}$ is a spanning class for $\wire$ satisfying \eqref{eq:spanning class assumption}, and $F$ and $V$ satisfy \eqref{eq:A1}-\eqref{eq:A4}. Then there exists a minimizer {$u \in C^0(\Omega;[0,1])$ with $\nabla u \in L^2_\loc(\Omega)$} for \eqref{new model volume constrained}, and if either
\begin{equation}\label{eq:sharp existence assumption 2}
  n\geq 2\qquad \mbox{or}\qquad  \mbox{there exists $t_j\searrow 0$ such that $F(t_j)>0$},
\end{equation}
there exists a minimizer {$u \in C^0(\Omega;[0,1])$ with $\nabla u \in L^2_\loc(\Omega)$} for \eqref{new model intro}.
	\end{theorem}

	\noindent  The proof of Theorem \ref{thm:main regularity theorem} comprises the bulk of the paper, and the Lipschitz continuity is used in the proof of the existence of minimizers for \eqref{new model volume constrained}, which is why we have opted to state it first. {We also point out that the precise value assigned by the volume constraint in \eqref{new model volume constrained} does not play any role in our proofs. Indeed, we can change the volume constraint from $1$ to any $m>0$ in \eqref{new model volume constrained} simply by considering the volume potential $\tilde{V} = V/m$, which satisfies \eqref{eq:A1}-\eqref{eq:A3} if $V$ does. }

 \begin{remark}[Optimality of the assumptions in Theorem \ref{thm:main existence theorem}]
     If $n=1$ and $F(t) =0$ for $t\in [0,t_0]$, then using logarithmic cutoffs it can be shown there does not exist a minimizer for \eqref{new model intro}; see Remark \ref{remark:optimality of ex assumptions}. An alternative in this case would be to solve the problem on a bounded domain $\Om$ with vanishing Dirichlet conditions; see Remark \ref{remark:bdd domains}.
 \end{remark}

	\begin{remark}[On the free boundary decomposition]
		Even with the assumption \eqref{eq:spanning class assumption} on $\mathcal{C}$, part (ii) of Theorem \ref{thm:main regularity theorem} is optimal in terms of its restriction to 
        \[
            \{\Om'\subset \Om:\mbox{$\Om'$ is a connected component of $\Om$ not contained in $\{u=1\}$}\}\,.
        \]
        {For instance, when $n\geq 2$,} one can consider the case that $\wire$ is the union of a solid torus $T$ and some $\overline{B_r \setminus B_s}$ with $T \cc B_s$, the spanning class $\mathcal{C}$ consists of curves contained in $B_s$ whose linking number with $T$ is $1$, and the potential $F$ vanishes at $1$, the global minimizer for \eqref{new model intro} is $1$ on $B_s\setminus T$ and $0$ on $\overline{B_r}^c$; it has zero energy. The same configuration is obviously minimizing in the additional presence of a volume constraint if in addition $\int_{B_s\setminus T}V(1)=1$. {An analogous example can be constructed when $n=1$ with $\Wbf=\partial B_r \cup \{0\}$, and the spanning class $\Ccal$ consisting of curves contained $B_r(0)\setminus \{0\}$ which have winding number one around the origin.} However, as long as $\mathbb{R}^{n+1}\setminus \wire$ does not have any bounded connected components, minimizers for \eqref{new model intro} or \eqref{new model volume constrained} cannot be locally constant. 
	\end{remark}

	\subsection{Discussion}\label{subsec:discussion}
	
	We begin with a detailed outline of the proofs of Theorems \ref{thm:main regularity theorem}-\ref{thm:main existence theorem} and comments on the key ideas.
	
	\subsubsection{Commentary on proofs and structure of article}

{The mathematical content of the article is divided into two blocks. The first one is given by Section \ref{s:freq-monotonicity} and Section \ref{s:fb-structure}, where we study the regularity of suitably well-behaved weak solutions to \eqref{ace modified}; while the second part, given by Section \ref{s:proof-main-theorems} and the appendices, is devoted to show that minimizers of the variational problems \eqref{new model volume constrained} and \eqref{new model intro} exist and are well-behaved weak solutions to \eqref{ace modified} to which the regularity theory developed in the first block apply. From a technical perspective, the methods of the first block are rather general and close to those used in the study of nodal sets of harmonic functions, optimal partitions and segregation problems (see Remark \ref{remark:segregation}); in fact, the spanning condition does not play any role in this part of the paper. The second part contains the theoretical framework and machinery tailored to the study of variational problems with homotopic spanning conditions, building upon the generalization of spanning from \cite{maggi2023plateau,maggi2023hierarchy}, and which is one of the technical cornerstones of the paper.}
        
	Since \eqref{ace modified} is completely formal, in order to rigorously analyze minimizers we do not appeal directly to it but instead base our arguments on several other properties of minimizers of \eqref{new model intro} and \eqref{new model volume constrained}. In Section \ref{s:freq-monotonicity}, we begin by presenting three criticality conditions, namely (the weak formulations of) the outer variation equation
		\begin{equation}\notag
			(1-u)\,\big\{2\Delta u-\Phi'(u)\big\}=0\qquad\mbox{in $\Omega$},
		\end{equation}
		the inner variation equation for $u$, and the differential inequality
		\begin{align*}
			2\Delta u \leq \Phi'(u)
		\end{align*}
		(see \eqref{eq innervariation}-\eqref{eq:differential inequality} and Remark \ref{remark:criticality remark}). {We will defer the derivation of these conditions for minimizers of \eqref{new model intro} and \eqref{new model volume constrained} until Section \ref{s:proof-main-theorems}.} Since this set of conditions comes from considering only first order variations of either \eqref{new model intro} or \eqref{new model volume constrained}, we adopt them as our notion of critical points or stationary solutions for these models. We then prove Theorem \ref{thm:main regularity theorem}-(i) for critical points of the problem {under the additional assumption of a uniform lower bound of Almgren's frequency function for $v=1-u$, which we later verify the validity of for minimizers in Section \ref{s:proof-main-theorems}.} Under this latter transformation, the new criticality conditions are given by \eqref{eq modified iv}-\eqref{eq modified ineq}, allowing us to see solutions as ``almost''-subharmonic functions satisfying a weak transmission condition at their zero set, and guaranteeing almost-monotonicity of the frequency function. Under this reformulation, our model is quite similar to those in optimal partition/segregation problems, and the arguments utilize some similar tools. More precisely, the proof of the Lipschitz regularity for stationary solutions {satisfying a uniform lower frequency bound (Theorem \ref{thm:holder reg for crit points})} fundamentally relies on the fact that conditions \eqref{eq modified iv}-\eqref{eq modified ineq} are enough to guarantee the almost-monotonicity of the Almgren frequency function which unlocks a series of monotonicity and unique continuation type properties that help us obtain regularity (see, e.g., \cite{CL2007, CL2008}). {Since a uniform lower bound on the frequency function in turn implies that $v$ is H\"{o}lder continuous, the key step in this part of our analysis is improving the lower frequency bound to 1 at all points in $\{v=0\}$, therefore improving the H\"{o}lder regularity to Lipschitz regularity.} This is achieved via a dimension reduction argument and a blow-up analysis. Note that since the absolute value of any harmonic function satisfies locally \eqref{eq modified iv}-\eqref{eq modified ineq} (with $G=0$), Lipschitz regularity is the sharp regularity of stationary solutions.
			
		In Section \ref{s:fb-structure}, we show Theorem \ref{thm:main regularity theorem}-(ii) again for any critical point $u$ satisfying the uniform lower frequency bound. {The starting point is the fact that points in $\{u=1\}$ with frequency greater than 1 actually have frequency greater than or equal to $\tfrac{3}{2}$. This frequency gap is sharp since the frequency $\frac{3}{2}$ is attained at triple junctions as in Figure \ref{fig:spanning}, and was obtained in the work \cite{SoaveTerracini}.} From this, we can split $\{u=1\}=  \Rcal(u) \sqcup \mathcal{S}(u)$  where $\Rcal(u)$ consists of the points in $\{u=1\}$ where $v=1-u$ has frequency value 1, and $\mathcal{S}(u)$ consists of those points in $\{u=1\}$ at which $v$ has frequency value greater or equal than $\tfrac{3}{2}$. After this point, we show that $\Rcal(u)$ is a regular manifold and we derive estimates on the dimension of $\mathcal{S}(u)$. These latter arguments are standard and essentially the same as in \cite{TT} (see also \cite{CL2007, CL2008}).  {Observe that all of the results in Section \ref{s:freq-monotonicity} and Section \ref{s:fb-structure} rely only on the validity of the Euler--Lagrange equations and on the uniform lower frequency bound, so the Lipschitz regularity of solutions and the partial regularity of their free boundary can be concluded conditionally on these hypotheses.}

        {The main body of Section \ref{s:proof-main-theorems} is primarily devoted to the proof of the lower frequency bound for minimizers of \eqref{new model intro} and \eqref{new model volume constrained}. Given a solution  $v=1-u$ to the Euler-Lagrange equations \eqref{eq modified iv}-\eqref{eq modified ineq}, we observe that a lower bound on the frequency at a point $x_0 \in \Omega$ holds if, in some sense, the free boundary $\{v=0\}$ disconnects $B_r(x_0)$ (for $r$ sufficiently small) and $\{v>0\}\cap B_r(x_0)$ has at least two comparable (in size) connected components. Motivated by this observation, we provide a suitable definition of \textit{essentially connected components} of $\{v>0\}$ for functions $v  \in W^{1,2}_{\rm loc}(\Omega;[0,1])$ satisfying a generalized spanning condition, see Lemma \ref{lemma connected 2}, which we exploit to  relax the variational problems to a natural larger admissible class, which remains closed under taking ``cup competitors" (see Definition \ref{d:cup-comp} and Remark \ref{remark cupcompetitor}). This strategy is analogous to the one in \cite{DGM17}, where the authors provide an alternative proof to the existence of solutions for the Harrison-Pugh formulation of the Plateau problem. We then proceed to derive the aforementioned Euler-Lagrange equations and the uniform lower frequency bound for minimizers of these relaxed problems. The validity of the lower frequency bound for minimizers of the relaxed variational problems, i.e.~ \eqref{new model intro 2} and \eqref{new model volume constrained 2}, is then a consequence of an energy comparison argument in order to rule out the existence of one (essentially) connected component of $v$ being much larger than all the others in some ball.

		With the Euler-Lagrange equations and lower frequency bound at hand, the regularity theory from Sections \ref{s:freq-monotonicity} and \ref{s:fb-structure} applies, thus yielding interior Lipschitz continuity and free boundary regularity for minimizers to the relaxed variational problems. This in turn implies that any such minimizer is a minimizer for the corresponding original variational problems \eqref{new model intro} or \eqref{new model volume constrained}, giving therefore the conclusions of Theorem \ref{thm:main regularity theorem} for minimizers of our original problems. Lastly, as a corollary of the Lipschitz continuity, we prove Theorem \ref{thm:main existence theorem} in Section \ref{s:existence}. The main difficulty in this case is to rule out volume loss at infinity in \eqref{new model volume constrained}. Our proof relies on the simple observation that limits of minimizing sequences to \eqref{new model volume constrained} are still minimizers to a problem of the same form with a (possibly) different volume constraint, which allows us to deduce continuity and uniform decay at infinity from Theorem \ref{thm:main regularity theorem}. With this knowledge at hand, we are able to rule out volume loss using an energy comparison argument.}

        

	
	\subsubsection{Further remarks}\label{subsubsec:further remarks}
	
	Here we collect some final observations regarding our results and related problems in the literature.
	
	\begin{remark}[Connection to optimal partition \& segregation problems]\label{remark:segregation}
		As pointed out above, models \eqref{new model intro} and \eqref{new model volume constrained} share several similarities with those stemming from optimal partition and segregation problems; this is mainly due to the fact that the homotopic spanning condition \eqref{eq:HP spanning} imposes a local separation property at each point $x_0 \in \{u=1\}$. More precisely, it forces $\{u=1\}$ to disconnect any small ball centered at $x_0$. This, coupled with the variational nature of \eqref{new model intro} and \eqref{new model volume constrained}, suggests that in $\{u=1\}\cap B_r(x_0)$ (for $r$ small) we should see an optimal partition of the connected components of $\{u<1\}\cap B_r(x_0)$. However, note that in contrast to general optimal partition problems, the number of these connected components is not prescribed and could be infinitely many a priori. Only in some special cases when we show that nearby certain free boundary points $\{u<1\}$ has finitely many components (see Lemma \ref{lemma density components}), is our model locally equivalent to the general framework of \cite{TT}. Indeed, after applying Theorem \ref{thm:main regularity theorem} to obtain the Lipschitz regularity of our solutions and localizing to any ball $B\subset \Omega$ for which there are finitely many connected components of $\{u<1\}$, we observe that for $v=1-u$, the functions $v_i = v\mid_{U_i}$ for connected components $U_1,\dots, U_{N-1}$ of $\{v>0\}$ and $U_N = B\setminus \bigcup_{i=1}^{N-1} U_i$, satisfy the hypotheses of the main theorem therein.
  
		We also point out that the same methods used in Section \ref{s:freq-monotonicity} show that {the Lipschitz continuity assumption in the class of functions considered in \cite[Theorem 1.1]{TT} can be relaxed to a lower frequency bound assumption of the form \eqref{eq:lower freq bound assumption theorem 2.1}.} This observation is immediate if the forcing term $f(x,u)$ considered in \cite[Theorem 1.1]{TT} is $x$-independent, since in that case it directly falls under our hypotheses. In the case of $x$ dependence, the result holds from straightforward adaptations of our arguments, exploiting fundamentally the assumption $|f(x,u)|\leq Cu$ in \cite{TT} (compare with \eqref{eq:g assumptions}), combined with the observation that the hypothesis \cite[(G3)]{TT} is sufficient to assume in place of the validity of the general inner variation identity \eqref{eq innervariation}. {We note that the lower frequency bound is satisfied if, for instance, solutions are assumed to be $\alpha$-H\"older regular for any $\alpha \in (0,1)$ - see Section \ref{subsec:discussion} for a further discussion on the frequency lower bound.}
        
        On the other hand, equipped with the Lipschitz continuity of our solutions $u$, together with a discrete spectral gap between the two lowest values (1 and $\tfrac{3}{2}$) of the frequency function, we in turn obtain an analogous structure for the free boundary $\{u=1\}$ to that for the segregated system considered in \cite{TT}. Therein, the authors also use the frequency function to distinguish between the regular and singular parts of the free boundary, and characterize the regular part as the points where the solution blows up to a linear function on either side of the free boundary. 
	\end{remark}

{
\begin{remark}[Further analysis of singularities]
   Classifying the types of free boundary singularities, at least in low dimensions, is one example of a natural follow-up question. {The singularities correspond to radially homogeneous solutions of \eqref{eq innervariation}-\eqref{eq:differential inequality} with $\Phi=0$. By rewriting in terms of $v=1-u$ and restricting to the unit sphere, these solutions may be identified with critical points of the optimal spectral partitioning problem on the sphere considered in \cite{bogosel2016method, helffer2009spectral}. In light of the asymptotic convergence of the rescaled Allen-Cahn problems \eqref{eq:MNR23a problem} to the Plateau problem as $\e\to 0$ and $v\to 0$, one may investigate the relationship between the types of singularities in each problem.} The limiting Plateau problem singularity models are conical $n$-dimensional area-minimizing (in the sense of Almgren) sets in $\R^{n+1}$. Since these have been classified as only $Y$-singularities when $n=1$ and only $Y$- and $T$-singularities when $n=2$ (see \cite{Taylor}) {and these cones coincide with $\{u=1\}$ for suitable homogeneous solutions of \eqref{eq innervariation}-\eqref{eq:differential inequality}}, this suggests that the other conical singularities that we find for general critical points in these dimensions {should not be present for minimizers of \eqref{eq:MNR23a problem} if $\e$ is small enough with respect to $\wire$ and $v$}. 

   Furthermore, building on the classification of singularities for general critical points when $n=1$ and corresponding local structure of the free boundary (see Lemma \ref{lemma:planar tangent classification} and Theorem \ref{thm:main regularity theorem}.(ii)), it would also be of interest to classify singularity models for critical points when $n=2$ and analyze the local structure of the free boundary near singularities there. {We refer the reader to the recent related work \cite{OgnVel24a}, where the authors obtain a structural result close to singularities of frequency 3/2.} The arguments in Section \ref{s:fb-structure} imply that after the two lowest frequencies $1$ and $3/2$ (corresponding to regular and $Y$-points), there is a non-explicit gap between $3/2$ and the next lowest frequency, and that if $\{u=1\}\cap \mathbb{S}^n$ is smooth and $n\geq 2$, then the frequency is at least $2$; cf.~ Proposition \ref{p: freq gap}.
\end{remark}
}
	
	\begin{remark}[Extension of existence from \cite{maggi2023hierarchy}]\label{r:improved-existence} In \cite[Theorem 1.2.(i)]{maggi2023hierarchy}, the existence of minimizers for 
		\begin{equation}\label{eq:ac problem in existence remark}
			\inf\Big\{\eps\int_{\Omega} |\nabla u|^2 +\frac{F(u)}{\eps^2}\, dx \,:  \int_{\Omega}V(u)=v,\,\mbox{$\{u = \delta\}$ is $\mathcal{C}$-spanning in the sense of \eqref{eq:c spanning for functions}} \Big\}\,,
		\end{equation}
		is established in a regime for $\e$, $v$, and $1/2<\delta\leq 1$ that ensures the proximity (in the sense of minimizers converging to minimizers) of \eqref{eq:ac problem in existence remark} to either the Plateau problem or its ``positive volume" extension
		\begin{equation}\notag
			\inf \big\{ \mathrm{Per}\,(E;\Om) : E\subset \Om, \,|E|=v ,\, \mbox{$E^\one \cup \pa^* E$ is $\mathcal{C}$-spanning}\big\}\,,  
		\end{equation}
		{where $\pa^* E$ denotes the reduced boundary of $E$ and $\mathrm{Per}\,(E;\Om) = \Hcal^n(\partial^* E \cap \Omega)$ is the relative perimeter of $E$ in $\Omega$.} The proof there relies on the asymptotic connection between \eqref{eq:ac problem in existence remark} and various sharp interface problems. Theorem \ref{thm:main existence theorem} strengthens this result significantly in the case $\delta=1$ by removing any restrictions on the parameters $\e$ and $v$: by setting $F=W/\e^2$ and replacing $V$ with $V/v$, we have the existence of continuous minimizers for any values of $\e$ and $v$ in \eqref{eq:ac problem in existence remark} as long as $\delta=1$. 
	\end{remark}
	
	\begin{remark}[Extension to bounded domains]\label{remark:bdd domains}
		Although we do not consider this problem here, one might also formulate \eqref{new model intro}-\eqref{new model volume constrained} on e.g.~ a bounded open set $\Omega$ with a corresponding spanning class $\mathcal{C}$ of smooth loops contained in $\Om$. The arguments in Theorem \ref{thm:main regularity theorem} are local in nature and would thus apply equally well in this scenario as well.
	\end{remark}
	
	\subsection{Notation}
	We will use $C$ to denote constants dependent only on the dimension $n+1$ of the ambient Euclidean space and on the fixed wire frame $\mathbf{W}$ in \eqref{new model intro} or \eqref{new model volume constrained}  throughout. If a constant $C$ has any additional dependencies on quantities $a,b,\dots$, we will use the notation $C(a,b,\dots)$. We denote open balls of radius $r$ centered at $x$ in $\R^{n+1}$ by $B_r(x)$. If $x=0$, we will omit the dependency on the center. $\Hcal^n$ denotes the $n$-dimensional Hausdorff measure (often used on $n$-dimensional spheres embedded in $\R^{n+1}$), while $\Lcal^{n+1}$ denotes the $(n+1)$-dimensional Lebesgue measure on $\R^{n+1}$. {For $\alpha \in (0,1)$, $[f]_{C^\alpha(U)}$ denotes the H\"{o}lder seminorm $\sup\big\{\frac{|f(x)-f(y)|}{|x-y|^\alpha} : x,y\in U, \ x\neq y\big\}$ of a function $f$ on a domain $U$, and $[f]_{\Lip(U)}$ denotes the Lipschitz seminorm $\sup\big\{\frac{|f(x)-f(y)|}{|x-y|} : x,y\in U, \ x\neq y\big\}$.}
	
	\subsection{Acknowledgements}
	The authors would like to thank Bozhidar Velichkov for pointing out numerous important references in the optimal partitions literature, to which our variational problem closely relates, Xavier Ros-Oton for many fruitful discussions that allow us to simplify some of our arguments, and Dennis Kriventsov for pointing out \cite{koch2015partial}.

	\section{Monotonicity properties and regularity of stationary solutions}\label{s:freq-monotonicity}
	
	The main result in this section, Theorem \ref{thm:holder reg for crit points}, establishes Lipschitz regularity for functions satisfying {a lower frequency bound} and the criticality conditions
    		\begin{align}\label{eq innervariation}
			&0=\int_\Omega\Big(|\nabla u|^2 + \Phi(u) \Big) \dive T - 2 \langle \nabla u, \nabla u \nabla T \rangle \,dx && \mbox{for all $T\in C_c^\infty(\Omega, \R^{n+1})$}\,,\\ \label{EL outer} 
			&2\,\,\int_{\Omega}|\nabla u|^2\,\varphi\,dx  = 	\int_{\Omega} (1-u)\,\Big\{ 2 \nabla u\cdot \nabla \varphi +\Phi'(u)\,\varphi\Big\}\,dx &&\mbox{for all $\varphi \in C_c^\infty(\Omega)$, and}\\ \label{eq:differential inequality}
			& 0\leq  \int_\Om 2\nabla u \cdot \nabla \varphi + \Phi'(u)\, \varphi \,dx&& \mbox{for all $\varphi \in C_c^1(\Om;[0,\infty))$,}
		\end{align}
    which we refer to as stationary solutions. 
    
    \begin{remark}[Spanning and the Euler-Lagrange equations]\label{remark:criticality remark}
		The first equation \eqref{eq innervariation} is the inner variational equation, the equation \eqref{EL outer} is the weak formulation of
		\begin{equation}\label{eq:weak 1}
			(1-u)\,\big\{2\Delta u-\Phi'(u)\big\}=0\qquad\mbox{in $\Omega$}\,,
		\end{equation} 
		and \eqref{eq:differential inequality} is the weak form of the differential inequality
		\begin{equation}\label{eq:weak 2}
			2\Delta u \leq \Phi'(u)\qquad\mbox{in $\Om$}\,.
		\end{equation}
		All of the equations above are quite natural in light of the spanning condition \eqref{eq:HP spanning}. The inner variational equation utilizes the simple fact that precomposing $u$ with a domain diffeomorphism preserves the spanning constraint; \eqref{EL outer} is a rigorous version of the intuition that $u$ should satisfy the usual volume-constrained Allen-Cahn equation on $\{u<1\}$, with possible singularities concentrated on $\{u=1\}$; and the differential inequality \eqref{eq:differential inequality} is the manifestation of the fact outer perturbations $u+\varphi$ by {\it negative} test functions $\varphi$ may disturb the spanning constraint, so that $u+\varphi$ is not an admissible variation.
	\end{remark}
    
    {In Section \ref{s:prelim}, we will show that minimizers to \eqref{new model intro}-\eqref{new model volume constrained} satisfy such conditions.} In Section \ref{subsec:statement of 3.1}, we state Theorem \ref{thm:holder reg for crit points} and show how it implies Theorem \ref{thm:main regularity theorem}.(i). The remaining subsections constitute the proof of Theorem \ref{thm:holder reg for crit points}.
	
	\subsection{Statement of Theorem \ref{thm:holder reg for crit points} and application to proof of Theorem \ref{thm:main regularity theorem}(i)}\label{subsec:statement of 3.1}
	Given  
	\begin{equation}\label{eq:g assumptions}
		\mbox{$G \in C^{2}([0,1])$ such that $G(0)=G'(0)=G'(1)=0$}, 
	\end{equation}
	we consider functions $v\in W^{1,2}_\loc(\Om;[0,1])$ satisfying
	\begin{align}\label{eq modified iv}  
		&\int_\Omega\Big(|\nabla v|^2 + G(v) \Big) \dive T - 2 \langle \nabla v, \nabla v \nabla T \rangle \,dx =0 &&\mbox{ for all $T\in C_c^\infty(\Omega; \R^{n+1})$,}\\
		\label{eq modified out} 
		&2\,\,\int_{\Omega}|\nabla v|^2\,\varphi \,dx = 	-\int_{\Omega} v\,\Big\{ 2\nabla v\cdot \nabla \varphi +G'(v)\,\varphi\Big\} dx&&\mbox{ for all $\varphi \in C_c^\infty(\Omega)$, and} \\ \label{eq modified ineq} 
		&\int_\Om \varphi \, d\mu =\int_\Om -2\nabla v \cdot \nabla \varphi - G'(v)\,\varphi\,dx&& \mbox{ for all $\varphi \in C_c^\infty(\Om)$,}
	\end{align}
	where $\mu$ is a non-negative Radon measure on $\Om$ depending only on $v$. In particular, we have the weak differential inequality
	\begin{equation}
		2\Delta v \geq G'(v) \quad \text{in $\Omega$}\,.
	\end{equation}
	We will verify that these variational identities are satisfied by $v=1-u$ when $u$ minimizes \eqref{new model intro} or \eqref{new model volume constrained}; see Corollary \ref{c:EL-v} below. In particular, they have natural interpretations in the context of \eqref{new model intro}-\eqref{new model volume constrained} (see Remark \ref{remark:criticality remark}).
	
	Observe that the assumptions \eqref{eq:g assumptions} together with the mean value theorem imply that
	\begin{equation}\label{eq:g estimates}
		|G(t)| \leq k t^2, \qquad |G'(t)| \leq kt\,,
	\end{equation}
	where $k = \sup_{[0,1]} |G''|$.
    
	The main result of this section establishes the optimal Lipschitz regularity  of $v$ for solutions of \eqref{eq modified iv}-\eqref{eq modified ineq} {satisfying an additional condition \eqref{eq:lower freq bound assumption theorem 2.1} on its zero set $\{v^*=0\}$ (see \eqref{eq: precise rep intro} for a definition of $v^*$)}. To state condition, we recall Almgren's frequency function
		\begin{equation}\notag
			N_{v,x_0}(r) = \frac{r\int_{B_r(x_0)}|\nabla v|^2\,dx }{\int_{\partial B_r(x_0)}v^2\,d\mathcal{H}^n(x)} = \frac{rD_{x_0,r}(r)}{H_{x_0,r}(r)}\qquad x_0 \in \Om\,, \,\, r < \dist(x_0,\partial \Om) \,.
	\end{equation}
    In Lemma \ref{lemma almgren} we will show that $N_{v,x_0}$ is monotone in $r$ for solutions of \eqref{eq modified iv}-\eqref{eq modified out}, so that the quantity
    \begin{equation}\notag
      N_{v,x_0}(0^+):= \lim_{r\to 0}N_{v,x_0}(r)  
    \end{equation}
    is well-defined {provided that $H_{v,x_0}(r) > 0$ for all $r>0$.}

\begin{theorem}[Lipschitz regularity for critical points {with frequency lower bound}]\label{thm:holder reg for crit points}
	Suppose $\wire = \rn \setminus \Om$ is closed, $G\in C^2([0,1])$ satisfies \eqref{eq:g assumptions}, and $v\in W^{1,2}_\loc(\Om;[0,1])$ satisfies \eqref{eq modified iv}-\eqref{eq modified ineq}, {and for any connected $\Omega' \cc \Omega$ on which $v$ is not uniformly $0$,
    \begin{equation}\label{eq:lower freq bound assumption theorem 2.1}
        \inf_{x_0\in \Om',\,x_0\in \partial \{v^*=0\} } N_{v,x_0}(r)>0\,.
    \end{equation}
    }
	
	\smallskip
	\noindent{\bf (i)} Then there is $r_{**}=r_{**}(G,n)>0$ with the following property: if
	$\Omega'\cc \Omega$ with $d=\dist(\wire,\partial \Om')>0$, then there are $M=M(\Om',v)>0$ such that 
	\begin{equation}\label{eq:freq bound for thm 3.1}
		N_{v,x}(r) \leq M \qquad \forall x \in \Om \mbox{ with $\dist(x,\Om') \leq d/2$},\quad  \,\, r < \min\{d/2, r_{**}\}\,,
	\end{equation}
	and $C=C(M,n)$ such that for any $x_0 \in \Omega'$ and $r< \min\{r_{**},d/3\}$,
	\begin{equation}\label{eq uniform lip 2}
		r[v]_{\Lip(B_{r/2}(x_0))}\leq C \Big(\frac{1}{r^{n-1}}\int_{B_{r(x_0)}}|\nabla v|^2\Big)^\frac{1}{2}\,.
	\end{equation}
	Also, $\mathcal{L}^{n+1}(\tilde{\Om} \cap \{v=0\})=0$ for any connected component $\tilde{\Om}\subset\Om$ on which $v$ is not identically zero.
	
	\smallskip
	
	\noindent{\bf (ii)} If in addition $\wire$ is compact, $\nabla v \in L^2(\Om)$, and $\mathcal{L}^{n+1}(\{v<t\})<\infty$ for all $t\in (0,1)$, then given $d>0$, there is $M(v,d)$ such that $N_{v,x}(r)\leq M$ for all $x\in\Om$ with $\dist(x,\partial \Om)\geq d$ and $r<\min\{d,r_{**}\}$. 
\end{theorem}
{\begin{remark}
    The hypothesis ``$\mathcal{L}^{n+1}(\{v<t\})<\infty$ for all $t\in (0,1)$" in part \textbf{(ii)} of Theorem \ref{thm:holder reg for crit points} is motivated by the properties of minimizers $u=1-v$ of \eqref{new model intro}. When $n\geq 2$, such minimizers satisfy $u \in L^{2(n+1)/(n-1)}(\Omega)$, where $\frac{2(n+1)}{n-1}$ is the Sobolev dual exponent of $2$, while for $n=1$, we no longer have the desired Sobolev embedding so instead we directly verify these measure bounds on the superlevel sets of $u$ (which correspond to sublevel sets of $v$). See the proof of Theorem \ref{thm:main existence theorem} in Section 5 for more details.
\end{remark}}

\subsection{Monotonicity formulae} Our first result is a semilinear version of the classical monotonicity formula for harmonic maps (see, e.g., \cite[Proposition 3.3.6]{lin2008analysis}).

\begin{lemma}[Almost-monotonicity of normalized Dirichlet energy]\label{lemma mon normalized Dir}
	Let $v\in W^{1,2}_\loc(\Om;[0,1])$ satisfy the inner variation equation \eqref{eq modified iv}. Then, given any $x_0 \in \Omega$ and $r_0=\dist(x_0, {\Wbf})$, we have that 
	\begin{equation}\label{eq almost monotonicity 0}
		\frac{d}{ds}\Big(	\frac{1}{s^{n-1}}\int_{ B_s(x_0)} |\nabla v|^2\Big) = \frac{2}{s^{n-1}} \int_{\partial B_{s}(x_0)}  |\nabla v \cdot \hat{x}|^2\, d\Hcal^n(x)	+M(s)\qquad \mbox{for a.e.~ $s \in (0,r_0)$},
	\end{equation}		
	where $\hat x = (x-x_0)/|x-x_0|$ and
	\begin{equation}\label{eq extra term Dir derivative}
		M(s)= \frac{n+1}{ s^n}\int_{B_{s}(x_0)}G(v) - \frac{1}{ s^{n-1}} \int_{\partial B_{s}(x_0)}G(v)\, d\mathcal{H}^{n}.
	\end{equation}
\end{lemma}
\begin{proof}
	Without loss of generality, let us assume $x_0=0$,  and fix $s < r_0$. Let us consider the vector field $T= \eta(|x|) x$, where $\eta \in C_c^\infty(B_{s})$  is a radially symmetric non-negative function. Then, 
	$$\nabla T = \eta(|x|) \Id + |x|\eta'(|x|) \hat{x}\otimes \hat{x}.$$	
	Therefore, 
	$$\langle \nabla v, \nabla v \nabla T\rangle=\eta(|x|)|\nabla v|^2 + \eta'(|x|) |x||\nabla v\cdot \hat{x}|^2 ,$$ 
	and
	$$\text{div}\,T= (n+1)\eta(|x|)+|x|\eta'(|x|).$$ 
	Plugging in the previous expressions into the inner variation \eqref{eq modified iv}, we obtain
	\begin{eqnarray}\label{eq powerfield 0}
		(n-1)\int_{B_{s}} \eta(|x|) |\nabla v|^2\, dx  &=&   \int_{B_{s}}  \eta'(|x|)|x|(2|\nabla v \cdot \hat{x}|^2-|\nabla v|^2)\, dx\\\notag
		&& -\int_{B_{s}}G(v)\big[(n+1)\eta(|x|)+|x|\eta'(|x|)\big]\, dx
	\end{eqnarray}
	Given $s \in (0,r_0)$, after an appropriate regularization procedure, we may consider the following family of Lipschitz test functions:
	\begin{equation*}
		\eta_k(t)=
		\begin{cases}
			1, \quad t\in [0,s-1/k],\\
			k(s-t), \quad t\in [s-1/k,s],
		\end{cases}
	\end{equation*}
	where $k >\frac{1}{s}$. Let us notice that since $\eta_k \to 1_{[0,s]}$ as $k\to \infty$ and $ s\mapsto \int_{ B_s} |\nabla v|^2 $ is absolutely continuous in $(0, r_0)$, we can take $k\to \infty$ in \eqref{eq powerfield 0} for almost every $s\in (0,r_0)$ to deduce 
	\begin{equation}\label{eq derivative dirichlet 2}
		(n-1)\int_{B_{s}}  |\nabla v|^2 =  s \int_{\partial B_{s}}  \big(|\nabla v|^2 -2|\nabla v \cdot \hat{x}|^2+ G(v)\big)\, d\mathcal{H}^{n-1}(x) -(n+1)\int_{B_{s}}G(v)\quad \mbox{for a.e. $s$}.
	\end{equation}
	Dividing by $s^n$, we can rewrite \eqref{eq derivative dirichlet 2} as
	\begin{equation*}
		\frac{d}{ds}\Big(	\frac{1}{s^{n-1}}\int_{ B_s} |\nabla v|^2\Big) = \frac{2}{s^{n-1}} \int_{\partial B_{s}}  |\nabla v \cdot \hat{x}|^2\, d\Hcal^n(x)	+\frac{n+1}{ s^n}\int_{B_{s}}G(v) - \frac{1}{s^{n-1}} \int_{\partial B_s} G(v)\, d\Hcal^n\, .
	\end{equation*}
	This is precisely the claimed identity \eqref{eq almost monotonicity 0}.	
\end{proof}

The next result shows the almost-subharmonicity of $v^2$, which is a slight variation of \cite[Theorem 1.3-Step 3]{maggi2023plateau}. {First, however, we recall the notion of precise representative, which allows us to make sense of pointwise values of $W^{1,2}_\loc$ functions up to $\Hcal^n$-a.e.}

If $u\in W^{1,2}(\Om)$, then ${\h}$-a.e.~ $x\in \Om$ is a Lebesgue point of $u$, and the {\bf precise representative} $u^*$ is given by
\begin{equation}\label{eq: precise rep intro}
u^*(x) =  \begin{cases}
\displaystyle\lim_{r\to 0}\displaystyle\frac{1}{\omega_{n+1}r^{n+1}}\int_{\{|z-x|<r\}} u(z)\,d\mathcal{L}^{n+1}(z) & \mbox{if the limit exists} \\ 
0  & \mbox{otherwise} \,,
\end{cases}
\end{equation}
with the above limit existing for ${\h}$-a.e.~ $x\in \Om$ (see e.g.~ \cite[Chapter 4.8]{EG}). 

\begin{lemma}[Almost-subharmonicity of $v^2$]\label{lemma alsmot subharmonicity} 
	Let $v\in W^{1,2}_\loc(\Om;[0,1])$ satisfy the outer variation equation \eqref{eq modified out} with $G$ satisfying \eqref{eq:g assumptions}. If $x_0\in\Om$ and $r_0=\dist(x_0,\pa\Om)$, then, for $k = \sup_{[0,1]}|G''|$, the function
	\begin{equation}\label{eq:almon of averages}
		g(r)= e^{\frac{k r^2}{4} }\,\fint_{B_r(x_0)} v^2\,dx
	\end{equation}
	is increasing on $(0,r_0)$. Furthermore, if $v^*$ is the precise representative of $v$, 
	\begin{equation}\label{eq subharmonic bound}
		\int_{B_r(x_0)} v^2 \,dx\leq \frac{2r}{n+1}\int_{\partial B_r(x_0)} (v^*)^2d\mathcal{H}^{n}
	\end{equation}
	for every $r\in (0,\min\{r_*, r_0\})$ with $r_* = \sqrt{(n+1)/k}$.
\end{lemma}

\begin{remark}\label{remark:weak uc}
	Note that an immediate consequence of \eqref{eq subharmonic bound} is the following statement: if $v\in W^{1,2}(\Om;[0,1])$ satisfies the outer variation equation \eqref{eq modified out} with $G$ satisfying \eqref{eq:g assumptions} and $\int_{\partial B_r(x)}v^2\,d\mathcal{H}^n=0$ for some $x\in \Om$ and $r<\min\{\dist(x,\partial \Om),r_*\}$, then $v\equiv 0$ on $B_r(x)$.
\end{remark}
\begin{proof}
	Assume without loss of generality $x_0=0$ and let $r\in (0,r_0)$. Testing \eqref{eq modified out} with ${\{\varphi_j\}_j}\subset C_c^1(\Omega;[0,\infty))$ such that $\varphi_j(x)\to \vphi(x):=[(r^2-|x|^2)/2]_+$ uniformly and $\nabla \varphi_j \to \nabla \varphi = - x {\mathbf{1}_{B_r}}$ in $L^2$, we obtain
	\begin{equation}\label{eq modified out v2} 
		\int_{B_r}  2\, v\, (\nabla v\cdot x)\, dx \geq \int_{B_r} G'(v)v\varphi\,dx\,.
	\end{equation}
	On the other hand, recalling from \eqref{eq:g estimates} that $|G'(v)v|\leq k v^2$ and using that $0 \leq \varphi \leq r^2/2$, the estimate \eqref{eq modified out v2} in turn yields
	\begin{equation}\label{eq lower boun der}
		\int_{B_r}  2\, v\, (\nabla v\cdot x) \, dx \geq -\frac{k r^2}{2}\int_{B_r} v^2\,dx.
	\end{equation}
	So, introducing the notation 
	$$\phi(r) := \fint_{B_r} v^2\,dx ,$$
	and using \eqref{eq lower boun der} leads us to the lower bound
	\begin{equation*}
		\phi'(r) = \frac{1}{r}\fint_{B_r}  2\, v\, (\nabla v\cdot x) \, dx \geq - \frac{k r}{2} \fint_{B_r} v^2\,dx = - \frac{k r}{2} \phi(r)\, ,
	\end{equation*}
	for a.e.~ $r\in (0,r_0)$. By combining this estimate with the absolute continuity of $g$, we deduce its monotonicity, {thus establishing \eqref{eq:almon of averages}.}
    
    Lastly, towards \eqref{eq subharmonic bound}, we differentiate $g$ to get
	\begin{equation*}
		0\leq   \frac{1}{r^{n+1}} \int_{\partial B_r} v^2\,d\mathcal{H}^{n} -  \frac{n+1}{r^{n+2}} \int_{B_r} v^2\,dx +  \frac{k r}{2 r^{n+1}} \int_{B_r} v^2\,dx\quad \mbox{for a.e.~ $r<r_0$},
	\end{equation*}
	which in turns implies that
	\begin{equation}\label{eq partial subharmonic bound}
		\int_{B_r} v^2 \,dx\leq \frac{2r}{n+1}\int_{\partial B_r(x_0)} (v^*)^2d\mathcal{H}^{n}\quad\mbox{for almost every $r<\min\{r_0,\sqrt{(n+1)/k}\}$}\,.
	\end{equation}
	To prove \eqref{eq subharmonic bound} for \emph{all} small $r$, we choose a more careful Lebesgue representative of $v$, since the measure zero set for which \eqref{eq partial subharmonic bound} fails depends on the ``choice" of $v$. We thus consider the precise representative
	\begin{equation}\notag
		v^*(x) = \lim_{s \to 0} \fint_{B_s(x)}v(y)\,dy\,.
	\end{equation}
    Note that $v^*(ty)$ is absolutely continuous as a function of $t\in (0,r_0)$ for $\mathcal{H}^n$-a.e.~ $y\in \mathbb{S}^n$ \cite[Section 4.9.2]{EG}. So, by fixing small $r$, letting $t_j\to r$ be a sequence of radii for which \eqref{eq partial subharmonic bound} holds for $t_j$ and $v^*$, and applying the dominated convergence theorem {in $\partial B_1$ to $\{v^*(t_j \cdot)\}_j$, bearing in mind that $v^*(t_jy) \to v^*(ry)$} for $\mathcal{H}^n$-.a.e $y\in \partial B_1$, we find that
	\begin{equation}\notag
		\int_{B_r}v^2\,dx = \lim_{j\to \infty} \int_{B_{t_j}}v^2\,dx \leq \limsup_{j\to \infty}\frac{2t_j}{n+1} \int_{\partial B_{t_j}}{(v^*)^2}\,d\mathcal{H}^n = \frac{2r}{n+1}\int_{\partial B_r}(v^*)^2\,{d\Hcal^n}\,.
	\end{equation}
	Thus \eqref{eq subharmonic bound} holds as claimed for all $r<\min\{r_0,r_*\}$.
\end{proof}

We also have the almost-subharmonicity of $v$ by \cite[Proof of Theorem 1.3, Step 3]{maggi2023hierarchy}, which is written for minimizers $u$ of \eqref{new model volume constrained} but in fact only relies on the Euler-Lagrange equations \eqref{eq innervariation}-\eqref{eq:differential inequality}, and may be rewritten in terms of $v=1-u$. The proof follows analogous reasoning to that of Lemma \ref{lemma alsmot subharmonicity} but since it is short, only exploits the 
{validity of \eqref{eq modified ineq}} and allows us to define the precise representative of $v$ as the limit of integral averages at \emph{all} points in $\Omega$, we include it here for the convenience of the reader.

\begin{lemma}[Almost-subharmonicity of $v$]\label{l:almost-subharm-v}
	Let $v\in W^{1,2}_\loc(\Om;[0,1])$ satisfy 
    {\eqref{eq modified ineq}} with $G$ satisfying \eqref{eq:g assumptions}. If $x_0\in \Om$ and $r_0 = \dist(x_0,{\Wbf})$, then for $k = \sup_{[0,1]}|G''|$ the function
	\begin{equation}\label{eq:almost mon of u averages}
		r\mapsto e^{\frac{kr^2}{8}}\fint_{B_r(x_0)}v\,dx
	\end{equation}
	is increasing on $(0,r_0)$. As a consequence, with $v^*$ denoting the precise representative (see \eqref{eq: precise rep intro}), we have
	\begin{equation}\label{eq:v convention}
		v^*(x_0) = \lim_{r\to 0}\fint_{B_r(x_0)} v\,dx \qquad \mbox{for every } x_0\in \Om\,,
	\end{equation}
    {and $v^*$ is upper-semicontinuous on $\Omega$.}
\end{lemma}

\begin{proof}
	We assume $x_0=0$ again, and once again recall the estimate \eqref{eq:g estimates} for $G$. By the same regularization procedure as in the proof of Lemma \ref{lemma alsmot subharmonicity}, we now test \eqref{eq modified ineq} with $\varphi:=[(r^2-|x|^2)/2]_+$ and estimate
	\begin{equation}
		\int_{B_r} {2}\, \nabla v \cdot x \, dx \geq  \int_{B_r} G'(v)\, \varphi \,dx \geq -\frac{kr^2}{2}\int_{B_r}v\,.\label{e:alm-subharm-v-ineq}
	\end{equation}
	Since the function
	\[
	\psi(r) := \fint_{B_r} v\,dx
	\]
	satisfies
	\begin{equation}\notag
		\psi'(r) = \frac{1}{r}\fint_{B_r} \nabla v \cdot x\, dx
	\end{equation}
	for a.e.~ $r\in (0,\dist(0,\pa \Om))$, the estimate \eqref{e:alm-subharm-v-ineq} implies that for a.e.~ $r\in (0,\dist(0,\pa \Om)\})$, we have
	\begin{equation}\notag
		\psi'(r) \geq -\frac{kr}{4} \fint_{B_r} v\,dx = -\frac{kr}{4} \psi(r)\,.
	\end{equation}
	From this inequality we easily conclude \eqref{eq:almost mon of u averages}, and \eqref{eq:v convention} in turn follows immediately, since we have in particular just demonstrated the limit therein exists for all points $x_0$ in $\Omega$. The upper-semicontinuity of $v^*$ is a standard consequence of the monotonicity \eqref{eq:almost mon of u averages}.
\end{proof}

\begin{remark}[Identification of $v$ with its precise representative]\label{remark:precise rep}
	Given a function $v\in W^{1,2}_\loc(\Omega;[0,1])$ satisfying \eqref{eq modified out}, {Lemma \ref{l:almost-subharm-v}} allows us to make a canonical choice for $v(x)$ via \eqref{eq:v convention}, by identifying $v$ with its precise representative $v^*$. {\bf For the rest of the paper, we identify $\boldsymbol{v}$ with $\boldsymbol{v^*}$.}
\end{remark}

We address now the key monotonicity property satisfied by $v$. Given $x_0 \in \Omega$, and $r \in (0, \dist(x_0, {\Wbf}))$ we recall Almgren's frequency function
\begin{equation}\label{eq freq}
	N_{v,x_0}(r) = \frac{{r}\int_{B_r(x_0)}|\nabla v|^2 }{  \int_{\partial B_r(x_0)}v^2 {d\mathcal{H}^{n}}} = \frac{D_{v,x_0}(r)}{H_{v,x_0}(r)}\, ,
\end{equation}
where
\begin{align}
	D_{v,x_0}(r) &= \frac{1}{r^{n-1}} \int_{B_r(x_0)}|\nabla v|^2\, , \notag \\
	H_{v,x_0}(r) &= \frac{1}{r^{n}}\int_{\partial B_r(x_0)}v^2 {d\mathcal{H}^{n}}\, , \label{e:D-H-def}
\end{align}
so that $N_{v,x_0}(r)$ is well defined when $H_{v,x_0}(r)> 0$. When it is clear from context, we will omit the dependency on $v$ and/or $x_0$ for the frequency function. The next lemma shows the almost-monotonicity of \eqref{eq freq} for $v$.

\begin{lemma}[Almost-monotonicity of the frequency]\label{lemma almgren}
	Let $v\in W^{1,2}_\loc(\Om;[0,1])$ satisfy both \eqref{eq modified iv} and \eqref{eq modified out} with $G$ satisfying \eqref{eq:g assumptions}. Then, there exists $\kappa=\kappa(\sup_{[0,1]}|G''|,n) \geq 0$ such that for any $x_0 \in \Omega$ {with $H_{v,x_0}(r)>0$ for some $r<\min\{r_0,r_*\}$}, the function
	\begin{equation}\label{e:almost-monotonicity}
		r\to e^{\frac{\kappa r^2}{2}}(N_{v,x_0}(r)+1)
	\end{equation}
	is well-defined {(absolutely continuous)} and non-decreasing on $(\inf\{r:H_{v,x_0}(r)>0\}, \min\{r_0, r_*,1\})$ with $r_0$, $r_*$ as in Lemma \ref{lemma alsmot subharmonicity}. Moreover, if $G=0$ then $N_{v,x_0}$ is increasing on $(\inf\{r:H_{v,x_0}(r)>0\},r_0)$ and is constant if and only if $v$ is homogeneous of degree $N_{v,x_0}(0^+)$. 
\end{lemma}
\begin{remark}\label{r:usc}
    The conclusion of Lemma \ref{lemma almgren}, together with Remark \ref{remark:weak uc} in particular allows one to make sense of the limit $N_{v,x_0}(0^+) :=\lim_{r\to 0^+} N_{v,x_0}(r)$, provided that $H_{v,x_0}(r) > 0$ for all $r>0$ sufficiently small. {In addition, the function $x\mapsto N_{v,x}(0+)$ is upper-semicontinuous.}
\end{remark}
\begin{proof}
	We assume, without loss of generality that $x_0=0$ and omit dependency of $N$, $D$ and $H$ on $x_0$. If $H(r)>0$ for some $r<\min\{r_0,r_*\}$, then by Remark \ref{remark:weak uc}, $H(s)>0$ for all $r<s<\min\{r_0,r_*\}$. Thus $\{r<\min\{r_0,r_*\}:H(r)=0\}$ coincides with the interval $(0,\inf\{r:H(r)>0\}]$, so we may as well restrict ourselves to the interval $(\inf\{r:H(r)>0\},\min\{r_0,r_*\})$ where $H>0$. {Clearly on this interval $N$ is absolutely continuous, since both $H$ and $D$ are.}
	Since 
	\begin{equation*}\label{eq derivative}
		N'(r) =  \frac{D'(r)H(r)-H'(r)D(r)}{H(r)^2}
	\end{equation*}
	the monotonicity of \eqref{e:almost-monotonicity} is equivalent to the bound $\partial_r[\log (N(r)+1)] \geq -\kappa r$, which may in turn be rewritten as
	\begin{equation}\label{eq lower bound}
		D'(r)H(r)-H'(r)D(r) \geq - \kappa r \Big(H(r)^2+D(r)H(r)\Big).
	\end{equation}
	Having \eqref{eq lower bound} in mind as our target, we compute each one of the terms, starting with $D'(r)$ which, thanks to \eqref{eq almost monotonicity 0}, has the form
	\begin{equation}\label{eq der dir}
		D'(r) = \frac{2}{r^{n-1}} \int_{\partial B_r}|\nabla v \cdot \hat{x}|^2d\mathcal{H}^n+ M(r)\,.
	\end{equation}
	On the other hand, testing \eqref{eq modified out} with $\varphi \to \mathbf{1}_{B_r}$ we deduce
	\begin{equation}\label{eq outer boundary} 
		\int_{B_r}\, \big(2|\nabla v|^2\, +G'(v)v\big) \,dx = 2\int_{\partial B_r} (\nabla v \cdot \hat{x}) v\, d\mathcal{H}^n.
	\end{equation}
	Thus, differentiating $H$ and using \eqref{eq outer boundary} yields
	\begin{eqnarray}\notag
		H'(r) &=& \frac{2}{r^{n}}  \int_{\partial B_r}(\nabla v \cdot \hat{x}) v \, d\mathcal{H}^n\\\label{eq der height}
		&=& 	\frac{1}{r^{n}}\int_{B_r}\, (2|\nabla v|^2\, +G'(v)v)\,dx \, ,
	\end{eqnarray}
	or equivalently,
	\begin{equation}\label{eq equivalent Dirichlet} 
		D(r) =  \frac{1}{r^{n-1}}  \int_{\partial B_r}(\nabla v \cdot \hat{x}) v\, d\mathcal{H}^n -I(r),
	\end{equation}
	where
	\begin{equation}\label{eq remainder I}
		I(r)= \frac{1}{2{r^{n-1}}}\int_{B_r}\, G'(v)v\,dx.
	\end{equation}
	Altogether, \eqref{eq der dir}, \eqref{eq der height}, and \eqref{eq equivalent Dirichlet} yield the estimate
	\begin{eqnarray}\notag
		D'(r)H(r)-H'(r)D(r)&=& \frac{2}{r^{n-1}}H(r)\,\int_{\partial B_r}|\nabla v \cdot \hat{x}|^2d\mathcal{H}^n   - {\frac{2}{r^n}}D(r) \int_{\partial B_r}(\nabla v \cdot \hat{x}) v\, d\mathcal{H}^n \\
		&&+M(r)H(r)\\\notag
		&=&\frac{2}{r^{2n-1}}\Bigg(\int_{\partial B_r}v^2 d\mathcal{H}^{n-1}\int_{\partial B_r}|\nabla v \cdot \hat{x}|^2d\mathcal{H}^n-   \Big(\int_{\partial B_r}(\nabla v \cdot \hat{x}) v\, d\mathcal{H}^n\Big)^2\Bigg)\\\notag
		&&+ H(r)M(r) +\frac{2 I(r)}{r^{n}} \int_{\partial B_r}(\nabla v \cdot \hat{x}) v\, d\mathcal{H}^n\\\label{eq cs lowerbound}
		&\geq& H(r)M(r) +\frac{2 I(r)}{r^{n}} \int_{\partial B_r}(\nabla v \cdot \hat{x}) v\, d\mathcal{H}^n
	\end{eqnarray}
	where in the last line we have used Cauchy-Schwarz.
	
	Our goal now is to bound the last line in \eqref{eq cs lowerbound} for $r<\min\{r_*,r_0,1\}$. With this idea in mind, let us notice that the bound $|G(t)| \leq k t^2$ (from \eqref{eq:g estimates}) and \eqref{eq subharmonic bound} (which applies since $r<\min\{r_*,r_0\}$) together imply that
	\begin{eqnarray}\label{eq bound solid integrals}
		|M(r)|&\leq& \frac{n+1}{ r^n}\int_{B_{r}}|G(v)| \,dx+ \frac{1}{ r^{n-1}} \int_{\partial B_{r}}|G(v)|\,{d\mathcal{H}^{n}} \\ \notag
		&\leq& \frac{k(n+1)}{ r^n}\int_{B_{r}}v^2 \,dx+ \frac{k}{ r^{n-1}} \int_{\partial B_{r}}v^2 \,{d\mathcal{H}^{n}}\\ \notag
		&\leq& \frac{k(n+1)}{r^n}\int_{\partial B_r}\frac{2r}{n+1}v^2\,d\mathcal{H}^n + \frac{k}{ r^{n-1}} \int_{\partial B_{r}}v^2 \,{d\mathcal{H}^{n}} = 3kr H(r),
	\end{eqnarray}
	where $k=\sup_{[0,1]}|G''|$, and similarly 
	\begin{eqnarray}\label{eq bound solid integrals 2}
		|I(r)|\leq\frac{1}{2r^{n-1}}\int_{B_r}kv^2 \,dx \leq  \frac{1}{2r^{n-1}} \cdot \frac{2r}{n+1}\int_{\partial B_r}kv^2\,d\mathcal{H}^n = \frac{kr^2}{n+1}H(r).
	\end{eqnarray}
	Additionally, from \eqref{eq equivalent Dirichlet} and \eqref{eq bound solid integrals 2}, we deduce that
	\begin{equation}\label{eq boundary gradient}
		\frac{1}{r^{n}}  \int_{\partial B_r}|(\nabla v \cdot \hat{x}) v|\, d\mathcal{H}^n \leq \frac{kr}{n+1}H(r)+ \frac{D(r)}{r}.
	\end{equation}
	Thus, by combining \eqref{eq cs lowerbound}, \eqref{eq bound solid integrals}, \eqref{eq bound solid integrals 2}, and  \eqref{eq boundary gradient} we deduce
	\begin{equation*}
		D'(r)H(r)-H'(r)D(r) \geq -3krH(r)^2- \frac{kr^2}{n+1}H(r) \Big[\frac{2kr}{n+1}H(r)+ \frac{D(r)}{r} \Big]
	\end{equation*}
	which, since $r<1$, yields \eqref{eq lower bound} for suitable $\kappa$ depending on $k$ and $n$.
	
	Let us finish by observing that in the absence of potential $G$, the classical frequency monotonicity formula holds, which amounts to the inequality
    \[
        D'(r)H(r) - H'(r)D(r) \geq 0\,.
    \]
    Furthermore, one has the usual characterization of the case when $r\mapsto N(r)$ is constant by analyzing the case when this is an equality. {See, for instance, \cite[Lemma 4.1]{fernandez2022thin}. Note that the energy-minimizing property is in fact not required therein for the monotonicity, only the inner and outer variations.} 
\end{proof}

Later, we will need information on functions satisfying a lower frequency bound. Towards this end, we give the following almost-monotonicity result for the normalized $L^2$ spherical averages of functions satisfying both the inner and outer variation equations (see, for instance, \cite[Lemma 4.2]{fernandez2022thin} or \cite[Corollary 3.18]{DLS_Q-valued}).

\begin{corollary}\label{corollary mon averages}
	Let $v\in W^{1,2}_\loc(\Om;[0,1])$ satisfy both \eqref{eq modified iv} and \eqref{eq modified out} with $G$ satisfying \eqref{eq:g assumptions}. Then given $\alpha>0$, there exists a constant {$\kappa_1>0$ depending on $n$, $\sup_{[0,1]}|G''|$, and $\alpha$} such that the following holds: if $x_0\in \Omega$ and $N_{v,x_0}(0^+)\geq \alpha$, then the function
	\begin{equation*}
		\phi(r)= \frac{e^{\kappa_1 r}}{r^{2 \alpha+n}}\int_{\partial B_r(x_0)} v^2 \, d\mathcal{H}^n
	\end{equation*}
	is non-decreasing for $r\in {(0,\min\{\dist(x_0, {\Wbf}), r_*, 1\})}$. 
\end{corollary}
\begin{proof}
	As usual, we omit dependency on $v$ and $x_0$ for $H$ and $N$. Also, recalling Remark \ref{remark:weak uc} and the triviality of the claim if $H\equiv 0$, we may as well assume we are working on an interval of scales $r$ where $H(r)>0$. We will choose $\kappa_1$ at the end. Noticing that $\phi(r) = \frac{e^{\kappa_1 r} H(r)}{r^{2\alpha}}$, we can compute directly its derivative using \eqref{eq der height}, yielding
	\begin{eqnarray}\notag
		\phi'(r) &=& \frac{\phi(r)}{r}\Big(\frac{r H'(r)}{H(r)}-2\alpha +\kappa_1 r\Big),\\\label{eq derivative phi}
		&=& \frac{\phi(r)}{r}\Big(2N(r)+\frac{1}{r^{n-1}H(r)}\int_{ B_r} G'(v)v\,dx -2\alpha +\kappa_1 r\Big).
	\end{eqnarray}
	On the other hand, thanks to Lemma \ref{lemma alsmot subharmonicity} and \eqref{eq:g estimates}, we deduce that
	\begin{equation}\label{eq bound average}
		\frac{1}{r^{n-1}H(r)}\int_{ B_r} |G'(v)v|\,dx \leq \frac{2k}{n+1}r^2,
	\end{equation}
	for $k=\sup_{[0,1]}|G''|$. Additionally, from Lemma \ref{lemma almgren} and our assumption that $ \lim_{r\to 0^+} N_{v,x_0}(r)\geq \alpha$, we have $N(r)\geq e^{-\frac{\kappa r^2}{2}}(\alpha+1)-1$. Combining this with \eqref{eq derivative phi} and \eqref{eq bound average} yields 
	\begin{eqnarray}\notag
		\phi'(r) &\geq& \frac{\phi(r)}{r}\Big( 2(e^{-\frac{\kappa r^2}{2}} - 1)(\alpha+1) - \frac{2k}{n+1}r^2 +\kappa_1 r\Big)\\\notag
		&\geq& \frac{\phi(r)}{r}\Big(-C(\alpha+1)r^2 - {\frac{2k}{n+1}r^2} +\kappa_1 r\Big)\\\label{eq lowerbound der phi}
		&=& \phi(r)\Big( -C(\alpha+1)r -{\frac{2k}{n+1}r} +\kappa_1\Big)\, ,
	\end{eqnarray}
	for $C$ depending on $\kappa=\kappa(k,n)$ from Lemma \ref{lemma almgren}. Finally, since $r \leq 1$, we can take $\kappa_1$ large enough (with the claimed dependencies) in \eqref{eq lowerbound der phi} to make the right-hand side positive and thus conclude the proof.
\end{proof}

\subsection{A criterion for H\"{o}lder regularity}

Here we show that locally uniform lower and upper bounds on the frequency function yield a locally uniform H\"{o}lder bound.

\begin{lemma}[Local H\"{o}lder regularity from frequency bounds]\label{lemma holder}
	Let $\alpha \in (0,1]$ and $M>0$. Suppose that $G$ satisfies \eqref{eq:g assumptions}, $v\in W^{1,2}_\loc(\Om;[0,1])$ satisfies \eqref{eq modified iv}-\eqref{eq modified out}, $\{v > 0\}$ is relatively open in $\Omega$, and $2\Delta v = G'(v)$ in the classical sense in $\{{v>0}\}$. Then there exists $r_{**}(\alpha, n,\sup_{[0,1]}|G''|)\in (0,r_*]$ with the following property:
	
	\smallskip
	
	if in a given subdomain $\Omega_2 \cc\Omega$ we have uniform frequency bounds, i.e.,
	\begin{align}\label{eq lower bound 2}
		\alpha&\leq  N_{v,x_0}(0^+) \qquad \mbox{for each $x_0\in {\Omega_2 \cap \partial}\{{v}=0\}$\,,}\quad\mbox{and} \\ \label{eq upper bound assump in holder lemma}
		N_{v,x_0}(r) &\leq M \qquad\qquad  \mbox{for each $x_0\in \Omega_2 \cap \overline{\{v>0\}}$ and $0<r<\min\{r_{**},\dist(\partial \Om_2,\partial \Om)/2\}$}
	\end{align}
	then there is $C=C(\alpha,M,n)$ such that for any $\Omega_1 \cc \Omega_2$, $x_0\in \overline{\Omega_1}$ and $0<2r\leq r_0\leq \min\{\dist(\partial \Om_1, \partial \Omega_2)/3, r_{**}\}$, we have that
	\begin{align}\label{eq interior bound}
		r^\alpha[{v}]_{C^{\alpha}(B_r(x_0))}&\leq C \Big(\frac{1}{r_0^{n-1}}\int_{B_{r_0}}|\nabla v|^2\Big)^\frac{1}{2}\qquad \mbox{if $\alpha<1$}\qquad \mbox{and} \\
		\label{e:Lipschitz}
		r\|v\|_{\Lip(B_r(x_0))}&\leq C \Big(\frac{1}{r_0^{n-1}}\int_{B_{r_0}}|\nabla v|^2\Big)^\frac{1}{2}\qquad \mbox{if $\alpha=1$}.
	\end{align}
\end{lemma}
\begin{proof}
	The estimates \eqref{eq interior bound}-\eqref{e:Lipschitz} would follow from obtaining $C(\alpha,M,n)>0$ such that
	\begin{equation}\label{eq gradient estimate}
		\int_{B_r(x_0)} |\nabla v|^2 \leq C  \Big( \frac{r}{r_0}\Big)^{2\alpha+n-1} \int_{B_{r_0}(x_0)} |\nabla v|^2,
	\end{equation}
	for all $x_0\in \overline{\Om_1}$ and $0 < r \leq r_0\leq \min\{\dist(\partial \Om_1, \partial \Omega_2)/3, r_{**}\}$. Indeed, from this, a direct application of Campanato's criterion (see e.g. \cite[Theorem 6.1]{maggi2012sets}) yields \eqref{eq interior bound}. {For reasons which will be evident from the arguments, we choose $r_{**}< \min\{1,r_*\}$ (cf.~ Lemma \ref{lemma alsmot subharmonicity}) small enough so that, again setting $k=\sup_{[0,1]}|G''|$ and recalling $\kappa$ from Lemma \ref{lemma almgren} and $\kappa_1$ from Lemma \ref{corollary mon averages}, we have
		\begin{equation}\label{eq:choice of r**}
			\alpha/2 \leq e^{-\kappa r^2/2}(\alpha + 1)-1\,,\quad e^{kr^2/8}\leq 2 \,,\quad\mbox{and}\quad e^{\kappa_1r}\leq 2 \qquad \mbox{for all $r\leq r_{**}$}\,.
		\end{equation}
		Since $r_*$ and $\kappa$ depend on $n$ and $k$ and $\kappa_1$ depends on $n$, $k$, and $\alpha$, we observe that $r_{**}$ depends on $n$, $k$, and $\alpha$.}

	{Before proving \eqref{eq gradient estimate}, we introduce the notations
		$$\varepsilon := \dist(\partial\Omega_1,\partial\Omega_2)\quad\mbox{and}\quad \Omega_2^t:=\{x\in \Omega_2 : \dist(x,\partial \Om_2)\geq t \}$$
		and make a preliminary observation.} 
	We claim that if $x\in \{v>0\}$, then
	\begin{equation}\label{eq case 2 harmonic}
		\int_{B_t(x)} |\nabla v|^2 \leq 2\Big(\frac{t}{s}\Big)^{n+1} \int_{B_s(x)} |\nabla v|^2\qquad\forall \, 0<s<t \leq \min\{\dist(x,\{v=0\}),r_{**}\}\,.
	\end{equation}
	To prove \eqref{eq case 2 harmonic}, we can use the equation $2\Delta v= G'(v)$ combined with Bochner's formula to deduce that in $\{{v}>0\}$,
	\begin{equation*}
		\frac{1}{2}\Delta |\nabla v|^2 = |D^2 v|^2 + \frac{1}{2}G''(v)|\nabla v|^2 \geq - \frac{k}{2} |\nabla v|^2\,.
	\end{equation*}
	Thus $|\nabla v|^2$ is almost subharmonic, and so thanks to Lemma \ref{l:almost-subharm-v} (note that the latter merely relies on an estimate of the above form), we see that
	\begin{equation*}
		t\to e^{{\frac{kt^2}{8}}}  \fint_{B_t(x)} |\nabla v|^2
	\end{equation*}
	is non-decreasing on $(0,\dist(x,\{v=0\}))$. Consequently, for $t<s<r_{**}$, \eqref{eq:choice of r**} implies \eqref{eq case 2 harmonic}.

	{The proof of \eqref{eq gradient estimate} at $x_0\in \overline{\Om}_1$ is split into five cases depending on $d = \dist(x_0, \partial\{v=0\})$, $r$, and $r_0$: i) $d=0$, ii) $r\geq r_0/10$, iii) $0<d\leq r<r_0/10$,  iv) $0<r<d<r_0/10$ and v) $0<r<r_0/10\leq d$.}
	
	\medskip
	
	\noindent{\it To prove \eqref{eq gradient estimate} if $d=0$}: Since it will be useful later, we prove \eqref{eq gradient estimate} for any $x_0\in \Omega_2^{2\e/3}$ such that $d=0$. Note that $v(x_0)=0$ since $\{v>0\}$ is relatively open. Then by \eqref{eq:choice of r**}, Lemma \ref{lemma almgren}, the lower frequency bound \eqref{eq lower bound 2}, and the upper frequency bound \eqref{eq upper bound assump in holder lemma} we deduce that
	\begin{equation}\label{eq bounds freq}
		\alpha/2 \leq  e^{-\kappa r^2/2}(\alpha + 1)-1 \leq N_{v,x_0}(r) \leq M\quad \forall\, 0 < r \leq r_{**}\,.
	\end{equation}
	By combining \eqref{eq bounds freq} with Corollary \ref{corollary mon averages}
    , we may estimate
	\begin{eqnarray}\notag
		\int_{B_r(x_0)} |\nabla v|^2\,dx &\leq&   \frac{M}{r} \int_{\partial B_{r}(x_0)} v^2 \, d\mathcal{H}^n\\ \notag
		&\leq& \frac{M}{r}  \Big(\frac{r}{r_0}\Big)^{2\alpha+n} e^{\kappa_1 r_0}\int_{\partial B_{r_0}(x_0)} v^2 \, d\mathcal{H}^n\\\label{eq estimate fb}
		&\leq& M\Big( \frac{r}{r_0}\Big)^{2\alpha+n-1}\frac{2e^{\kappa_1 r_0}}{\alpha} \int_{B_{r_0}(x_0)} |\nabla v|^2.
	\end{eqnarray}
	Since $e^{\kappa_1r_0}\leq e^{\kappa_1r_{**}}$ and $\kappa_1$ from Corollary \ref{corollary mon averages} depends on $n$, $k$, and $\alpha$, after decreasing $r_{**}$ depending on $\kappa_1$ if needed, we have proved \eqref{eq gradient estimate} with $C(\alpha,M) = \frac{2M}{\alpha}$.
	
	\medskip
	
	\noindent{\it To prove \eqref{eq gradient estimate} if $r\geq r_0/10$}: Since $10r/r_0\geq 1 $, we have
	\begin{equation}\notag
		\int_{B_r(x_0)}|\nabla v|^2 \leq \int_{B_{r_0}(x_0)}|\nabla v|^2 \leq 10^{2\alpha + n-1}\Big(\frac{r}{r_0} \Big)^{2\alpha+n-1}\int_{B_{r_0}(x_0)}|\nabla v|^2\,,
	\end{equation}
	which is \eqref{eq gradient estimate} with constant $C(\alpha,n) = 10^{2\alpha+n-1}$.
	
	\medskip
	
	\noindent{\it To prove \eqref{eq gradient estimate} if $0<d\leq r<r_0/10$}. First, since $0<d\leq r<\e/3$ and $x_0\in \overline{\Om_1}$ (so $\dist(x,\partial \Om_2)\geq \e$), we may choose $y\in\{v=0\}\cap \Om_2^{2\e/3}$ such that $d=\dist(x_0,\{v=0\})=|x_0-y|\leq r$. Since also $r<r_0/10$, we thus have the inclusions
	\begin{equation}\notag
		B_r(x_0) \subset B_{2r}(y) \subset B_{r_0}(y)\,.
	\end{equation}
	By these inclusions and \eqref{eq estimate fb} applied at $y\in \{v=0\} \cap \Om_2^{2\e/3}$,
	\begin{equation}\notag
		\int_{B_r(x_0)}|\nabla v|^2 \leq \int_{B_{2r}(y)}|\nabla v|^2 \leq  \frac{2M}{\alpha}\Big(\frac{2r}{r_0}\Big)^{2\alpha+n-1} \int_{B_{r_0}(y)}|\nabla v|^2\,,
	\end{equation}
	which is \eqref{eq gradient estimate} with constant $C(\alpha,M,n)=\frac{2^{2\alpha+n} M}{\alpha}$.
	
	\medskip
	
	\noindent{\it To prove \eqref{eq gradient estimate} if $0<r<d<r_0/10$}: First, note that since $0<r<d$, then either $B_r(x_0)\subset \{v=0\}$ or $B_r(x_0) \subset \{v>0\}$. If $B_r(x_0) \subset \{v=0\}$, the the estimate \eqref{eq gradient estimate} is trivial, so we may as well assume that $B_r(x_0)\subset \{v>0\}$. Similar to the previous case, since $0<d\leq r_0<\e/3$ and $x_0\in \overline{\Om_1}$, we may choose $y\in \{v=0\}\cap \Om_2^{2\e/3}$ such that $|x_0-y|=d$. 
    Now, observe that $\dist(y,\partial B_{r_0}(x_0)) = r_0 - d > \frac{9r_0}{10}$, so $B_{9r_0/10}(y)\subset B_{r_0}(x_0)$. Combining this elementary observation with the estimate \eqref{eq estimate fb} already established for $y\in \{v=0\}\cap \Omega_2^{2\e/3}$, we find that
	\begin{eqnarray}\notag
		\int_{B_d(x_0)} |\nabla v|^2 &\leq& \int_{B_{2d}(y)} |\nabla v|^2 \\\notag
		&\leq& \frac{2M}{\alpha}\Big( \frac{18d}{10r_0}\Big)^{2\alpha +n-1} \int_{B_{\frac{9r_0}{10}}(y)} |\nabla v|^2 \\ \label{eq tedious radii}
		&\leq& C(\alpha,M,n) \Big( \frac{2d}{r_0 }\Big)^{2\alpha +n-1} \int_{B_{r_0} (x_0)} |\nabla v|^2.
	\end{eqnarray}
	Hence, by combining \eqref{eq case 2 harmonic} at scales $r<d$ and \eqref{eq tedious radii}, we deduce that for $r<d$,
	\begin{align*}
		\int_{B_r(x_0)} |\nabla v|^2 &\leq 2\Big(\frac{r}{d}\Big)^{n+1}  C(\alpha,M,n)\Big( \frac{2d}{r_0}\Big)^{2\alpha +n-1} \int_{B_{r_0}(x_0)} |\nabla v|^2 \\
		&={2^{2\alpha+n}C(\alpha,M,n)\Big(\frac{d}{r}\Big)^{2\alpha - 2}  \Big( \frac{r}{r_0}\Big)^{2\alpha +n-1}} \int_{B_{r_0}(x_0)} |\nabla v|^2 \,.
	\end{align*}
	Since $\alpha\in (0,1]$ and $r<d$ together imply that $(d/r)^{2\alpha-2}\leq 1$, this yields \eqref{eq gradient estimate} when $0<r<d<r_0/10$ with constant $2^{2\alpha+n}C(\alpha,M,n)$. 
	
	\medskip
	
	\noindent{\it To prove \eqref{eq gradient estimate} if $0<r<r_0/10\leq d$}: Again, we have either $B_r(x_0) \subset \{v=0\}$ or $B_r(x_0)\subset \{v>0\}$, and \eqref{eq gradient estimate} is trivial in the former case. So we take $B_r(x_0)\subset \{v>0\}$. By \eqref{eq case 2 harmonic} applied at scales $r$ and $r_0/10$ (which applies since $r_0/10\leq \min\{d,r_{**}\}$), we have
	\begin{equation}\notag
		\int_{B_r(x_0)}|\nabla v|^2 \leq 2\Big( \frac{10r}{r_0}\Big)^{n+1} \int_{B_{r_0/10}(x_0)}|\nabla v|^2 \leq 2\Big( \frac{10r}{r_0}\Big)^{2\alpha + n -1} \int_{B_{r_0}(x_0)}|\nabla v|^2\,,
	\end{equation}
	where we have used $10r/r_0< 1$ and $\alpha\in (0,1]$. This is precisely \eqref{eq gradient estimate} with constant $2\cdot 10^{2\alpha + n-1}$.
\end{proof}

\subsection{Frequency bounds and H\"{o}lder regularity}
In this subsection we establish local upper bounds for the frequency function for solutions of \eqref{eq modified iv}-\eqref{eq modified out}, then use them together with Lemma \ref{lemma holder} to establish local H\"{o}lder regularity {under the assumption of a uniform lower frequency bound, which we will later verify in Section \ref{sec freqbound}}.

\begin{lemma}[{Lebesgue points and gradient decay}]\label{lemma vanishing frequency}
	Let $v\in W^{1,2}_\loc(\Om;[0,1])$ satisfy \eqref{eq modified out} with $G$ satisfying \eqref{eq:g assumptions}. Then every $x_0 \in \Omega$ is a Lebesgue point of $v$ and
	\begin{equation}\label{eq vanishing dirichlet}
		\lim_{r\to 0^+} \frac{1}{r^{n-1}}\int_{B_r(x_0)}|\nabla v|^2 =0\,.
	\end{equation}
	
	
\end{lemma}
\begin{proof}
	First of all, recall that we are identifying $v$ with its precise representative $v^*$ (which requires only the equation \eqref{eq modified out}, cf.~ Remark \ref{remark:precise rep}) so that \eqref{eq:v convention} holds. Before proving the lemma, we make a preliminary computation testing \eqref{eq modified out} with the mollified fundamental solution. Assuming for notational convenience that $0\in \Omega$, let us consider the functions $\Gamma_t = \eta_t \star \Gamma$ and  $\Gamma_{t}^\sigma = \eta_t\star \Gamma^\sigma$, where $\{\eta_t\}_t\subset C_c^\infty(B_1)$ are an approximation to the identity {(namely, $\int \eta_t = 1$, $\eta_t \to \delta_0$ in distribution as $t\to 0$, where $\delta_0$ is the Dirac delta mass at 0)}, $t,\sigma>0$ and 
	\[
	\Gamma(x) = \begin{cases}
		|x|^{1-n} & \text{when $n\geq2$} \\
		-\ln(|x|) & \text{when $n=1$}
	\end{cases},\qquad \qquad \Gamma^\sigma(x) = \Gamma(x/\sigma)\,.
	\]
	Fix $\psi\in C_c^\infty(B_1;[0,1])$ with $\psi=1$ on $B_{1/2}$, so that for each $r>0$, $\psi_r:=\psi(\cdot /r) \in C_c^\infty(B_{r};[0,1])$, $\psi_r=1$ on $B_{r/2}$, and $r|\nabla \psi_r|+r^2|\Delta \psi_r|\leq C$.
	By testing \eqref{eq modified out} with $\Gamma_{t}^\sigma \psi_r\in C_c^\infty(B_r)$ and then integrating by parts, we obtain
	\begin{eqnarray}\label{eq est 1 grad}
		2\int_{B_{r}}\Gamma_{t}^\sigma \psi_r |\nabla v|^2 &=& - \int_{B_{r}} \left(\nabla v^2\cdot \nabla (\Gamma_{t}^\sigma \psi_r) +G'(v)v \psi_r \Gamma^\sigma_{t}\right)\\ \notag
		&=& \int_{B_r}v^2\, \Delta (\Gamma_t^\sigma \psi_r)-G'(v)v  \psi_r \Gamma_{t}^\sigma \,.
	\end{eqnarray}
	Since $\nabla \psi_r =0$ in $B_{r/2}$, we have that
	\begin{eqnarray}\label{eq intbyparts fundsol}
		\int_{B_{r}} v^2\, \Delta ( \Gamma_{t}^\sigma \psi_r)  &=& \int_{B_{r}} v^2(\Delta \Gamma_{t}^\sigma)\psi_r + 2 \int_{B_{r}\setminus B_{r/2}} v^2 \nabla \Gamma_{t}^\sigma \cdot \nabla\psi_r + \int_{B_{r}\setminus B_{r/2}} v^2 \Gamma_{t}^\sigma \Delta \psi_r.
	\end{eqnarray}
	In summary, we have found
	\begin{eqnarray}\notag
		2\int_{B_{r}}\Gamma_{t}^\sigma \psi_r |\nabla v|^2 &=& \int_{B_{r}} v^2(\Delta \Gamma_{t}^\sigma)\psi_r\,  + 2 \int_{B_{r}\setminus B_{r/2}} v^2 \nabla \Gamma_{t}^\sigma \cdot \nabla\psi_r + \int_{B_{r}\setminus B_{r/2}} v^2 \Gamma_{t}^\sigma \Delta \psi_r\\\label{eq est 3 grad}
		&&- \int_{B_{r}} G'(v)v  \psi_r \Gamma_{t}^\sigma\,;
	\end{eqnarray}
	observe that the same equality for the translation $v(\cdot + x)$ replacing $v$ holds whenever $B_r(x)\subset \Om$.
	
	\medskip
	
	\noindent{\it To prove \eqref{eq vanishing dirichlet}}: Let us assume without loss of generality that $x_0=0$. Since \eqref{eq vanishing dirichlet} is trivial if $n=1$ (by the continuity of the Lebesgue integral), for this step we assume $n\geq 2$. We notice now that $\Delta \Gamma_t = \bar c\eta_t$ for some dimensional constant $\bar c>0$ (which only depends on $n$, not $\mathbf{W}$), since $\Delta \Gamma = \bar c\delta_0$, where $\delta_0$ is the Dirac mass at $0$. Altogether, bearing in mind that $\eta_t$ is an approximation to the identity and that $0\leq v, \psi_r \leq 1$, we have that 
	\begin{equation}\label{eq aux bound1}
		\limsup_{t \to 0} \Big|\int_{B_{r}} v^2(\Delta \Gamma_t) \psi_r \Big| \leq \bar c 	\limsup_{t \to 0} \int_{B_{r}} \eta_t  \leq \bar c.
	\end{equation}
	On top of this, since $n\geq 2$, the estimates for $\nabla \psi_r$ and $\Delta \psi_r$ yield
	\begin{equation}\label{eq aux bound2}
		\lim_{t \to 0} \int_{B_{r}\setminus B_{r/2}} \Big(|\Delta \psi_r||\Gamma_t(x)|+|\nabla \psi_r| |\nabla \Gamma_t(x)|\Big) v^2\leq \frac{C}{r^{n+1}}\int_{B_{r}\setminus B_{r/2}} v^2\,,
	\end{equation}
	after updating the constant $C$. Thus, by combining \eqref{eq est 3 grad}-\eqref{eq aux bound2}, Fatou's lemma, and the estimate \eqref{eq:g estimates} for $|G'|$, we deduce
	\begin{equation}\label{eq provisional bound}
		\int_{B_{r/2}} \Gamma |\nabla v|^2 \leq C\Bigg( 1+\frac{1}{r^{n+1}}\int_{B_{r}\setminus B_{r/2}}  v^2+\int_{B_{r}} \Gamma v^2\Bigg).
	\end{equation}
	In particular, again exploiting the fact that $|v|\leq 1$, we have shown that $\Gamma |\nabla v|^2$ is locally integrable.

	From this integrability we may conclude \eqref{eq vanishing dirichlet}. Indeed,  \eqref{eq provisional bound} and the dominated convergence theorem applied to the 1-parameter family of functions $f_s = \mathbf{1}_{B_s} \Gamma |\nabla v|^2$ give
	\begin{equation*}
		0 = \lim_{s\to 0^+} \int_{B_1} f_s \geq \limsup_{s\to 0^+} \frac{1}{s^{n-1}}\int_{B_s}|\nabla v|^2 =0.
	\end{equation*}

\medskip

\noindent{\it To prove that $x_0$ is a Lebesgue point}: Again working at the origin for convenience, we set $v_r = \fint_{B_r} v$. By Poincar\'e's inequality, we have that
\begin{eqnarray}\notag
	\frac{1}{r^{n+1}}\int_{B_r}|v(x)-v(0)|^2 &\leq& \frac{C}{r^{n+1}}\int_{B_r}|v(x)-v_r|^2 +|v_r-v(0)|^2\,dx\\\label{eq lebesgue point}
	&\leq& \frac{C}{r^{n-1}}\int_{B_r}|\nabla v|^2\,dx +C|v_r-v(0)|^2 \to 0,
\end{eqnarray}
as $r\to 0^+$ in virtue of \eqref{eq vanishing dirichlet} and \eqref{eq:v convention}; recall that here $v(0)$ is defined via \eqref{eq:v convention}.\end{proof}

\begin{lemma}[Locally uniform frequency upper bound and consequences]\label{lemma: uc take 2}
Let $v\in W^{1,2}_\loc(\Om;[0,1])$ satisfy both \eqref{eq modified iv} and \eqref{eq modified out}, with $G$ satisfying \eqref{eq:g assumptions}. Then given any $\Omega'\cc \Omega$, there exist constants $M>0$, $K>0$ and $\zeta\in (0,1)$, all depending on $v$, $\Om'$, $\sup_{[0,1]}|G''|$, and $n$, such that for any $x_0\in \Omega' \cap \overline{\{v>0\}}$ and $r\leq\min\{\dist(\partial \Om',\partial \Om)/2,r_*,1\}$, we have
\begin{align}\label{eq upper bound frequency take 2}
	N_{v,x_0}(r) &\leq {M}
	\\
	\label{eq doubling shell take 2}
	\int_{\partial B_{r}(x_0)} v^2 \, d\Hcal^n &\leq K\int_{\partial B_{r/2}(x_0)} v^2\, d\Hcal^n \,, \\ \label{eq leb dens bound}
	\mathcal{L}^{n+1}(B_{r/2}(x_0) \cap \{v=0\}) &\leq \zeta \omega_{n+1}(r/2)^{n+1}\,\quad\mbox{and} \\ \label{eq linf for rescaled guys}
	\| v \|_{L^\infty(B_{r/2}(x_0))} &\leq K H_{v,x_0}(r/2)^{1/2}\,.
\end{align}
\end{lemma}
\begin{proof}
We set $\tilde{r}=\min\{\dist(\partial \Om',\partial \Om)/2,r_*,1\}$.

\medskip

\noindent{\it To prove \eqref{eq upper bound frequency take 2}}: The function $x\mapsto H_{v,x}(\tilde{r})$ is continuous in $x$ {(by the compactness of the trace operator in $W^{1,2}_{\text{loc}}$)} so it achieves a minimum on the compact set $\Om \cap \overline{\{v>0\}}$ at some $y_0$. By Remark \ref{remark:weak uc}, if it were the case that $H_{v,y_0}(\tilde{r})=0$, we would have $v\equiv 0$ on $B_{\tilde{r}}(y_0)$, contradicting the fact that $y_0\in \overline{\{v>0\}}$ (recall that we are taking $v=v^*$). It then follows by the frequency almost-monotonicity in Lemma \ref{lemma almgren} that
\begin{align}\notag
	\sup_{B_{r}(x)\,:\,x\in \overline{\{v>0\}} \cap \Om', r\leq \tilde{r} } N_{v,x}(r) &\leq \sup_{x\in \overline{\{v>0\}}\cap \Om' } e^{\kappa \tilde{r}^2/2}N_{v,x}(\tilde{r}) +e^{\kappa\tilde{r}^2/2} {-1} \\
	&\leq e^{\kappa\tilde{r}^2/2} +\frac{e^{\kappa\tilde{r}^2/2}}{\tilde{r}^{n-1}H_{v,y_0}(\tilde{r})}\int_{\{\dist(x,\Om')<\tilde{r}\}}|\nabla v|^2\,dx=:M\,.\label{e:freq-upper-bd-M}
\end{align}

\medskip

\noindent{\it To prove \eqref{eq doubling shell take 2}}: Again using Remark \ref{remark:weak uc}, we have $H_{v,x_0}(r)>0$ for all $0<r<\tilde{r}$ and $x_0\in \Om' \cap \overline{\{v>0\}}$. Therefore, given $r \in (0, \tilde{r}]$, in virtue of \eqref{eq der height}, Lemma \ref{lemma alsmot subharmonicity} and \eqref{eq:g estimates}, we compute
\begin{equation}\label{eq derivative rep freq take 2}
	\frac{d}{dr}\ln(H_{v,x_0}(r)) = \frac{2}{r} N_{v,x_0}(r) + \frac{1}{r^n H_{v,x_0}(r)} \int_{B_r(x_0)} G'(v) v \leq \frac{2}{r} N_{v,x_0}(r) + Cr\,,
\end{equation}
where $C$ depends on $\sup_{[0,1]}|G''|$ and $n$.
So, using Lemma \ref{lemma almgren}, and integrating \eqref{eq derivative rep freq take 2} between $r/2$ and $r$, with $r\in (0,\tilde{r}]$, on both sides, we deduce
\begin{equation*}
	\ln\Big( \frac{H_{v,x_0}(r)}{H_{v,x_0}(r/2)}\Big) \leq  2\ln(2) e^{2\kappa \tilde{r}^2}(N_{v,x_0}(\tilde{r})+1) + \frac{C\tilde{r}^2}{2},
\end{equation*}
which in turns implies
\begin{equation} \label{eq doubling boundary take 2}
	\int_{\partial B_{r}(x_0)} v^2 \, d\mathcal{H}^{n}\leq  e^{{C \tilde{r}^2}/2}2^{2 \psi(x_0,\tilde{r})}\int_{\partial B_{r/2}(x_0)} v^2 \, d\mathcal{H}^{n},
\end{equation}
where $\psi(x_0,\tilde{r}) := e^{2\kappa \tilde{r}^2}(N_{v,x_0}(\tilde{r})+1)$. After exploiting the definition of $M$ given by \eqref{e:freq-upper-bd-M}, this is precisely the claimed doubling estimate \eqref{eq doubling shell take 2} on spherical shells, with $K = e^{C \tilde{r}^2/2} 2^{2M}$. 

\medskip

\noindent{\it To prove \eqref{eq leb dens bound}}: Fix $x_0\in \overline{\{v>0\}} \cap \Om'$. We first integrate \eqref{eq doubling shell take 2} with respect to the radius to deduce the doubling property
\begin{equation}\label{eq doubling solid1}
	\int_{B_{r}(x_0)} v^2\,dx \leq 2K\int_{B_{r/2}(x_0)} v^2\,dx \quad \forall\, 0<r\leq\tilde{r},
\end{equation}
on balls. On the other hand, let us notice that by the almost-subharmonicity in Lemma \ref{lemma alsmot subharmonicity}, 
\begin{equation}\label{eq subharmonicity}
	\Vert v\Vert_{L^\infty(B_{r/2}(x_0))}^2\leq \frac{C(n,\sup |G''|)}{r^{n+1}}\int_{B_{r}(x_0)} v^2\,dx \qquad \forall\, 0 < r \leq \tilde{r}.
\end{equation}
The estimates \eqref{eq doubling solid1} and \eqref{eq subharmonicity} together imply
\begin{equation}\label{eq doubling solid2}
	\Vert v\Vert_{L^\infty(B_{r/2}(x_0))}^2\leq \frac{2CK}{r^{n+1}}\int_{B_{r/2}(x_0)} v^2\,dx \qquad \forall\, 0 < r \leq \tilde{r}.
\end{equation}
{Using that $|v|\leq 1$ (and H\"older's inequality if $q\geq 2$),} we deduce the reverse H\"{o}lder type inequality
\begin{equation}\label{eq lp equivalence}
	\Big( \frac{1}{r^{n+1}}\int_{B_\frac{r}{2}(x_0)} v^p\,dx \Big)^\frac{1}{p}\leq 2CK  \Big( \frac{1}{r^{n+1}}\int_{B_{r/2}(x_0)} v^q\,dx \Big)^\frac{1}{q},
\end{equation}
for any $1\leq p, q\leq \infty$, where the constants are independent of $x_0 \in \overline{\{v>0\}} \cap \Om'$. {To deduce the Lebesgue density upper bound from \eqref{eq lp equivalence}, we first apply H\"older's inequality to estimate
	\begin{equation}\notag
		\fint_{B_{r/2}(x_0)}v\,dx \leq \bigg(\fint_{B_{r/2}(x_0)} v^2 \,dx\bigg)^{1/2} \bigg( \fint_{B_{r/2}(x_0)} \mathbf{1}_{\{v>0\}}\,dx\bigg)^{1/2} \qquad \forall \, 0 < r \leq \tilde{r}\,.
	\end{equation}
	After rearranging this inequality and applying \eqref{eq lp equivalence} with $p=2$ and $q=1$, we arrive at
	\begin{align*}\notag
		\bigg( \fint_{B_{r/2}(x_0)} \mathbf{1}_{\{v>0\}}\,dx\bigg)^{-1/2} \leq \bigg( \fint_{B_{r/2}(x_0)}v\,dx \bigg)^{-1} \bigg(\fint_{B_{r/2}(x_0)} v^2 \,dx\bigg)^{1/2} \leq 2CK \,,
	\end{align*}
	which implies \eqref{eq leb dens bound} with $\zeta = 1-(2CK)^{-2}$.
	
	\medskip
	
	\noindent{\it To prove \eqref{eq linf for rescaled guys}}: We use \eqref{eq doubling solid2} followed by \eqref{eq subharmonic bound} to estimate
	\begin{equation}\notag
		\|v\|_{L^\infty(B_{r/2}(x_0))}^2 \leq \frac{2CK}{r^{n+1}}\int_{B_{r/2}(x_0)}v^2\,dx \leq \frac{2CK}{r^{n+1}}\cdot\frac{r}{2} C\big(\sup_{[0,1]}|G''|,n\big)\int_{\partial B_{r/2}(x_0)} (v^*)^2\, d\mathcal{H}^{n}\,.
	\end{equation}
	By taking square root on both sides and renaming the constant, we obtain \eqref{eq linf for rescaled guys}.
}
\end{proof}



As an intermediate step towards Theorem \ref{thm:holder reg for crit points}, we prove interior H\"older regularity of solutions, {under the additional assumption of a uniform lower frequency bound}.

\begin{lemma}[Interior H\"{o}lder regularity {from lower frequency bound}]\label{thm:inter reg for crit points}
Let $\wire = \rn \setminus \Om$ be closed, $G\in C^2([0,1])$ satisfy \eqref{eq:g assumptions}, and $v\in W^{1,2}_\loc(\Om;[0,1])$ satisfy \eqref{eq modified iv}-\eqref{eq modified ineq} {and \eqref{eq:lower freq bound assumption theorem 2.1}}. If {$\Omega_1 \cc\Omega_2 \cc \Omega$, and $\alpha\in (0,1]$ and $M$ are the lower and upper frequency bounds on $\Omega_2$ from \eqref{eq upper bound frequency take 2} and \eqref{eq:lower freq bound assumption theorem 2.1} respectively, then there exists $C(\alpha,M,n)>0$ and $r_{**}=r_{**}(\alpha,n,\sup_{[0,1]}|G''|)$, such that for every $x_0 \in \Omega_1$ and $r<\min\{r_{**},\dist(\partial \Omega_1,\partial \Omega)/3\}$, we have}
\begin{equation}\label{eq uniform holder 2}
	r^\alpha[v]_{C^{\alpha}(B_{r/2}(x_0))}\leq C \Big(\frac{1}{r^{n-1}}\int_{B_{r}}|\nabla v|^2\Big)^\frac{1}{2}\,.
\end{equation}
\end{lemma}

\begin{proof}

Firstly, note that Lemma \ref{lemma vanishing frequency} guarantees that every point of $v$ is a Lebesgue point, so $v$ is defined pointwise as a limit of its integral averages. Our goal is to show that $v$ satisfies the hypotheses of Lemma \ref{lemma holder}. Applying Lemma \ref{lemma holder} will then yield the estimate \eqref{eq uniform holder 2}. Since the frequency bounds \eqref{eq lower bound 2}-\eqref{eq upper bound assump in holder lemma} hold due to \eqref{eq:lower freq bound assumption theorem 2.1} and \eqref{eq upper bound frequency take 2} as noted in the statement of the lemma, we must therefore demonstrate that $\{v>0\}$ is relatively open in $\Omega$, and that $v$ solves $2\Delta v= G'(v)$ in the classical sense in $\{v>0\}$. 

\medskip

\noindent{\it To verify that $\{v>0\}$ is relatively open}: It suffices to show that $\{v=0\}$ is relatively closed in $\Omega$. {First of all, note that we just have to show that any accumulation point $x'$ of $\partial\{v=0\}$ remains in $\{v=0\}$, since any accumulation point of interior points $\{v=0\}$ will either also be an interior point, or will be a boundary point, for which we know that \eqref{eq:lower freq bound assumption theorem 2.1} holds.} Now, thanks to Lemma \ref{lemma almgren} {(see Remark \ref{r:usc})}, the mapping $x\mapsto N_{x}(0^+)= \lim_{r\to 0^+} N_{x}(r)$ is upper-semicontinuous. Thus, any accumulation point $x'$ of $\partial\{v=0\}$ in the interior of $\Omega$ satisfies $N_{x'}(0^+)>0$, namely the lower bound in \eqref{eq:lower freq bound assumption theorem 2.1} applies at $x'$. On the other hand, \eqref{eq vanishing dirichlet}, {\eqref{eq subharmonic bound}, and \eqref{eq:v convention} together imply} that $N_{x}(0^+)=0$ for $x \notin \{v=0\}$, implying that we must have $v(x')=0$.

\medskip

\noindent{\it To verify $2\Delta v = G'(v)$ in the classical sense in $\{v>0\}$}: Let $\mu$ be the non-negative measure from \eqref{eq modified ineq}, so that in particular
\begin{equation}\label{eq:modified ineq in thm 3.1 proof}
	\int_{\Om} \varphi \, d\mu = \int_{\Om} -2 \nabla \varphi \cdot \nabla v - \varphi G'(v)\,dx \qquad \forall \varphi\in C_c^\infty(\{v>0\})\,.
\end{equation}
If $\mu=0$ on $\{v>0\}$, then $v$ would solve the equation $2\Delta v = G'(v)$ in the usual weak sense in the open set $\{v>0\}$, at which point the standard elliptic regularity theory shows that $v$ is a classical solution there. To prove that $\mu =0$ on $\{v>0\}$, we claim that it suffices to show that
\begin{equation}\label{eq:vmu vanishes}
	\int_{\{v>0\}} \varphi\, v\, d\mu = 0 \qquad \forall \varphi\in C_c^\infty(\{v>0\})\,,
\end{equation}
Indeed, \eqref{eq:vmu vanishes} implies that the non-negative Radon measure $v\mu\mres \{v>0\}$ is the zero measure, but since $v(x)>0$ for {\it every} $x\in \{v>0\}$, this forces $\mu=0$ there. So our task is reduced to proving \eqref{eq:vmu vanishes}. 

Given an arbitrary test function $\varphi \in C_c^\infty(\{v>0\})$, let us consider the mollifications $(\varphi v)_\e:=(\varphi v) * \eta_\e$ for a family $\{\eta_\e\}$ of smooth mollifiers. By the property $0\leq v \leq 1$ and the fact that every point of $v$ is a Lebesgue point (Lemma \ref{lemma vanishing frequency}.(i), we have
\begin{equation}\label{eq:DCT assumptions}
	0\leq (\varphi v)_\e\leq 1\quad \mbox{and}\quad \big((\varphi v)* \eta_\e\big)(x) \to (\varphi v)(x)\mbox{ for all $x\in \{v>0\}$}.
\end{equation}
Then, since $(\varphi v)_\e\in C_c^\infty(\{v>0\}$, we may test \eqref{eq:modified ineq in thm 3.1 proof} with $(\varphi v)_\e$ and apply the Dominated Convergence Theorem to compute
\begin{align}\notag
	\int_{\{v>0\}} \varphi v \,d\mu = \lim_{\e\to 0} \int_{\{v>0\}} (\varphi v)_\e \,d\mu &= \lim_{\e\to 0}\int_{\{v>0\}}-2 \nabla (\varphi v)_\e \cdot \nabla v - (\varphi v)_\e G'(v)\,dx \\ \label{eq:first vmu vanish}
	&= \int_{\{v>0\}}-2 \nabla (\varphi v) \cdot \nabla v - (\varphi v)G'(v)\,dx\,,
\end{align}
where in the last equality we have used the strong $W^{1,2}$-convergence of $(\varphi v)_\e$ to $\varphi v$. Now by the product rule for products of $C_c^\infty$ and $W^{1,2}$ functions and then \eqref{eq modified out}, the right hand side expands as
\begin{equation}\label{eq:second vmu vanish}
	\int_{\{v>0\}}-2 \nabla (\varphi v) \cdot \nabla v - (\varphi v)G'(v)\,dx =\int_{\{v>0\}}-2 \varphi |\nabla v|^2 - 2 v \nabla \varphi \cdot \nabla v - (\varphi v)G'(v)\,dx =0\,.
\end{equation}
Putting \eqref{eq:first vmu vanish}-\eqref{eq:second vmu vanish} together yields \eqref{eq:vmu vanishes}, as desired. 
\end{proof}

\subsection{Compactness}

In this subsection we show that solutions of \eqref{eq modified iv}-\eqref{eq modified ineq} enjoy strong compactness in $W^{1,2}$.

\begin{lemma}[Compactness for solutions of \eqref{eq modified iv}-\eqref{eq modified ineq}]\label{lemma compactness cp}
Let $B_{3r_0}(x_0)\subset \mathbb{R}^{n+1}$, $v_k\in (W^{1,2}_\loc\cap {C^0})(B_{3r_0};[0,{\infty}))$ satisfy \eqref{eq modified iv}-\eqref{eq modified ineq} for some $G_k\in C^1(v_k(B_{3r_0}))$ and non-negative Radon measures $\mu_k$. If, furthermore, there exists a function $v\in (W^{1,2})\cap C^0)(\overline{B_{2r_0}(x_0)})$ such that
\begin{align}\label{eq uniform energy bound}
	&v_k \weak v\quad \mbox{weakly in $W^{1,2}(B_{2r_0}(x_0))$}\,,\\ \label{eq uniform linf}
	&\|v-v_k\|_{L^\infty(B_{2r_0}(x_0))} \to 0\qquad \mbox{and} \\ \label{eq gvk convergence}
	&\|G_k'(v_k) - G_0'(v)\|_{L^2(B_{2r_0}(x_0))}\to 0
\end{align}
for some $G_0\in C^1(v(B_{2r_0}(x_0))$, then there exists a non-negative Radon measure $\bar{\mu}$ in $B_{2r_0}$ such that, up to extracting a subsequence,
\begin{align}\label{eq weak convergence measures0}
	&\mu_{k}\overset{\ast}{\rightharpoonup} \bar \mu\quad \mbox{as measures in $B_{2r_0}(x_0)$}, \\ \label{eq convergence v}
	& v_{k}\to  \bar v\quad \mbox{strongly in $W^{1,2}(B_{r_0}(x_0))$},\quad\mbox{and} \\ \label{eq vbar solves equations} 
	&\mbox{$\bar v$ satisfies \eqref{eq modified iv}-\eqref{eq modified ineq} with $\mu =\bar \mu$, $G= G_0$ and $\Omega = B_{r_0}(x_0)$}\,.
\end{align}
\end{lemma}

\begin{remark}
The uniform convergence assumption is not restrictive, as it is satisfied in the case of blow-ups, as we shall verify shortly, or in any case where $v_k\in W^{1,2}_\loc(B_{3r_0};[0,1])$ enjoy uniform upper and lower frequency bounds according to Lemma \ref{lemma holder}.
\end{remark}

\begin{proof}
Thanks to a translation and rescaling argument we can take, without loss of generality, $x_0=0$ and $r_0=1$. 
Now test \eqref{eq modified ineq} with $\varphi \in C_c^\infty(B_2)$ such that $\varphi \geq \mathbf{1}_{B_{\frac{3}{2}}}$ and combine with \eqref{eq uniform energy bound} and the uniform bounds on $\|\nabla v_k\|_{L^2(B_2)}$ and $\|G'_k\|_{L^1(B_{2})}$ (consequences of \eqref{eq uniform energy bound} and \eqref{eq gvk convergence}) to deduce
\begin{equation}\label{eq bound measures compactness}
	\mu_{k}(B_\frac{3}{2}) \leq C + \int_{B_2}|\nabla \varphi \cdot \nabla v_k|\leq Cr^2 + C\Big(\int_{B_2}|\nabla v_k|^2\Big)^\frac{1}{2} \leq C\, .
\end{equation}
Thus, we may conclude \eqref{eq weak convergence measures0} from \eqref{eq bound measures compactness}. Moreover, taking the limit as $k\to \infty$ in \eqref{eq modified ineq} for $v_k$ using that $\mu_k\weakstar \bar\mu$, \eqref{eq uniform energy bound}, and \eqref{eq gvk convergence}, we find that
\begin{equation}\label{eq blow up limiting}
	2\Delta \bar v = \bar \mu +G_0'(\bar v) \qquad \mbox{ distributionally in $B_\frac{3}{2}$,}
\end{equation}
which verifies the validity of \eqref{eq modified ineq} for $\bar v$ on $B_1$ with $G=G_0$ and $\mu=\bar \mu$.

To complete the proof, it suffices to verify the strong $W^{1,2}$-convergence. Indeed, this would readily imply the validity of both \eqref{eq modified iv} and \eqref{eq modified out} for $\bar v$ with $G=G_0$ and $\Omega =B_1$.  To prove the strong convergence, fix $\varphi \in C_c^\infty(B_\frac{3}{2};[0,1])$ with $\varphi \equiv 1$ on $B_1$, and choose, for each $k$, $w_k\in C_c^\infty(B_2)$ approximating $v_k-\bar v$ well enough in $(L^\infty \cap W^{1,2})(B_{3/2})$ (which is possible since we are assuming that $v_k$ and $\bar v$ are continuous) so that
\begin{equation}
	\bigg|\int_{B_2} \varphi \nabla (v_k - \bar v) \cdot \nabla w_k\,dx - \int_{B_2} \varphi|\nabla (v_{k} - \bar v)|^2 \, dx\bigg| \leq \frac{1}{k}\,, \label{eq wk closeness}
\end{equation}
and
\begin{equation}
	\|w_k\|_{L^2(B_{3/2})} + \| w_k\|_{L^\infty(B_{3/2})} \leq \frac{1}{k} + \| \bar v - v_k\|_{L^\infty(B_{3/2})}\,.\label{eq wk closeness 2}
\end{equation}
Next, if we use \eqref{eq wk closeness}, then subtract \eqref{eq modified ineq} for $v_k$ from \eqref{eq blow up limiting} and test it with $\varphi w_k$, we deduce that
\begin{align*}
	\int_{B_1}|\nabla(v_k - \bar v)|^2\,dx &\leq \int_{B_2} |\nabla (v_{k} - \bar v)|^2 \varphi\, dx\\
	&\leq \int_{B_2} \varphi \nabla (v_k - \bar v) \cdot \nabla w_k\,dx + \frac{1}{k}  \\ 
	&=-\int_{B_2} {w_k}\nabla (v_{k}-\bar v)\cdot \nabla \varphi\, dx  -\int_{B_2} \varphi w_k \,d(\mu_{k} - \bar\mu) \\ 
	&\quad -\int_{B_2} (G_k'(v_k)-G_0'(\bar v))(v_k-\bar v)\varphi+ \frac{1}{k}\,.
\end{align*}
As $k\to \infty$, the first integral vanishes due to the fact that $w_k\to 0$ in $L^\infty$ from \eqref{eq wk closeness 2} and $\nabla v_k-\nabla \bar v\weak 0$ in $L^2$, the second integral vanishes because of the vanshing $L^\infty$-norms of $w_k$ again, combined with $\mu_k\weakstar \bar\mu$, and the last vanishes by the {uniform} convergence of $v_k\to v$ and the $L^2$ convergence \eqref{eq gvk convergence}.
\end{proof}

To apply the preceding compactness arguments, we introduce the rescalings
\begin{equation}\label{e:rescaling}
v_{x_0,r} := \frac{v(x_0 + r \cdot)}{H_{v,x_0}(r)^{1/2}}\quad\mbox{for $x_0 \in {\Om \cap \overline{\{v>0\}}}$ and $r\in (0,\dist(x_0,{\Wbf}))$},
\end{equation}
where $H_{v,x_0}(r)$ is the $L^2$ height function of $v$ centered at $x_0$ as introduced in \eqref{e:D-H-def}. {Note that by Remark \ref{remark:weak uc}, $x_0\in \Om \cap \overline{\{v>0\}}$ implies that
$$H_{v,x_0}(r)>0\qquad  \mbox{for $0<r<\min\{\dist(x,\partial\Om), r_*,1\}$,}$$
so that $v_{x_0,r}$ is well-defined for all small enough $r$.}

\begin{lemma}[Compactness for $v_{x_0,r}$]\label{lemma strong convergence blow-ups}
Suppose that $\wire = \rn \setminus \Om$ is closed, $G\in C^2([0,1])$ satisfies \eqref{eq:g assumptions}, and $v\in W^{1,2}_\loc(\Om;[0,1])$ satisfies \eqref{eq modified iv}-\eqref{eq modified ineq} {and \eqref{eq:lower freq bound assumption theorem 2.1}}.
\begin{enumerate}
	\item Given  $x_0 \in \Om \cap \overline{\{v>0\}}$ and $d =\dist(x_0,{\Wbf})$, if $r<\min\{d,r_*,1\}$ the rescaling $v_{x_0,r}$ satisfies
	
	\begin{equation}\label{eq blow up equation}
		\Delta v_{x_0,r} = \frac{r^2}{{2H_{x_0}(r)^{1/2}}}G'(v_{x_0,r}H_{x_0}(r))+\mu_{x_0,r}  \qquad \mbox{ distributionally in $B_\frac{d}{r}$,}
	\end{equation}
	with 
	\begin{equation*}
		\mu_{x_0,r}= \frac{(\Psi_{x_0,r})_{\#} \, \mu}{2H_{x_0}(r)^{1/2} r^{n-1}}
	\end{equation*}
	and where $(\Psi_{x_0,r})_{\#}\mu$ represents the push-forward of the measure $\mu$ with respect to the function $\Psi_{x_0,r}(y)= \frac{y-x_0}{r}$. 
	
	\item Let $\{x_k\} \subset {\Om \cap \partial \{v>0\}}$ such that $x_k\to {\bar{x}} \in \Omega$ 
	and $r_k \to 0$. Then, up to subsequences, there exists a non-negative Radon measure $\bar{\mu}$ in $B_1$ and a function $\bar v \in (C^{\alpha}\cap W^{1,2})(\overline B_1)$ for some $\alpha \in (0,1]$ such that
	\begin{equation}\label{eq weak convergence measures}
		\mu_{x_k, r_k}\overset{\ast}{\rightharpoonup} \bar \mu,
	\end{equation}
	as measures and 
	\begin{equation}
		v_{x_k,r_k}\to  \bar v
	\end{equation}
	strongly in $W^{1,2}(B_1)$ and locally uniformly as $k\to \infty$. 
	
	\item $\bar v$ satisfies the criticality conditions \eqref{eq modified iv}, \eqref{eq modified out}, and \eqref{eq modified ineq} with $\Omega = B_1$, $G=0$ and $\mu = \bar \mu$.
\end{enumerate}
\end{lemma}
\begin{proof}

We start with (1). Let $x_0\in \Om \cap  \overline{ \{v>0\}}$. Given $\varphi \in C_c^\infty(B_\frac{d}{r})$,  testing \eqref{eq modified ineq} with $\varphi\circ \Psi_{x_0,r}\in C_c^\infty(B_d(x_0))$ (extended by zero to $\Omega$), we obtain
\begin{equation*}
	\int_{B_d(x_0)} \varphi\circ \Psi_{x_0,r}  \, d\mu =\int_{B_d(x_0)} -2\nabla v \cdot \nabla (\varphi\circ \Psi_{x_0,r}) - G'(v)\,\varphi\circ \Psi_{x_0,r}\,dx,
\end{equation*}
which, using the definition of push-forward measure on the left hand side and applying the change of variables  $z = \Psi_{x_0,r}(x)$ on the right, can be rewritten as follows:
\begin{equation*}
	\int_{B_\frac{d}{r}} \varphi  \, d (\Psi_{x_0,r})_{\#} \, \mu =r^{n}\int_{B_\frac{d}{r}}  -2\Big(\nabla v\circ \Psi_{x_0,r}^{-1}\Big) \cdot \nabla \varphi - r G'(v_{x_0,r}H_{x_0}(r)^{1/2})\,\varphi \,dz \,.
\end{equation*}
Dividing both sides by $2H_{x_0}(r)^{1/2} r^{n- 1}$, we obtain \eqref{eq blow up equation} in distributional form.

We now prove (2) and (3). Then by combining \eqref{eq upper bound frequency take 2}, \eqref{eq linf for rescaled guys}, and \eqref{eq uniform holder 2} on a suitable open set containing $\bar x$ and compactly contained in $\Om$, we obtain $\alpha \in (0,1]$ and $k_0\in \mathbb{N}$ such that for all $k\geq k_0$,
\begin{eqnarray}\label{eq uniform bound}
	\Vert v_{x_k,r_k}\Vert_{C^{\alpha}(B_2)} \leq C,\\ \label{eq sobolev bound}
	\Vert v_{x_0,r}\Vert_{W^{1,2}(B_2)} \leq C,
\end{eqnarray}
for some $C$ depending on the open set, $v$, $G$, etc.~ but not on $x_k$. In particular, up to a subsequence, the weak $W^{1,2}$ and uniform convergence of $v_k$ to some $\bar v\in (C^0 \cap W^{1,2})(\overline{B}_1)$ is immediate. Also, as a consequence of \eqref{eq uniform bound} and the Lipschitz bound \eqref{eq:g estimates} for $G'$, we have the estimate
\begin{equation}\label{eq remainder potential}
	\Big|\frac{r^2}{2H_{x_0}(r)^{1/2}}G'(v_{x_0,r}H_{x_0}(r)^{1/2})\Big|\leq Ckr^2
\end{equation}
on $B_2$. From here, we notice that, up to a subsequence, $\{v_{x_k,r_k}\}_k$ satisfies the hypotheses of Lemma \ref{lemma compactness cp} with $G_k\to 0$ in $C^1([0,1])$, which finishes the proof.
\end{proof}

\subsection{Tangent functions and unique continuation}
{In this subsection, we use the compactness results from the preceding subsection to} study blow-ups and tangent functions. Lastly, we prove a unique continuation-type result.
The main consequence of Lemma \ref{lemma strong convergence blow-ups} is the subsequential $W^{1,2}$-compactness of the rescalings $\{v_{x_0, r}\}_r$ as defined by \eqref{e:rescaling} as $r\downarrow 0$. This will allow us to deduce some fundamental properties of the subsequential limits, which we will refer to from now on as \emph{tangent functions} (of $v$). 

In the next lemma, we exploit the compactness properties derived in Lemma \ref{lemma strong convergence blow-ups} to prove in our setting some well-known properties of tangent functions and the behavior of Almgren's frequency function for them.

\begin{lemma}[Tangent functions]\label{lemma prop blowups}
Suppose that $\wire = \rn \setminus \Om$ is closed, $G\in C^2([0,1])$ satisfies \eqref{eq:g assumptions}, and $v\in W^{1,2}_\loc(\Om;[0,1])$ satisfies \eqref{eq modified iv}-\eqref{eq modified ineq} {and \eqref{eq:lower freq bound assumption theorem 2.1}}. If $x_0\in \Om \cap  \overline{ \{v>0\}}$, $r_j$ is a sequence of scales with $r_j \downarrow 0$, then, up to extracting a subsequence, there exists 
\begin{equation}\label{e:blowup-limit}
	\bar v(x) = \lim_{j\to \infty} \frac{v(x_0+r_j x)}{{H_{x_0}(r_j)^{1/2}}}
\end{equation}
which is a non-zero tangent function of $v$ at $x_0$, with the limit taken in $W^{1,2}(B_1)$ and locally uniformly (see Lemma \ref{lemma strong convergence blow-ups}). Also, $\bar v$ satisfies the criticality conditions \eqref{eq modified iv}, \eqref{eq modified out}, and \eqref{eq modified ineq} with $\Omega = B_1$, $G=0$ and a non-negative Radon measure $\bar \mu$, and
\begin{enumerate}
	\item $N_{v, x_0}(0^+) = N_{\bar v, 0}(0^+)$,
	\item $\bar v$ is radially homogeneous of degree $N_{\bar v, 0}(0^+)$, and
	\item $N_{\bar v, e}(0^+) \leq N_{\bar v, 0}(0^+)$ for any $e \in \SS^{n}$ with equality if and only if $\bar v(x+te)=\bar v(x)$ for any $t \in \R$.
\end{enumerate}
\end{lemma}

\begin{proof}

Let us assume without loss of generality that $x_0=0$. In virtue of Lemma \ref{lemma strong convergence blow-ups}, we have that $\bar v \in (W^{1,2}\cap C^\alpha)(B_1)$ and we can assume that the limit $\bar v$ as defined in \eqref{e:blowup-limit} indeed 
exists. From Lemma \ref{lemma strong convergence blow-ups} we also have that $\bar v$ satisfies the criticality conditions \eqref{eq modified iv},  \eqref{eq modified out}, and \eqref{eq modified ineq} with $\Omega = B_1$, $G=0$, and the measure $\bar \mu$ as given by \eqref{eq weak convergence measures} with $x_k \equiv 0$. Furthermore,  $\|\bar v \|_{L^2(\partial B_1)} = 1$, by our choice of normalization.

Now, for any $\rho \in (0,1)$, in virtue of the strong convergence of
\[
v_{r_j} := \frac{v(r_j \cdot)}{H(r_j)^{1/2}}
\]
to $\bar v$ in $W^{1,2}(B_1)$, we have that  $\int_{B_\rho}|\nabla  v_{r_j}|^2\to  \int_{B_\rho}|\nabla \bar v|^2$ and $\int_{\partial B_\rho}  v_{r_j}^2\, d\h\to  \int_{\partial B_\rho}  {\bar v}^2\, d\h$. Additionally, since $x\in \Om \cap \partial \{v>0\}$, Remark \ref{remark:weak uc} implies that $\int_{\partial B_\rho}  {\bar v}^2\, d\h\neq 0$ for any such $\rho$. Thus,
\[
N_{\bar v,0}(\rho) = \frac{\rho \int_{B_\rho}|\nabla \bar v|^2}{\int_{\partial B_\rho} |\bar v|^2}= \lim_{j\to\infty} \frac{ \rho \int_{B_\rho }|\nabla v_{r_j}|^2}{\int_{\partial B_\rho} v_{r_j}^2} =  \lim_{j\to\infty} \frac{ \int_{B_1}|\nabla v_{ \rho r_j}|^2}{\int_{\partial B_1} v_{ \rho r_j}^2\, d\h} = \lim_{j\to\infty} N_{v,0}(\rho r_j) = N_{v,0}(0^+).
\]
Thus, in virtue of the constancy case in Lemma \ref{lemma almgren}, we deduce that $\bar v$ must be radially $\alpha$-homogeneous with $\alpha = N_{\bar v, 0}(0^+)$. From here we also deduce that $N_{\bar v,0}(0^+)= N_{ v,0}(0^+)$. Meanwhile, the conclusion (3) is a simple consequence of the upper-semicontinuity of Almgren's frequency function and a blowdown argument; see, for instance, \cite[Section 3.3]{Simon_theorems_on_reg}.\end{proof}

\begin{lemma}[Classification of planar tangent functions]\label{lemma:planar tangent classification}
If $\wire = {\R^{2}} \setminus \Om$ is closed, $G\in C^2([0,1])$ satisfies \eqref{eq:g assumptions}, and $v\in W^{1,2}_\loc(\Om;[0,1])$ satisfies \eqref{eq modified iv}-\eqref{eq modified ineq} {and \eqref{eq:lower freq bound assumption theorem 2.1}}, then up to rotation any tangent function $\bar v$ at $x_0\in \Om \cap  \overline{ \{v>0\}}$ is given by
\begin{equation}\label{e:structure-2d-tangents}
\bar v(r,\theta) = \frac{1}{\sqrt{\pi}} r^{N/2} \bigg|\sin\left(\frac{N\theta}{2} \right)\bigg| , \qquad N\in \N_{\geq 2}\, ,
\end{equation}
where $(r,\theta)$ denote polar coordinates in $\R^2$.
\end{lemma}

\begin{remark}
    {An immediate consequence of Lemma \ref{lemma:planar tangent classification}, when combined with Lemma \ref{lemma prop blowups} and Lemma \ref{lemma holder}, is that planar solutions of \eqref{eq modified iv}-\eqref{eq modified ineq} {satisfying in addition \eqref{eq:lower freq bound assumption theorem 2.1}} (with $G\in C^2([0,1])$ satisfying \eqref{eq:g assumptions}) are Lipschitz.}
\end{remark}

\begin{proof}
If $\bar v$ is any tangent function to $v$ at $x_0$, thanks to Lemma \ref{lemma prop blowups} we have that $\bar v\Delta \bar v =0$ weakly in $B_1$ (recall that this is \eqref{eq modified out} with $G=0$). {Furthermore, the homogeneity of $\bar v$ and the fact that $v\in W^{1,2}(\Om)$ together imply that $\restr{\bar v}{\partial B_r}$ belongs to $W^{1,2}(\partial B_r)$ for every $0<r<1$ instead of just almost every $r$. Thus by the Morrey-Sobolev embedding in one dimension and the homogeneity, $\bar v$ is continuous on $B_1$ and thus $\{\bar v=0\}$ is closed. As a consequence, $\Delta \bar v = 0$ on the open set $\{\bar v>0\}$ in the classical sense}, and writing the equation in polar coordinates $(r,\theta)$ yields
$$\Delta \bar v = \partial_{rr} \bar v + r^{-1} \partial_r \bar v + r^{-2}\partial_{\theta \theta} \bar v = 0 \qquad \text{on $B_1\setminus \{\bar v=0\}$.}$$
In virtue of Lemma \ref{lemma prop blowups}, we can exploit the radial homogeneity of $\bar v$, {the degree of which we denote by $\alpha>0$}, to conclude that we have
\[
r^{-2} [\alpha^2 \bar v + \partial_{\theta \theta} \bar v] = 0\, ,
\]
in any open, convex cone $\Ccal$ formed from a single connected component of $\R^2\setminus \{\bar v = 0\}$. Solving this ODE in $\theta$, we obtain
\begin{equation}\label{eq cone formula}
\bar v(r,\theta) = r^\alpha[a \sin (\alpha \theta) + b\cos(\alpha \theta)] \quad \text{in $\Ccal$\,,}
\end{equation}
for some $a,b\in \R$. Up to rotation, we may without loss of generality assume that $\bar v=0$ when $\theta = 0$. Thus, $b=0$. Furthermore, observe that the exponent $\alpha$ is the radial homogeneity of $\bar v$, so is \emph{the same} for any such convex cone that is a connected component of $\R^2\setminus \{\bar v = 0\}$. \\

We claim that $\{\bar v = 0\}$ consists of finitely many half-lines emanating from the origin. Indeed, observe that we have already demonstrated the fact that $\bar v$ has radial homogeneity of fixed degree $\alpha$ in each open, convex, connected conical component of $\R^2 \setminus \{\bar v = 0\}$. This in particular implies that the angle of any such conical component must be an integer multiple of $\frac{\pi}{\alpha}$, in order to ensure that $\bar v=0$ on the boundary of the cone. This in turn implies that there are only finitely many such connected components. Their complement will thus consist of finitely many closed, convex cones $K_1,\dots, K_N$, on each of which $\bar v = 0$. {By a standard argument based on the inner variational equation \eqref{eq modified iv} (see e.g. \cite[Theorem 2.4]{ACF} or \cite[Proposition 1.4]{maggi2023hierarchy}), on $\partial K_i\setminus \{0\}$ for each $i$ we have the transmission condition 
$$
|\partial_\nu^- u| = |\partial_\nu^+ u|\,,
$$
for the one-sided normal derivatives of $u$. If $K_i$ had non-empty interior for some $i$, this gives a contradiction, since, coming from the side where $\bar v >0$, the one-sided derivative normal derivative does not vanish (as is seen by direct computation using \eqref{eq cone formula}). So each $K_i$ must have empty interior and be a half-line.} This observation combined with the periodicity of $\sin(\alpha \theta)$ and the fact that $\bar v$ is given by \eqref{eq cone formula} on connected components of $\{\bar v>0\}$ implies that $\alpha = \frac{N}{2}$, with $N\geq 1$. {Additionally, $a=\frac{1}{\sqrt{\pi}}$ since $\bar v\geq 0$ and $\Vert \bar v\Vert_{L^2(\partial B_1)}=1$.} \\

We complete the proof by observing that $\bar v(r,\theta) =\frac{1}{\sqrt{\pi}}r^\frac{1}{2}\sin (\tfrac{\theta}{2})$ cannot arise as a tangent function. This is the case because $N_{\bar v, 0}(0^+) =\frac{1}{2}$ whereas $N_{\bar v, te_1}(0^+)=1$ for any $t\in (0,1)$, simply because $\bar v$ is Lipschitz at any of those points (this can be explicitly verified). This clearly contradicts the upper semicontinuity of $x\mapsto N_{\bar v, x}(0^+)$.\end{proof}

Finally, we prove a sort of unique continuation result for solutions of \eqref{eq modified iv}-\eqref{eq modified ineq} that, roughly speaking, says the free boundary $\{v=0\}$ is Lebesgue negligible if $v$ is non-constant.

\begin{lemma}[Unique continuation]\label{lemma:unique continuation}
Let $v\in W^{1,2}_\loc(\Om;[0,1])$ satisfy \eqref{eq modified iv}-\eqref{eq modified ineq} with $G$ satisfying \eqref{eq:g assumptions} and {\eqref{eq:lower freq bound assumption theorem 2.1}}. Then for any connected component $\Omega'$ of $\Omega$, either $\Lcal^{n+1}( \{v=0\}\cap\Omega')=0$ or $v=0$ $\Lcal^{n+1}$-a.e.~ in $\Omega'$. As a consequence, if $\restr{v}{\Om'}$ is not the zero function, then $\Om' \cap \overline{\{v>0\}}=\Om'$ and the upper frequency bound \eqref{eq upper bound frequency take 2}, doubling estimate \eqref{eq doubling shell take 2}, and the $L^\infty$-bound \eqref{eq linf for rescaled guys} hold on {any $U \cc \Omega'$} with constants independent of $x\in U$.
\end{lemma}

\begin{proof}[Proof of Lemma \ref{lemma:unique continuation}]
The validity of \eqref{eq upper bound frequency take 2}, \eqref{eq doubling shell take 2}, and \eqref{eq linf for rescaled guys} on $U\cc \Om'$ if $\restr{v}{\Om'}$ is not the zero function follow immediately from Lemma \ref{lemma: uc take 2} since $\Om' \cap \overline{\{v>0\}}=\Om'$. So we prove that either $\Lcal^{n+1}( \{v=0\}\cap\Omega')=0$ or $v=0$ $\Lcal^{n+1}$-a.e.~ in $\Omega'$. Suppose, for contradiction, that for some connected component $\Om'$ of $\Om$,
\begin{equation}\label{eq uc contra assumption}
0<\mathcal{L}^{n+1}(\Om' \cap \{v=0\})<\mathcal{L}^{n+1}(\Om')\,.
\end{equation}
Then, {since $\Omega'$ is connected,} the perimeter $P(\{v=0\};\Om')$ of $\{v=0\}$ in $\Omega'$ is either infinity or strictly positive; it cannot be zero. Letting
\begin{equation}\notag
\Om' \cap \partial^e \{v=0\} = \{x\in \Om' : x\notin \{v=0\}^\one \cup \{v=0\}^\zero \}
\end{equation}
denote the essential boundary of $\{v=0\}$ relative to $\Om$, where $\{v=0\}^{(i)}$ denote the points in $\{v=0\}$ of Lebesgue density $i$, we claim that it is non-empty. Indeed, if it were empty, then by Federer's criterion for sets of finite perimeter, we must have $P(\{v=0\};\Om')=0$, which is impossible. So there exists $x\in \Om' \cap \pa^e \{v=0\}$. By the containment $\Om' \cap \partial^e \{v=0\}\subset \Om' \cap \partial \{v>0\}$ and \eqref{eq leb dens bound}, we have
\begin{equation}\notag
\limsup_{r\to 0} \frac{\mathcal{L}^{n+1}(\{v=0\} \cap B_r(x))}{\omega_{n+1}r^{n+1}}< 1\,.
\end{equation}
Since the limsup vanishing would imply that $x\in \{v=0\}^\zero$ (against $x\in \partial^e \{v=0\}$) there must exist $r_j\to 0$ and $\beta\in (0,1)$ such that
\begin{equation}\notag
\limsup_{j\to \infty} \frac{\mathcal{L}^{n+1}(\{v=0\} \cap B_{r_j}(x))}{\omega_{n+1}r_j^{n+1}}= \beta \in (0,1)\,;
\end{equation}
by restricting to a further subsequence using Lemma \ref{lemma prop blowups}, we may obtain a tangent function $\bar v$ such that 
\begin{equation}\label{eq bad leb density}
0<\mathcal{L}^{n+1}(B_1 \cap \{\bar v=0\})<\mathcal{L}^{n+1}(B_1)\,.
\end{equation}
We use \eqref{eq bad leb density} to obtain a contradiction, first in two dimensions and then in higher dimensions.

\medskip

\noindent{\it Contradiction in 2D}: The equation \eqref{eq bad leb density} directly contradicts the classification of tangent functions in Lemma \ref{lemma:planar tangent classification}, since it implies that $\bar v$ has Lebesgue non-trivial zero set.

\medskip

\noindent{\it Contradiction in higher dimensions}:
By the same perimeter argument as above, \eqref{eq bad leb density} implies the existence of $y\in B_1 \cap \partial^e \{v>0\}$, and again \eqref{eq leb dens bound} implies the existence of $s_j\to 0$ such that
\begin{equation}\notag
\limsup_{j\to \infty} \frac{\mathcal{L}^{n+1}(\{\bar v =0\} \cap B_{s_j}(y))}{\omega_{n+1}s_j^{n+1}} \in (0,1)\,.
\end{equation}
Up to a further subsequence, we therefore have a non-zero tangent function $w$ to $\bar v$ at $y$ with $\mathcal{L}^{n+1}$-nontrivial zero set. Furthermore, since $N_{\bar v,ty}(0^+)$ is constant for $t\in (0,\infty)$, parts one and three of Lemma \ref{lemma prop blowups} show that $w$ is independent of $y$, {namely, it is translation-invariant in the $y$-direction}. Thus the restriction $w:y^\perp \to \mathbb{R}$ is a homogeneous solution of \eqref{eq modified iv}-\eqref{eq modified ineq} with $G=0$ and $\mu=0$ in $\mathbb{R}^n$, and $\mathcal{L}^n$-nontrivial zero set. By induction, since there is no such solution in $\mathbb{R}^2$, it is impossible in $\mathbb{R}^{n+1}$ and we have a contradiction.
\end{proof}

\subsection{Sharp frequency lower bound and the proof of Theorem \ref{thm:holder reg for crit points}}






Our final step is to improve the initial H\"older regularity to Lipschitz regularity via a blow-up analysis. In this order of ideas, given any $x_0 \in \{v=0\}$, we recall the blowups
\begin{equation}\notag
v_{x_0,r} := \frac{v(x_0 + r \cdot)}{H_{x_0}(r)^{1/2}}.
\end{equation}
from \eqref{e:rescaling} for $r\in (0,\dist(x_0,{\Wbf}))$, where  $H_{x_0}(r)$ is the $L^2$ height function of $v$ centered at $x$ as introduced in \eqref{e:D-H-def}. The next result is a classification of tangent functions which, in particular, completes our regularity analysis.

\begin{proposition}\label{p:tangents} 
Let $\wire = \rn \setminus \Om$ be closed, $G\in C^2([0,1])$ satisfy \eqref{eq:g assumptions}, and $v\in W^{1,2}_\loc(\Om;[0,1])$ satisfy \eqref{eq modified iv}-\eqref{eq modified ineq} {and \eqref{eq:lower freq bound assumption theorem 2.1}}. Then, $N_{v,x_0}(0^+)\geq 1$ for any $x_0\in \{v=0\}$ {such that $v\not\equiv 0$ on the connected component of $\Omega$ containing $x_0$}.  Moreover, {for any sequence $\{v=0\}\ni x_k \to x_0$ with $N_{v,x_0}(0^+) = 1$, any subsequential limit
\begin{equation}\label{e:diagonal-blowp}
    \bar v := \lim_{k\to\infty} \frac{v(x_k + r_k \cdot)}{H_{x_k}(r_k)^{1/2}}
\end{equation}
satisfies}
\begin{equation}\label{e:linear-tangent}
    \bar v(x) = \frac{1}{\sqrt{\omega_{n+1}}} |x\cdot e|\, ,
\end{equation}
for some $e \in \mathbb{S}^{n}$ and where $\omega_{n+1}$ is the Euclidean volume of the unit ball in $\R^{n+1}$. {In particular, \eqref{e:linear-tangent} holds true for any tangent function $\bar v$ to $v$ at $x_0$.}
\end{proposition}

\begin{proof}[Proof of Proposition \ref{p:tangents}]

The proof is divided into steps.

\medskip

\noindent {\it Step 1.} In this step we carry out a dimension reduction argument to show that $N_{\bar v, 0}(0^+)\geq 1$ in all dimensions for any tangent function $\tilde{v}$.

Let us start by noticing that {any non-zero homogeneous solution $v$ of \eqref{eq modified iv}-\eqref{eq modified ineq} with $G=0$ and $\Omega={\mathbb{R}^{n+1}}$ satisfies $\inf_{\mathbb{R}^{n+1}}N_{v,x}(0^+)=N_{v,0}(0^+)$ and thus satisfies a uniform lower frequency bound on all of space.} Therefore, the class of {non-zero} homogeneous solutions of \eqref{eq modified iv}-\eqref{eq modified ineq} with $G=0$ and $\Omega={\mathbb{R}^{n+1}}$ is closed under taking tangent functions at any point, in virtue of Lemma \ref{lemma prop blowups} {(note that by Lemma \ref{lemma:unique continuation}, any such solution $\bar v$ satisfies $\overline{\{\bar v>0\}}=\mathbb{R}^{n+1}$, so Lemma \ref{lemma prop blowups} holds at every point)}. This class contains, in particular, all tangent functions to $v$. Let us denote this class as $\tau_{n+1}$ and let us define
\begin{equation*}
m_{n+1} = \inf_{\bar v \in \tau_{n+1}}\inf_{x \in B_1} N_{\bar v, x}(0^+)\, ,
\end{equation*}
which, thanks to the closure of $\tau_{n+1}$ with respect to blow-ups and property (1) in Lemma \ref{lemma prop blowups}, can be written as
\begin{equation}\label{eq minimal freq}
m_{n+1} = \inf_{\bar v \in \tau_{n+1}} N_{\bar v, 0}(0^+)\, .
\end{equation}

Our goal is to show, by induction on the dimension $n + 1$, that $m_{n+1}\geq 1$. Let us notice that by Lemma \ref{lemma:planar tangent classification} the base case $n=1$ is already covered. Let us assume now that $n\geq 2$. Suppose that for every dimension $\bar n \leq n$,
\[
m_{\bar{n}} \geq 1.
\]
First of all, we claim that \eqref{eq minimal freq} is attained for some $\bar v \in \tau_{n+1}$. Indeed, if $\bar v_k$ is an infimizing sequence for \eqref{eq minimal freq}, by homogeneity of $\bar v_k$ we have 
\begin{equation}\label{eq bound freq}
N_{\bar v_k, 0}(1) = N_{\bar v_k, 0}(0^+)\to m_{n+1} 
\end{equation}
as $k \to \infty$. In particular, the functions $\tilde v_k = \frac{\bar v_k}{H_{\bar v_k}(1)}$ satisfies the hypotheses of Lemma \ref{lemma compactness cp}, implying that $\tilde v_k $ converges locally strongly in $W^{1,2}$, up to subsequences, to a {non-zero} function $\bar v_0 \in \tau_{n+1}$ such that $N_{\bar v_0, 0}(0^+)= N_{\bar v_0, 0}(1)=m_{n+1}.$


Now take $\bar v_0$ attaining \eqref{eq minimal freq}; we claim that $\bar v_0$ is translation-invariant along some line through the origin. In other words, up to rotation, $\bar v_0(x_1,\dots, x_{n+1}) = w(x_1,\dots, x_{n})$ for an $m_{n+1}$-homogeneous solution $w$ of \eqref{eq modified iv}-\eqref{eq modified ineq} on $\R^{n+1}$. Observe that after proving this, the inductive hypothesis would imply that $ m_{\bar{n}} \geq 1$. Turning into the proof of the claim, let us notice that $0$ cannot be an isolated zero for $\bar v_0$, otherwise $\bar v_0$ would be a continuous function in $B_1$, harmonic in $B_1 \setminus \{0\}$, which implies that $\bar v_0$ is harmonic in $B_1$ yielding a contradiction to the minimum principle- since $\bar v_0(0) = 0$. Hence, we have deduced the existence of a ray of zeros with frequency greater or equal than  $m_{\bar{n}}$ which combined with Lemma \ref{lemma prop blowups} proves the claim.

\medskip

\noindent {\it Step 2.} {We complete the proof by characterizing limiting functions $\bar v$ given by \eqref{e:diagonal-blowp} whenever $N_{v, x_0}(0^+)=1$. Given such a function $\bar v$, we begin by demonstrating that $\bar v$ is still radially 1-homogeneous in this case, despite the varying centers. In light of Lemma \ref{lemma strong convergence blow-ups} and Lemma \ref{lemma almgren}, it suffices to demonstrate that $r\mapsto N_{\bar v, 0}(r)$ is identically equal to 1. Fix $\eps> 0$ arbitrarily. Since $N_{v, x_0}(0^+)=1$, the absolute continuity of $N$ guarantees that there exists $\bar\rho \in (0,\dist(x_0,\partial\Omega))$ such that
\[
    N_{v,x_0}(\rho) \leq 1+ \frac{\eps}{4}\qquad \forall \rho\in (0,\bar\rho]\,.
\]
In particular, when combined with the Lemma \ref{lemma strong convergence blow-ups}, we have
\begin{equation}\label{e:pinched-freq}
    N_{v,x_k}(\bar\rho) \leq 1+\frac{\eps}{2}\,,
\end{equation}
for every $k$ sufficiently large. Now for any given $r>0$, up to taking $k$ even larger if necessary so that $r_k \leq \frac{\bar\rho}{r}$, we further have
\[
    N_{v,x_k}(r_k r) + 1 \leq e^{\tfrac{\kappa (\bar\rho^2- r_k r)}{2}}(N_{v,x_k}(\bar\rho) + 1)\,.
\]
By further decreasing $\bar\rho$ if necessary and combining with \eqref{e:pinched-freq}, we can therefore ensure that
\[
    N_{v,x_k}(r_k r) \leq 1+\eps\,,
\]
for all $k$ sufficiently large. Letting $v_k:=\frac{v(x_k + r_k\cdot)}{H_{x_k}(r_k)^{1/2}}$, we have $N_{v,x_k}(r_k r)=N_{v_k,0}(r)$, and so Lemma \ref{lemma strong convergence blow-ups} guarantees that $N_{\bar v,0}(r) \leq 1+\eps$. Since $\eps>0$ is arbitrary, we deduce that $N_{\bar v,0}(r) \leq 1$.

Now, in Step 1 we verified that $N_{\bar v,0}(0^+) \geq 1$, which, combined with the monotonicity in the $G=0$ case of Lemma \ref{lemma almgren} yields $N_{\bar v,0}(r) \geq 1$. The desired conclusion that $N_{\bar v, 0}\equiv 1$ follows. 

Thus, any such limit $\bar v$ lies in the class $\tau_{n+1}$ and attains $m_{n+1}$ as in \eqref{eq minimal freq}. The argument in Step 1, iterated inductively, in fact implies that up to rotation, $\bar v = \bar w(x_1,x_2)$ for a $1$-homogeneous function $\bar w$ satisfying \eqref{eq modified iv}-\eqref{eq modified ineq} with $G=0$ in $B_1 \subset \R^2$. Hence, by the classification in $\R^2$, we find that $\bar v$ must be a rotation of $L |x_1|$. Finally, since $\Vert \bar v\Vert_{L^2(\partial B_1)}=1$, we have that $L=\frac{1}{\sqrt{\omega_{n+1}}}$.}
\end{proof}

We conclude with the proof of the main theorem of this section.

\begin{proof}[\textit{Proof of Theorem \ref{thm:holder reg for crit points}}]
First of all, by Lemma \ref{lemma:unique continuation}, $\{v=0\}\cap {\Om}'$ is Lebesgue null whenever ${\Om}'\subset \Om$ is a connected component on which $v\not\equiv 0$. So we may as well assume that $\overline{\{v>0\}}\cap \Om=\Om$, since the conclusions are trivial when $v\equiv 0$ on a given connected component of $\Omega$.
To prove item (i), thanks to Proposition \ref{p:tangents}, we note that  $N_{v,x_0}(0^+)\geq 1$ for any $x_0\in \{v=0\}$, thus a direct application of Lemma \ref{lemma holder} implies local Lipschitz continuity for $v$ together with the estimate \eqref{eq uniform lip 2}. {To prove (ii), by the monotonicity of the frequency and the local frequency bound (Lemma \ref{lemma: uc take 2}), it suffices to show that (up to renaming $r_{**}$)
\begin{equation}\label{eq limsup freq}
    \limsup_{|x|\to \infty}N_{v,x}(r_{**})< \infty\,.
\end{equation}
Now since $\nabla v \in L^2$, $r_{**}D_{v,x}(r_{**})$ decays uniformly as $|x|\to \infty$. Furthemore, since $\mathcal{L}^{n+1}(\{v<t\})<\infty$ for all $t\in (0,1)$ (in particular for $t=1/2$), Chebyshev's inequality yields
\begin{equation}\notag
    \int_{B_{r_{**}}(x)}v^2\,dy \geq \mathcal{L}^{n+1}(\{v>1/2\} \cap B_{r_{**}(x)})/4 \to \omega_{n+1}r_{**}^{n+1}/4\qquad \mbox{uniformly as $|x|\to \infty$}
\end{equation}
also. Thus by Fubini's theorem, there exists $c>0$ such that for all large enough $|x|$, $H_{v,x}(r)>c$ for some $r\in (r_{**}/2,r_{**})$. After replacing $r_{**}$ with $r_{**}/2$, these two observations and the monotonicity of $N$ imply \eqref{eq limsup freq}.
}
\end{proof}

\section{Regularity and structure of the free boundary}\label{s:fb-structure}

In this section, we begin our description of the structure of the free boundary $\{u=1\}$ for solutions $u\in W^{1,2}_\loc(\Omega;[0,1])$ of \eqref{eq innervariation}-\eqref{eq:differential inequality} {satisfying an additional lower frequency bound}. In order to carry out our analysis of the free boundary, we crucially rely on the following proposition, which establishes a local separation property for the set $\{v>0\}$ of $v=1-u$, into two components near points of frequency 1.

\begin{proposition}\label{lemma density components}
    Let $v\in W^{1,2}_\loc(\Omega;[0,1])$ be a solution of \eqref{eq modified iv}-\eqref{eq modified ineq} {satisfying \eqref{eq:lower freq bound assumption theorem 2.1}} and suppose that $x_0\in\{v=0\}$. {In addition, suppose that $N_{x_0}(0^+) = 1$ and that there exists $R_0 > 0$ such that for each $y\in B_{R_0}(x_0)\cap \{v=0\}$, $N_y(0^+)=1$.}

Then there exists {$r_0\in (0,\tfrac{R_0}{2})$} (depending on $x_0$) such that $\{v>0\}\cap B_{r_0}(x_0)$ has exactly 2 connected components.
\end{proposition}

\begin{remark}
    The additional requirement that $\{y: N_y(0^+) = 1\}$ is relatively open in $B_{R_0}(x_0)\cap \{v=0\}$ in Proposition \ref{lemma density components} will shortly become superfluous; see Corollary \ref{corollary decomp fb}.
\end{remark}

\begin{proof}
    The argument follows the same reasoning as that in the proof of \cite[Proposition 5.4]{TT}. We provide an outline here for the purpose of clarity, and refer the reader to \cite{TT} for more details.

\medskip

    \noindent\textit{Step 1.} We claim that for any $\delta \in (0,1)$, there exists $r_0=r_0(x_0,\delta)\in (0,\tfrac{R_0}{2})$ such that $\{v=0\}\cap B_{R_0/2}(x_0)$ is $(\delta,r_0)$-Reifenberg flat, namely for each $x\in \{v=0\}\cap B_{R_0/2}(x_0)$ and $r\in (0,r_0]$, there exists an $n$-dimensional linear subspace $L_{x,r}$ such that
    \begin{equation}\label{e:RF}
        d_H(\{v=0\}\cap B_{r}(x), (x + L_{x,r})\cap B_r(x)) \leq \delta r\,,
    \end{equation}
    where $d_H$ denotes the Hausdorff distance. To see that this claim holds, we argue by contradiction. Namely, suppose that there exists $\delta> 0$ such that for some sequence $r_k \downarrow 0$ and $\{v=0\}\cap B_{R_0/2}(x_0) \ni x_k \to \bar x$ with $N_{\bar x}(0^+)=1$, the rescalings
    \[
        v_{x_k,r_k}(x) := \frac{v(x_k + r_k x)}{H_{x_k}(r_k)^{1/2}}
    \]
    satisfy
    \begin{equation}\label{e:nonplanar-zero-set}
        d_H(\{v_{x_k,r_k}=0\}\cap B_1, L\cap B_1) > \delta\,,
    \end{equation}
    for any $n$-dimensional linear subspace $L$. Applying Lemma \ref{lemma strong convergence blow-ups} and Lemma \ref{p:tangents} and recalling that $N_{\bar x}(0^+)=1$ in light of the hypothesis, we conclude that $v_{x_k,r_k} \to \bar v$ in $W^{1,2}(B_1)$ and locally uniformly, where $v(x) = \tfrac{1}{\omega_{n+1}}|x\cdot e|$ for some $e\in \Sbb^n$. In particular, $\{\bar v=0\} = L_0 \cap B_1$ for some $n$-dimensional linear subspace $L_0$. This implies that
    \begin{equation}\label{e:zeros-conv-d_H}
        d_H(\{v_{x_k,r_k}=0\}\cap B_1, L_0\cap B_1) \to 0\,.
    \end{equation}
    Indeed, this can be proven directly from the definition; one inclusion is a mere consequence of the uniform convergence, while the other is due to the fact that there must be zeros of $v_{x_k,r_k}$ converging to each zero of $\bar v$, in light of the minimum principle for harmonic functions. {More precisely, if there is a zero $\bar x \in L_0 \cap B_1$ of $\bar v$ which has a neighborhood around it containing no zeros of $v_{x_k,r_k}$ for all $k$ sufficiently large, then $v$ has an isolated zero at $\bar x$, which violates the minimum principle.} The validity of \eqref{e:zeros-conv-d_H}, however, directly contradicts \eqref{e:nonplanar-zero-set}. 

\medskip

\noindent\textit{Step 2.} We may now exploit the local Reifenberg flatness of $\{v=0\}$ around $x_0$ to deduce the local separation property as follows. By Step 1, given a fixed absolute $\delta \in (0,\tfrac{1}{4})$, there exists a linear $n$-dimensional subspace $L_{x_0,r_0}$ such that \eqref{e:RF} holds with $x=x_0$ and $r=r_0(x_0,\delta)$. Thus, letting $B_0^\pm$ denote the two connected components of $B_{r_0}(x_0)\setminus B_{\delta r_0}(x_0+L_{x_0,r_0})$, {where the latter denotes the open neighborhood of radius $\delta r_0$ around the affine subspace $x_0+L_{x_0,r_0}$}, there exist two connected components $D^\pm$ of $\{v>0\}\cap B_{r_0}(x_0)$ such that $B_0^+ \subset D^+$ and $B_0^-\subset D^-$. Define a function $\epsilon: B_0^+\cup B_0^- \to \{+1,-1\}$ by
    \[
        \epsilon =
        \begin{cases}
            +1 & \text{in $B_0^+$} \\
            -1 & \text{in $B_0^-$}\,.
        \end{cases}
    \]

    We may now cover $B_{\delta r_0}(x_0+L_{x_0,r_0})$ by a finite number of balls $B_{r_0/2}(x_i)$, $i=1,\dots N$, with $x_i \in \{v=0\}$, and apply the conclusion of Step 1 to each of these balls. Proceeding as above and exploiting overlaps, this implies that
    \[
        \bigcup_{i=1}^N B_{r_0/2}(x_i)\setminus B_{\delta r_0/2}(x_i+L_{x_i,r_0/2})
    \]
    consists of two mutually disjoint connected components $B_1^+$ and $B_1^-$, which are respectively contained in $D^+$ and $D^-$. Moreover, we may continuously extend $\epsilon$ to $B_0^+\cup B_0^-\cup B_1^+\cup B_1^-$. We may now proceed iteratively, using balls of radius $\frac{r_0}{2^k}$ at the $k$-th stage of the iteration, at each stage extending $\epsilon$ continuously to the pair of mutually disjoint connected components $\bigcup_{j=0}^k B_j^+ \cup \bigcup_{j=0}^k B_j^-$ formed at each stage. A final application of the Reifenberg property at the nearest point in $\{v=0\}\cap B_{r_0}(x_0)$ to an arbitrary given point in $\{v>0\}\cap B_{r_0}(x_0)$, at a scale comparable to the distance between these two points, guarantees that $\epsilon$ extends continuously to the entirety of $\{v>0\}\cap B_{r_0}(x_0)$; the conclusion follows (see \cite{TT} for more details).
\end{proof}

We now characterize points $x_0\in \{v=0\}$ with $N_{v,x_0}(0^+) >1$.

\begin{proposition}\label{p: freq gap}
Let $v\in W^{1,2}_\loc(\Omega)$ be a solution of \eqref{eq modified iv}-\eqref{eq modified ineq} {satisfying \eqref{eq:lower freq bound assumption theorem 2.1}}. Suppose that $x_0\in \partial \{v>0\}\subset \Omega$ and that $N_{v,x_0}(0^+)> 1$. Then, any tangent function $\bar v$ at $x_0$ satisfies the following dichotomy. Either: 
\begin{itemize}
\item[(1)] there exists $e_0\in \{\bar v=0\}\cap \partial B_1$ with $N_{\bar v, e_0}(0^+) \geq \frac{3}{2}$, or
\item[(2)] $\bar v = |h|$ where $h$ is a homogeneous harmonic polynomial of degree at least 2.
\end{itemize}
In particular, we have $N_{v,x_0}(0^+)\geq \frac{3}{2}$.
\end{proposition}

{The proof of Proposition \ref{p: freq gap} follows a very similar line of reasoning to that of \cite[Proof of Lemma 4.2]{SoaveTerracini}; however, we provide a proof here for clarity and due to the fact that we learned of the result \cite[Lemma 4.2]{SoaveTerracini} after this article was completed.} In order to prove Proposition \ref{p: freq gap}, we require the following key characterization of radially homogeneous minimizers of our variational problem.

\begin{lemma}\label{l:harmonic-polyn}
Suppose that $v$, $x_0$ and $\bar v$ are as in Proposition \ref{p: freq gap}. Moreover, suppose that $n \geq 2$ and that $N_{\bar v, e}(0^+)=1$ for every $e\in \{\bar v=0\}\cap \partial B_1$. Then $\bar v = |h|$ for a harmonic polynomial $h$.
\end{lemma}

We remark that tangent functions of the type (2) in Proposition \ref{p: freq gap} do indeed exist. In \cite{lewy1977mininum} it is shown that homogeneous harmonic polynomials of even degree must have at least three nodal domains and, moreover, that for every $k$ even there exist a harmonic polynomial of degree $k$ with exactly 3 nodal domains in $\SS^2$. Similarly, Lewy showed that for any $k$ odd there exists a polynomial of degree $k$ with exactly 2 nodal domains in $\SS^2$ (see \cite[Figure 1]{badger2011harmonic} for an explicit example of the latter).

Our proof of Lemma \ref{l:harmonic-polyn} relies on the following reflection property for $v$ locally around points with frequency 1, which is useful in its own right.

\begin{lemma}\label{l:reflection}
Suppose that $\Wbf$, $G$ and $v$ are as in Lemma \ref{thm:inter reg for crit points}. In addition, suppose that for some $B_{\rho_0}(x_0) \subset \Omega$ centered at a point $x_0\in \Rcal(u)$, there are exactly two connected components $B^\pm$ of $\{v>0\}\cap B_{\rho_0}(x_0)$ and suppose that $N_{v,y}(0^+) = 1$ for every $y\in \{v=0\}\cap B_{\rho_0}(x_0)$. Then the function $\tilde v := v\mathbf{1}_{\overline{B}^+} - v\mathbf{1}_{B^-}$ is a weak solution of
\begin{equation}\label{e: reflection}
\Delta \tilde v = \frac{1}{2}\tilde H(\tilde v) \mbox{  in  $B_{\rho_0}(x_0)$}\,,
\end{equation}
where $\tilde H$ is the odd reflection of $G'$, i.e.,
\begin{equation}\label{eq odd reflection}
\tilde H(t) = 
\begin{cases}
	G'(t), \mbox{ if $t\in[0,1]$},\\
	-G'(-t), \mbox{if $t \in [-1,0)$}\,.
\end{cases}
\end{equation}
\end{lemma}

In order to prove Lemma \ref{l:reflection}, we require the following basic property of $v$ restricted to its connected components.

\begin{lemma}\label{lemma restriction subsol}
Let $v\in W^{1,2}_\loc(\Omega;[0,1])$ satisfy \eqref{eq modified iv}-\eqref{eq modified ineq} {and \eqref{eq:lower freq bound assumption theorem 2.1}}, with $G$ satisfying \eqref{eq:g assumptions}. Let $x_0\in\{v=0\}$, let $r>0$ be such that $B_r(x_0) \subset \Omega$, and let $D \subset B_{r}(x_0)$ be an open set such that $v=0$ on $\partial D \cap B_r(x_0)$. Then, the function
\begin{equation*}
v_1(x) =
\begin{cases}
	v(x), \qquad  x\in D,\\
	0, \qquad x\in B_{r}(x_0)\setminus D
\end{cases}
\end{equation*}
is Lipschitz in $B_r(x_0)$ and satisfies, in the sense of distributions, the equation
\begin{equation}\label{eq subsolution}
2 \Delta v_1 = G'(v_1) +\mu_1, \quad \mbox{in $B_r(x_0)$}
\end{equation}
for some non-negative Radon measure $\mu_1$. 
\end{lemma}
\begin{proof}[Proof of Lemma \ref{lemma restriction subsol}]
As usual, we may assume without loss of generality that $x_0=0$. In virtue of the Lipschitz regularity of $v$ proved in Theorem \ref{thm:holder reg for crit points}, we immediately have that $v_1$ is also Lipschitz continuous.

Let us notice that \eqref{eq subsolution} amounts to showing that $2\Delta v_1 - G'(v_1) \geq 0$ in the sense of distributions, {in light of the correspondence between monotone linear functionals on $C_c^\infty(\Om)$ and non-negative Radon measures (see e.g. \cite[pg 53]{EG})}.  On the other hand, since $v_1$ and $v$ agree on the open set $D$ and we have the validity of \eqref{eq modified ineq} for $v$, it suffices to show that $2\Delta v_1 - G'(v_1) \geq 0$ nearby $\partial D \cap B_r$. Let $y_0 \in \partial D \cap B_r$ and let $\varphi \in C^\infty_c(B_r)$ be a non-negative test function supported on a neighborhood of $y_0$. Since $v_1$ is Lipschitz, and satisfies $2\Delta v_1=G'(v_1)$ in $D$, we have that for almost every $t\in (0, \infty)$, $\{v_1>t\}$ is a set of finite perimeter and the integration by parts formula
\begin{eqnarray}\notag
- 2\int_{\{v_1>t\}} \nabla v_1\cdot \nabla \varphi  &=& 2\int_{\partial^*\{v_1>t\}} |\nabla v_1| \varphi + 2\int_{\{v_1>t\}} \Delta v_1 \varphi\\ \label{eq int by parts}
&\geq& \int_{\{v_1>t\}} G'(v_1) \varphi
\end{eqnarray}
holds. Thus, taking a sequence $t_k \downarrow 0$ such that \eqref{eq int by parts} holds, we deduce that
\begin{eqnarray}\notag
-2\int_{B_r} \nabla v_1\cdot \nabla \varphi = - 2\int_{\{v_1>0\}} \nabla v_1\cdot \nabla \varphi \geq \int_{\{v_1> 0\}} G'(v_1) \varphi = \int_{B_r} G'(v_1) \varphi,
\end{eqnarray}
where we have used that $\nabla v_1 =0$ $\Lcal^{n+1}$-a.e. in $\{v_1=0\}$, and $G'(0)=0$.

\end{proof}

\begin{proof}[Proof of Lemma \ref{l:reflection}]
First of all, observe that $\tilde v$ is Lipschitz in light of Theorem \ref{thm:holder reg for crit points}. Furthermore, Lemma \ref{lemma restriction subsol} {may be applied with $D= B^\pm$ and the associated component functions $v_1$ for $B^+$ and $v_2$ for $B_-$. This produces two non-negative Radon measures $\mu_1$ and $\mu_2$ supported in $\{v=0\} \cap B_{\rho_0}(x_0)$ so that \eqref{eq subsolution} holds for $v_1$ and $v_2$ respectively. In particular,}
\begin{equation}\label{e: reflection 2}
\Delta \tilde v = \frac{1}{2}\tilde H(\tilde v) +\mu_1-\mu_2 \mbox{  in  $B_{\rho_0}(x_0)$},
\end{equation}
in the sense of distributions. So, showing \eqref{e: reflection} amounts to proving $\mu_1 =\mu_2$ which, in virtue of the Lebesgue-Besicovitch differentiation theorem (see, e.g., \cite[Theorem 5.8]{maggi2012sets}) is equivalent to showing
\begin{eqnarray}\label{e: radonder1}
&&\lim_{r\to 0^+} \frac{\mu_1(B_r(y))}{\mu_2(B_r(y))}=1 \mbox{ for $y \in$ supp$(\mu_2)$\,, }\\\label{e: radonder2}
&&\lim_{r\to 0^+} \frac{\mu_2(B_r(y))}{\mu_1(B_r(y))}=1 \mbox{ for $y \in$ supp$(\mu_1)$\,. }
\end{eqnarray}

We will show \eqref{e: radonder1} since the argument for \eqref{e: radonder2} is completely analogous. Let $y \in$ supp$(\mu_2)$ and consider a sequence $\{r_k\}$ with $r_k\to 0^+$ as $k\to \infty$. We will show that, up to taking a further subsequence (which we won't relabel), we have that
\begin{equation}\label{eq convergence of density}
\lim_{k\to \infty} \frac{\mu_1(B_{r_k}(y))}{\mu_2(B_{r_k}(y))}=1\,.
\end{equation}
Since the sequence $\{r_k\}$ is arbitrary, the desired conclusion follows immediately. With this goal in mind, recalling the $L^2$ height function $H_y(r)$ of $v$ centered at $y$ as introduced in \eqref{e:D-H-def}, {we proceed as follows. For $i=1,2$ and $y$ fixed as above, consider the rescaled functions $v_{i,r}(x) = \frac{v_i(y+rx)}{H_y(r)^{1/2}}$, and the rescaled measures $\mu_{i,r}$ given by $\mu_{i,r}(E) = \frac{\mu_i(rE+y)}{H_y(r)^{1/2}r^{n-1}}$ for any Borel set $E$.} Here, we take $r \in (0,\rho_0 - |y-x_0|)$. Clearly we may then rewrite \eqref{eq convergence of density} as
\begin{equation}\label{eq convergence rescaled}
\lim_{k\to \infty} \frac{\mu_{1,r_k}(B_{1}(y))}{\mu_{2,r_k}(B_{1}(y))}=1
\end{equation}
In addition, recall that by analogous reasoning to that in the proof of Lemma \ref{lemma strong convergence blow-ups}, the rescalings satisfy
\begin{equation}\label{eq rescaled equation}
\Delta v_{i,r} = \frac{r^2}{2H_y(r)^{1/2}}G'(v_{i,r}H_y(r)^{1/2})+\mu_{i,r}
\end{equation}
in the sense of distributions for $i=1,2$, together with the estimate
\begin{equation}\label{eq remainder potential 2}
\left|\frac{r^2}{2H_y(r)^{1/2}}G'(v_{i,r}H_y(r)^{1/2})\right|\leq Cr^2\,.
\end{equation}
On the other hand, since $v_1$ and $v_2$ have disjoint supports, we have that for $r$ small enough
\begin{equation}\label{eq Dirbound rescalings}
\int_{B_2}|\nabla v_{i,r}|^2 \leq  \int_{B_2}|\nabla v_{y,r}|^2 \leq  C\, ,
\end{equation}
where $v_{y,r}(x) = \frac{v(y+rx)}{H_y(r)^{1/2}}$ and where we have used \eqref{eq doubling shell take 2} and the almost monotonicity of the frequency function proved in Lemma \ref{lemma almgren}. We now proceed as in the proof of Lemma \ref{lemma compactness cp} to conclude the weak convergence (up to subsequence) of $\mu_{i,r_k}$ and $v_{i,r_k}$. More precisely, let $\varphi \in C_c^\infty(B_2)$ with $\varphi \geq \mathbf{1}_{B_{\frac{3}{2}}}$, testing \eqref{eq rescaled equation} and combining it with \eqref{eq remainder potential 2} and \eqref{eq Dirbound rescalings}, we deduce
\begin{equation*}
\mu_{i,r}(B_\frac{3}{2}) \leq Cr^2 + \int_{B_2}|\nabla \varphi \cdot \nabla v_{i,r}|\leq Cr^2 + \Big(\int_{B_2}|\nabla v_{i,r}|^2\Big)^\frac{1}{2} \leq C,
\end{equation*}
for $i=1,2$ and for $r$ small enough. So, up to extracting a subsequence of $\{r_k\}$, there exist $\tilde \mu_i$ and $\tilde v_i$ such that $\mu_{i,r_k}\overset{\ast}{\rightharpoonup} \tilde \mu_{i}$  as Radon measures in $B_\frac{3}{2}$ and that $v_{i,r_k} \rightharpoonup \tilde v_i$ weakly in $W^{1,2}(B_\frac{3}{2})$ {and locally uniformly} as $k\to \infty$ for $i=1,2$. However, since {$N_{y}(0^+) = 1$ for every $y\in\{v=0\}\cap B_{\rho_0}(x_0)$}, {the local uniform convergence and} Proposition \ref{p:tangents} implies that (up to taking a new subsequence) $\tilde v_1(x) = L (x\cdot e)_+$ and $\tilde v_2(x) = L (x\cdot e)_-$ for some $L>0$ and some $e \in \mathbb{S}^n$. Furthermore, by weak convergence, we have that
\begin{equation}\label{eq limiting equation}
\Delta \tilde v_{i} = \tilde \mu_i
\end{equation}
holds in the sense of distributions for $i=1,2$. From here, since $\tilde v_1(x)-\tilde v_2(x) = \frac{1}{\sqrt{|\omega_{n+1}|}} \, (x\cdot e)$, we deduce that $\tilde \mu_1= \tilde \mu_2$. In addition, by the particular form of $\tilde v_i$, we deduce from \eqref{eq limiting equation}  that $\tilde \mu_i =  \frac{1}{\sqrt{|\omega_{n+1}|}} \, \h \mres \{x\in B_{\frac{3}{2}} : x\cdot e=0\}$ and, thus, $\tilde\mu_i(\partial B_1)=0$. From here \eqref{eq convergence rescaled} follows immediately. 
\end{proof}

\begin{proof}[Proof of Lemma \ref{l:harmonic-polyn}]
We will demonstrate that we may identify the set of connected components of $\{\bar v >0 \}$ with the set of vertices for a bipartite graph, when $n \geq 2$. Once we show this, we may conclude as follows. Recall that every bipartite graph is two-colorable. Let $\Fcal_1, \Fcal_2$ denote the two mutually disjoint subsets of connected components of $\{\bar v >0 \}$, each corresponding to the set of vertices of the same color in the graph. Define
\[
h=\begin{cases}
\bar v & \text{on every connected component in $\Fcal_1$,} \\
-\bar v & \text{on every connected component in $\Fcal_2$,} \\
0 & \text{on $\{\bar v=0\}$.}
\end{cases}
\]
Observe that by construction, $\bar v = |h|$. Thus, we just need to verify that $h$ is a harmonic polynomial. To see this, first of all notice that the harmonicity of $h$ follows immediately from Lemma \ref{l:reflection}. Indeed, this is due to the fact that the hypotheses of the lemma guarantee that $\{\bar v=0\}\cap B_1\setminus \{0\}\subset \{y: N_{\bar v,y}(0^+) = 1\}$, combined with the bipartite graph property of the connected components of $\{\bar v > 0\}$, and the fact that $\{0\}$ forms a capacity zero subset of $B_1$. Moreover, note that the function $\tilde H$ given by \eqref{eq odd reflection} associated to the tangent function $\bar v$ vanishes identically (see Lemma \ref{lemma strong convergence blow-ups}). To see that $h$ is a polynomial, we simply exploit the radial homogeneity of $\bar v$, together with the well-known classification of radially homogeneous harmonic functions {(see, for instance, \cite[Chater III]{Stein-SI})}.

It now remains to prove the aforementioned claim that the connected components of $\{\bar v >0 \}$ identify with the set of vertices of a bipartite graph. Note that this claim crucially requires $n\geq 2$, and is false when $n=1$. {First, note our assumption that $N_{x}(0^+)=1$ for all $x\in \{\bar v = 0\}\cap \partial B_1$ combined with the radial homogeneity of $v$ implies that $N_{x}(0^+)=1$ for all $x\in \{\bar v = 0\}\setminus \{0\}$. Then we can apply Proposition \ref{lemma density components}, Lemma \ref{l:harmonic-polyn}, and the classification of frequency one blowups to conclude that $\{\bar v=0\}\setminus \{0\}$ locally coincides with the zero set of a harmonic function with non-vanishing gradient on its nodal set. As a consequence, the Implicit Function Theorem yields that $\{\bar v=0\}\cap \partial B_1$ is a smooth (even analytic), embedded $(n-1)$-manifold.} The coloring can be done now by exhaustion as follows. Suppose, without loss of generality, that the two colors are red and blue. Consider a connected component $U_0$ of $\partial B_1 \cap \{\bar v > 0\}$, and assign this the color red. We assign each connected component of $\{\bar v>0\}$ neighboring $U_0$ the color blue, {and call these $\{U_1^i\}_i$. We claim that
\begin{equation}\label{eq:no common boundary with U and V}
    \mbox{if $i\neq j$, then $\partial^{\partial B_1} U_1^i \cap \partial^{\partial B_1} U_1^j = \varnothing$.}
\end{equation}
Assume for contradiction that \eqref{eq:no common boundary with U and V} did not hold for some $U_1^i$ and $U_1^j$. Then by the smoothness of $\{\bar v = 0\}$, their common boundary is also smooth, and so we can choose a smooth connected component of $\partial^{\partial B_1} U_1^i \cap \partial^{\partial B_1} U_1^j$ and call it $M$. By the Jordan-Brouwer separation theorem on $\partial B_1$ (which follows for $n\geq 2$ from e.g.~ the statement on $\mathbb{R}^{n}$ \cite[pg 89]{GP} and a stereographic projection, but does not hold on $\mathbb{S}^1$), denoting by $A$ and $B$ the open sets on $\partial B_1$ with $\partial^{\partial B_1} A \cap \partial^{\partial B_1} B=M$, we have, up to relabelling, $U_1^i \subset A$ and $U_1^j \subset B$. Also, since $U_0$ is open and connected and does not intersect $M$, we must have either $U_0 \subset A$ or $U_0 \subset B$. But either case leads to a contradiction: if $U_0 \subset A$, it cannot border $U_1^j$ since $U_1^j\subset B$ and $M \cap \partial U_0=\varnothing$ by the smoothness of $\{\bar v=0\}$, with a similar contradiction if $U_0 \subset B$. Next we color in red every open connected component of $\{\bar v >0\}$ bordering some $U_1^i$; this is well-defined, since by \eqref{eq:no common boundary with U and V} no blue sets share a common boundary.} We can now proceed inductively in this manner, exhausting all of the connected components of $\{\bar v > 0\}\cap \partial B_1$ (of which there are finitely many according to the smooth embeddedness of $\{\bar v=0\} \cap \pa B_1$).
\end{proof}

\begin{proof}[Proof of Proposition \ref{p: freq gap}]




Fix $x_0\in \{v=0\}$ with $N_{v,x_0}(0^+)>1$ and consider any tangent function $\bar v$ at $x_0$. First of all, recall from Lemma \ref{lemma prop blowups}, $\bar v$ is radially homogeneous of degree $\alpha := N_{v,x_0}(0^+) = N_{\bar v,0}(0^+) > 1$. 

We proceed to argue by induction on $n$, for solutions of \eqref{eq modified iv}-\eqref{eq modified ineq}, which in particular includes all tangent functions $\bar v$, in light of Lemma \ref{lemma strong convergence blow-ups}. Let us begin with the base case $n=1$. In this case, the classification of Lemma \ref{lemma:planar tangent classification} immediately implies that the alternative (1) holds and $N_{v,x_0}(0^+) \geq \tfrac{3}{2}$. Note that in this case, the alternative (2) is impossible, since there are exactly two connected components of $\{\bar v>0\}\cap B_1$ if and only if $\bar v(x) = \frac{1}{\sqrt{\pi}} |x\cdot e|$ for some $e\in \mathbb{S}^1$, in which case $N_{v, x_0}(0^+) = 1$.

Now fix $n \geq 2$ and suppose that the conclusions of the proposition hold (including the lower frequency bound) in every dimension $m+1 \leq n$, in place of $n+1$. Let $x_0$, $v$ be as in the statement of the proposition. {In particular, we suppose that $N_{v,x_0}(0^+) > 1$.} There are two possibilities. Either
\begin{itemize}
\item[(a)] there exists $e_0\in \{\bar v=0\}\cap \partial B_1$ with $N_{\bar v, e_0}(0^+) > 1$, or
\item[(b)] for every $e\in \{\bar v=0\}\cap \partial B_1$, $N_{\bar v, e}(0^+) = 1$.
\end{itemize}
In case (a), by Lemma \ref{lemma prop blowups}, any tangent function $\bar w$ of $\bar v$ at $e_0$ is translation-invariant in the direction $e_0$ and thus identifies with a solution of \eqref{eq modified iv}-\eqref{eq modified ineq} with $G=0$ that is a function of $n$ real variables. Since we additionally have $N_{\bar v, 0}(0^+) \geq N_{\bar v, e_0}(0^+) = N_{\bar w, 0}(0^+)$, $\bar w$ satisfies the hypotheses of the proposition at the origin (where any tangent function of it will be itself). The inductive hypothesis therefore allows us to conclude in this case. In case (b), we simply apply Lemma \ref{l:harmonic-polyn}, which implies that $N_{\bar v,0}(0^+) = N_{h,0}(0^+) \geq 2$.

When $n=2$ and (2) holds, notice that the alternative (b) from the above dichotomy must hold. Indeed, if $N_{\bar v, e_0}(0^+)>1$ for some $e_0\in \{\bar v=0\}\cap \partial B_1$, then the tangent function $\bar w$ at $e_0$ as above will be a function of $2$ real variables. However, we additionally have exactly two connected components of $\{\bar v>0\}\cap B_1(e_0)$, which in turn implies that $\{\bar w > 0\}$ has exactly two connected components. This, combined with the fact that $N_{\bar w, 0}(0^+) = N_{\bar v, e_0}(0^+) > 1$ is in contradiction with the classification given by Proposition \ref{p:tangents}.\end{proof}

\section{Minimizers and their regularity: Proofs of Theorems \ref{thm:main regularity theorem} and \ref{thm:main existence theorem}}\label{s:proof-main-theorems}

\subsection{Competition class and cup competitors}\label{s:prelim}
Here we derive the criticality conditions \eqref{eq innervariation}-\eqref{eq:differential inequality} for minimizers to our variational problem over a larger class of ``admissible" functions in $W^{1,2}_\loc(\Omega;[0,1])$ which, together with their ``cup competitors" (see Definition \ref{d:cup-comp} below), satisfy the spanning condition in a suitable sense. This class will in particular contain any minimizer of \eqref{new model intro} and \eqref{new model volume constrained}.


Let us recall the generalized spanning condition from \cite{maggi2023plateau,maggi2023hierarchy}. We begin with the following definition of spanning for closed sets from \cite[Definition 3]{DGM17}, which is a slight generalization of the one from \cite[pg 359]{HP16} and has stimulated much recent progress on the Plateau problem; see e.g.~ \cite{DGM17,DPDRG,HP16b,DLDRG19,HP17, FangKo,DR,DPDRG2}. 

\begin{definition}[Homotopic spanning for closed sets]\label{def:span closed}
A {\bf spanning class} $\mathcal{C}$ is a family of smooth embeddings of $\mathbb{S}^1$ into $\Omega$ which is closed by homotopy relative to $\Om$, and a relatively closed set $K\subset \Omega$ is said to be {\bf $\mathcal{C}$-spanning} $\wire$ if $K \cap \gamma\neq \varnothing$ for every $\gamma\in \mathcal{C}$.
\end{definition}

To generalize this, we recall the notion of $\mathcal{C}$-spanning introduced by the first two authors and F. Maggi in \cite[Definition B]{maggi2023plateau} and applied to the Allen-Cahn energy in \cite{maggi2023hierarchy}. It uses the notion of measure theoretic connectedness introduced in \cite{CCDPMgauss,CCDPMsteiner}. A Borel set $K$ {\bf essentially disconnects} another Borel set $G$ if there exist Borel $G_1,G_2 \subset G$ such that
\begin{equation}\label{def:ess disconn}
\mathcal{L}^{n+1}\big(G \Delta (G_1 \cup G_2)\big)=0,\quad \mathcal{L}^{n+1}(G_1)\mathcal{L}^{n+1}(G_2)>0, \quad \h\big((G^\one \cap \pa^e G_1 \cap \pa^e G_2 )\setminus K \big) =0\,.
\end{equation}
Here, for any Borel $B \subset \mathbb{R}^{n+1}$, $B^{\scriptstyle{(t)}}$ is the set of points of Lebesgue density $t\in [0,1]$ and $\pa^e B$ is the essential boundary of $B$, or $\mathbb{R}^{n+1} \setminus (B^\zero \cup B^\one)$. We also denote by $B_1^n$ the ball of radius one in $\mathbb{R}^n$.

\begin{definition}[Homotopic spanning for Borel sets]\label{def:span borel}
Given a spanning class $\mathcal{C}$, the associated {\bf tubular spanning class} $\mathcal{T}(\mathcal{C})$ is the family of all triples $(\gamma, \Psi , T)$ where $\gamma \in \mathcal{C}$, 
\begin{equation}\notag
\mbox{$\Psi:\mathbb{S}^1 \times  \overline{B_1^n}\to \Omega$ is a diffeomorphism with $\restr{\Psi}{\mathbb{S}^1 \times \{0\}}=\gamma$},
\end{equation}
and $T=\Psi(\mathbb{S}^1 \times B_1^n)$. A Borel set $K\subset \Omega$ is {\bf $\mathcal{C}$-spanning} $\wire$ if for every $(\gamma,\Psi, T)\in \mathcal{T}(\mathcal{C})$, $\mathcal{H}^1$-a.e. $s\in \mathbb{S}^1$ has the following property: for $\mathcal{H}^n$-a.e. $x\in \Psi(\{s\}\times B_1^n)$, there exists a partition $\{T_1,T_2\}$ of $T$ with $x\in\partial^e T_1 \cap \partial^e T_2$ and such that $K \cup  \Psi(\{s\}\times B_1^n)$ essentially disconnects $T$ into $\{T_1,T_2\}$.
\end{definition}

\begin{remark}[Consistency of Definitions \ref{def:span closed}-\ref{def:span borel}]\label{remark:consistency of spanning defs}
The previous two definitions are consistent because for any relatively closed $K\subset \Omega$, it is $\mathcal{C}$-spanning according to the former if and only if it is $\mathcal{C}$-spanning according to the latter \cite[Theorem A.1]{maggi2023plateau}.
\end{remark}

The $\h$-stability of the class of $\mathcal{C}$-spanning sets \cite[Page 8]{maggi2023plateau} is the key property that allows for an acceptable definition of $\mathcal{C}$-spanning for the $1$-level set of $u\in W^{1,2}(\Om;[0,1])$. Recall the definition of precise representative $u^*$ from \eqref{eq: precise rep intro}.

\begin{definition}[Generalization of \eqref{eq:HP spanning}]\label{def:gen of spanning}
For $u\in W^{1,2}_\loc(\Om;[0,1])$, the reformulation of \eqref{eq:HP spanning} is 
\begin{equation}\label{eq:c spanning for functions}
\{u^* \geq t\} \mbox{ is $\mathcal{C}$-spanning according to Definition \ref{def:span borel} for all $t\in (1/2,1)$}\,.
\end{equation}
\end{definition}

\begin{remark}[Consistency of Definition \ref{def:gen of spanning} and \eqref{eq:HP spanning}]\label{r:consistency-spanning-cts}
    Note that when $u$ is additionally continuous, the generalized spanning condition \eqref{eq:c spanning for functions} is equivalent to ``$\{u=1\}\cap \gamma \neq \varnothing$ for all $\gamma \in \mathcal{C}$" (see Lemma \ref{lemma:equiv for continuous functions}), therefore it indeed generalizes \eqref{eq:HP spanning}. Also, \eqref{eq:c spanning for functions} is preserved under uniform Dirichlet energy bounds, proven in \cite{maggi2023hierarchy} and recalled in Theorem \ref{thm:prelim compactness thm}. Lastly, since supersets of $\mathcal{C}$-spanning sets are $\mathcal{C}$-spanning, choosing some other lower bound than $1/2$ does not change whether the condition holds for some $u$, as it is only those super-level sets where $u$ takes values arbitrarily close to $1$ that matter.
\end{remark}

In order to define our admissible class of functions, we wish to make sense of cup competitors for functions $u\in W^{1,2}_\loc(\Om;[0,1])$ satisfying \eqref{eq:c spanning for functions} (cf. \cite[Definition 1]{DGM17} for the analogue for closed sets). To do this, we must first introduce connected components in a measure-theoretic sense. 

\begin{lemma}[Essentially connected components of $\{v>0\}$]\label{lemma connected 2}
    Let $u\in W^{1,2}_{\rm loc}(\Omega;[0,1])$ satisfy \eqref{eq:c spanning for functions} and let $v=1-u$. If $U \subset\!\subset \Omega$ is open and has finite perimeter, then either $|\{v^*>0 \}\cap U|=0$ or there exist $T \subset(0,1)$ with $|T|=1$ and partition $\{C_i\}_i$ of $U \cap \{v>0\}$ (up to Lebesgue null sets) such that for any $t_j\searrow 0$ with $\{t_j\}\subset T$,
    \begin{equation}\notag
        \mbox{$|C_i \cap C_j|=0$ for $i\neq j$,\quad $|C_i|>0$, \quad $C_i$ is not essentially disconnected by $\{v^*=0\}$}\,,
    \end{equation}    
and each $C_i$ is an $L^1$-limit of sets of finite perimeter $\{F_j^i\}_{j\geq j_i}$ with $|F_{j}^i \setminus F_{j+1}^i|=0 $, $F_j^i \subset \{v^*>t_j\}$ and $\partial^* F_j^i \cap U \subset \{v^* = t_j\}$ such that $\{v^* \leq t_j\}$ does not essentially disconnect $F_j^i$ for all $j\geq j_i$. Furthermore, if $t\in T$ and $F_t\subset \{v^* > t\}$ is not essentially disconnected by $\{v^* = t\}$, then there is unique $C_i$ such that
\begin{equation}\label{eq:ci ft dichotomy}
|F_t \setminus C_i|=0\,;
\end{equation}
up to null sets and relabelling, $\{C_i\}_i$ is the unique partition of $U \cap \{ v>0\}$ satisfying these properties. We refer to each $C_i$ as an \textbf{essentially connected component} of $\{v^*>0\}$ in $U$. {Lastly, $\{u^*=1\}$ is $\mathcal{C}$-spanning if and only if $\{u^*\geq t\}$ is $\mathcal{C}$-spanning for all $t\in (1/2,1)$.}
\end{lemma}

\begin{proof}
Suppose that
    \begin{equation}\label{eq:there's some positive v}
        |\{v^*>0\} \cap U|>0\,.
    \end{equation}
We will now proceed to construct the desired partition $\{C_i\}_i$.
\medskip
    
\noindent{\it Step 1 (identification of $T$ and preliminary partition)}: 
To identify $T$, by standard facts about $W^{1,2}_\loc$ functions \cite[Equation 3.2]{maggi2023hierarchy}, there is a full measure subset $T\subset (0,1)$ such that if $t\in T$, then $E_t:=\{v^* \leq t\}$ are sets of locally finite perimeter in $\Omega$ with
\begin{equation}\label{eq:wheres pa Ej}
    \mbox{$\partial^* E_t \cap \Omega = \{v^* = t\}$ up to $\mathcal{H}^n$-null sets}\,.    \end{equation}
We now consider any $t_j\searrow 0$ with $\{t_j\}\subset T$.
    
Since $u=1-v$ satisfies \eqref{eq:c spanning for functions}, \cite[Theorem 2.1]{maggi2023plateau} implies that for each $j$ there exists a Caccioppoli partition $\{ U_m^j\}_m$ of $U$ (in the terminology of \cite{maggi2023plateau}, an ``essential partition") into sets of finite perimeter, such that, setting $K_j = \partial^* E_j \cap \Omega$,
    \begin{align}\label{eq:mon of Uij}
        &|U_1^j| \geq |U_2^j| \geq \cdots \geq |U_m^j| \geq |U_{m+1}^j| \geq \cdots >0\\ \label{eq:containment of reduced boundaries}
        &\mathcal{H}^n(\partial^* U_m^j \cap U \setminus K_j) = 0 \\ \label{eq:does not disconnect}
        &\mbox{$K_j$ does not essentially disconnect $U_m^j$ for each $m$}\,.
    \end{align}
Note that for each $U_m^j$, \eqref{eq:does not disconnect} and $|\{v^*=t_j\}|=0$ imply that
\begin{equation}\label{eq:cant split}
    \mbox{either \quad $(U_m^j)^\one \subset E_j^\one=\{v^*<t_j\}^\one$ \quad or \quad $(U_m^j)^\one \subset E_j^\zero=\{v^*>t_j\}^\one$\,,}    
\end{equation}
otherwise we could non-trivially partition $U_m^j$ into $U_m^j \cap E_j$ and $U_m^j \setminus E_j$, contradicting \eqref{eq:does not disconnect}. Using \eqref{eq:cant split}, we thus divide the sets $U_m^j$ into sets $B_m^j$ and $C_m^j$ which respectively satisfy
\begin{equation}\label{eq:wheres Cij}
(B_m^j)^\one \subset \{v^*<t_j\}^\one \qquad \text{and} \qquad (C_m^j)^\one \subset \{v^*>t_j\}^\one\,.    
\end{equation}
{Note that by standard density properties of sets of finite perimeter and their reduced boundaries, the containments $(B_m^j)^\one \subset \{v^*<t_j\}^\one$ and $(C_m^j)^\one \subset \{v^*>t_j\}^\one$ combined with \eqref{eq:containment of reduced boundaries} imply that
\begin{equation}\label{eq:wheres pa cij bij}
    U \setminus \cup_m (C_m^j)^\one \ \text{is $\mathcal{H}^n$-contained in} \ \{ v^* \leq t_j\}\,,
\end{equation}
meaning that the containment holds up to a $\Hcal^n$-null set. L}et us order the sets $C_m^j$ so that
\begin{equation}
\label{eq:mon of Cij}
        |C_1^j| \geq |C_2^j| \geq \cdots \geq |C_m^j| \geq |C_{m+1}^j| \geq \cdots > 0\,.
\end{equation}
We now claim that 
\begin{equation}\label{eq:small guy does not disconnect}
    \mbox{$\{v^* \leq t_j \}$ does not essentially disconnect $C_m^j$ for each $m$}.
\end{equation}
Indeed, assuming for a contradiction that $\{v^* \leq t_j\}$ essentially disconnected some $C_m^j$, we would have $A,B \subset C_m^j$ such that $|A||B|>0$ and
\begin{equation}\label{eq:wheres the reduced}
\mbox{$\partial^* A \cap \partial^* B \cap (C_m^j)^\one$ is $\mathcal{H}^n$-contained in $\{v^* \leq t_j\}$}\,.
\end{equation}
Since $A,B \subset C_m^j$, we also have
\begin{equation}\label{eq:where's the reduced 2}
\partial^* A \cup \partial^* B \subset (C_m^j)^\one \cup \partial^* C_m^j \qquad \mbox{up to $\mathcal{H}^n$-null sets};
\end{equation}
(see e.g. \cite[Equation (16.7)]{maggi2012sets}). Now by \eqref{eq:wheres pa Ej}, \eqref{eq:containment of reduced boundaries}, and \eqref{eq:wheres Cij}, we have
$$
U \cap \big((C_m^j)^\one \cup \partial^* C_m^j \big)\subset U\cap \big( \{v>t_j \}^\one \cup \{ v^* = t_j\}\big)\,.
$$
In addition, $\{v^*>t_j\}^\one \subset \{v^* \geq t_j\}$ up to $\mathcal{H}^n$-null sets, see \cite[Section 3.1]{maggi2023hierarchy}. Inserting these previous two containments into \eqref{eq:where's the reduced 2}, we find that
\begin{equation}
    \label{eq:where's the reduced 3}
U \cap (\partial^* A \cup \partial^* B) \subset \{v^* \geq t_j\} \qquad \mbox{up to $\mathcal{H}^n$-null sets}.
\end{equation}
Combining this with \eqref{eq:wheres the reduced}, we conclude that 
\begin{equation}\label{eq:where's the reduced 4}
\partial^* A \cap \partial^* B \cap (C_m^j)^\one \subset  \{v^* = t_j\}\subset K_j\qquad \mbox{up to $\mathcal{H}^n$-null sets}\,,
\end{equation}
which contradicts \eqref{eq:does not disconnect}. Therefore \eqref{eq:small guy does not disconnect} holds. 

We additionally claim that if $j<k$, then for any $C_m^j$ and $C_{m'}^k$, either
\begin{equation}\label{eq:dichotomy}
    \mbox{either\quad $|C_m^j \cap C_{m'}^k| = |C_m^j|$ \quad or \quad $|C_m^j \cap C_{m'}^k | =0$}\,.
\end{equation}
If this were not the case, then we would have $j<k$ and $C_m^j$, $C_{m'}^k$ such that $0<|C_m^j \cap C_{m'}^k| < |C_m^j|$. Let us consider the nontrivial partition $A=C_m^j \cap C_{m'}^k$, $B=C_m^j \setminus C_{m'}^k$ of $C_m^j$. By standard facts regarding reduced boundaries (see e.g. \cite[Eq. (16.7)-(16.8)]{maggi2012sets}),
\begin{equation}\label{eq:bad disconnect}
    \partial^* A \cap \partial^* B \cap (C_m^j)^\one \subset \partial^* C_{m'}^k \cap (C_m^j)^\one \quad \mbox{up to $\mathcal{H}^n$-null sets}.
\end{equation}
But $U \cap \partial^* C_{m'}^k \subset \{v^* = t_k\}$ by \eqref{eq:wheres pa Ej} and \eqref{eq:containment of reduced boundaries}. Thus \eqref{eq:bad disconnect} implies that the set $\{v^* = t_k\} \subset \{v^* \leq t_j\}$ essentially disconnects $C_m^j$, contradicting \eqref{eq:small guy does not disconnect}. So \eqref{eq:dichotomy} indeed holds. 

\medskip

\noindent{\it Step 2 (construction of $C_i$ and $F_j^i$)}: We now use \eqref{eq:dichotomy} take the limit in $j$. First for $j=2$, we note that \eqref{eq:dichotomy} implies that there is a unique $C_m^2$ such that $|C_1^1 \cap C_{m}^2| = |C_1^1|$. Let us call this $C_m^2=: F_2^1$. Next, again using \eqref{eq:dichotomy}, we choose $F^1_3$ to be the unique $C_\ell^3$ containing $F^1_2$ up to Lebesgue null sets. Continuing on as such for each $j$, we obtain an increasing sequence of sets $C_1^1=: F_1^1 \subset F^1_2 \subset \cdots \subset F^1_j\subset F_{j+1}^1 \subset \cdots$. We define
\begin{equation}\notag
    C_1 = \bigcup_j F_j^1\,.
\end{equation}

Next, in order to define $F_j^2$ and $C_2$, we first claim that for each $j$ and $m$, either 
\begin{equation}\label{eq:second dichotomy}
|C_m^j \cap C_1|= |C_m^j|\quad\mbox{or}\quad| C_m^j \cap C_1|=0\,.
\end{equation}
To see this, for each $j$, \eqref{eq:dichotomy} implies that $|C_m^j \cap F^1_k|$ is either $|C_m^j|$ or $0$ for each $k>j$. Moreover,  if $|C_m^j \cap F^1_{k'}|=|C_m^j|$ for a single $k'>j$, then by the nestedness of the sets $F^1_{k}$ we have $|C_m^j \cap F^1_{k}| = |C_m^j|$ for all $k\geq k'$ and thus $|C_m^j \cap C_1| = |C_m^j|$. So if $|C_m^j \cap F^1_k| =|C_m^j|$ for a single $k$, the first condition in \eqref{eq:second dichotomy} holds, and if $|C_m^j \cap F^1_k| = 0$ for all $k$, then $|C_m^j \cap C_1|=0$ since $|F^1_k\Delta C_1|\to 0$. So \eqref{eq:second dichotomy} holds, and this allows us to repeat the process all over again. We choose the smallest $j_2\geq 2$ such that for some $m$, $|C_m^{j_2} \cap C_1|=0$. We choose $m$ to be minimal with this property and call the resulting set $F_{j_2}^2$. By \eqref{eq:dichotomy}, there is unique $C_m^{j_2+1}$ containing $F_{j_2}^2$ (up to Lebesgue null sets); we call it $F_{j_2+1}^2$. Iterating this procedure in $j$ yields an increasing sequence of sets $\{F_j^2\}_{j\geq j_2}$. Since the sequence is increasing, there is a limiting $C_2$ with $|F_j^2 \Delta C_2|\to 0$ and that satisfies the analogue of \eqref{eq:second dichotomy}. Furthermore, \eqref{eq:second dichotomy} implies that $|C_2 \cap C_1|=0$. Indeed, if this was not the case, then $|C_2 \cap C_1|>0$, in which case $|F_j^2 \cap C_1|>0$ for some $j\geq j_2$. But then by \eqref{eq:dichotomy}, $|F_j^2 \cap C_1|=|F_j^2|$, and since $F_j^2 \supset F_{j_2}^2$, this contradicts our choice of $F_{j_2}^2$.

Continuing on in this fashion, to contruct $C_i$ and $\{F_{j}^i\}_{j\geq j_i}$, we start by choosing the smallest $j_i>j_{i-1}$ such that some $|C_m^{j_i}\cap (C_1 \cup \cdots \cup C_{i-1})|=0$, where $m$ is minimal with this property, and let $F_{j_i}^i:= C^{j_i}_m$. If no such $j_i$ exists, then the process terminates. If there is such a $j_i$, then the same procedure as above yields analogous increasing sets $\{F_{j}^i\}_{j\geq j_i}$ and limiting $C_i$. In summary, we have obtained $\{j_i\}_{i}$, families $\{ F_j^{i}\}_{j\geq j_i}$ for each $i$, and sets $\{C_i\}_{i}$ satisfying
\begin{align}\label{eq:increasing of ji's}
    &1=j_1 < j_2 < \cdots < j_i < \cdots\,, \\ \label{eq:minimality of m}
    &\mbox{$F_{j_i}^i=C_m^{j_i}$, for the minimal number $m$ such that $|C_m^{j_i} \cap (\cup_{i'<i}C_{i'})|=0$}\,,\\ \label{eq:nestedness of fs}
    &F_{j_i}^i\subset F_{j_{i}+1}^i\subset \cdots \subset F_{j_i+k}^i\subset \cdots  \quad \mbox{for each $i$}\,, \\ \label{eq:limiting to ci}
    &|F_{j}^i \Delta C_i| \to 0 \quad \mbox{as $j \to\infty$, for each $i$}\,,\\ \label{eq:disjointness of C_i}
    &|C_i \cap C_{i'}|=0 \quad \mbox{if}\,i\neq i'\,,\qquad \mbox{and} \\ \label{eq:Ci cim dichotomy}
    &|C_i \cap C_m^j| = 0 \mbox{ or }|C_i \cap C_m^j| = |C_m^j| \qquad \forall i, m, j\,.
\end{align}

\medskip

\noindent{\it Step 3 ($\{C_i\}$ is a Lebesgue partition of $\{v^*>0\}\cap U$)}:
By \eqref{eq:disjointness of C_i}, we know that $\{C_i\}$ are disjoint up to Lebesgue null sets, so it remains to show that they exhaust $U \cap \{v^*>0\}$. If this were not the case, and $|(U\cap \{v^*>0\} )\setminus \cup_i C_i |>0$, then since $t_j \searrow 0$, we would have $|(U\cap \{v^*>t_j\} )\setminus \cup_i C_i|>0$ for some $j$. Since $\{C_m^j\}_m$ partition $U \cap \{v^*>t_j\}$ for each $j$, this implies that for some $m$,
\begin{equation}\label{eq:contra with the cis}
    |C_m^j \setminus \cup_i C_i|=:\varepsilon>0\,.
\end{equation}

In particular, there must be infinitely many $C_i$'s, because the procedure in the previous step only terminated at stage $i$ if there is no $j' > \max\{j_1,\dots,j_{i-1}\}$ such that $|C_{m'}^{j'} \cap (C_1\cup \cdots \cup C_{i-1})|=0$ for some $m'$. Now this yields a contradiction for the following reason. 
Due to the fact that $|U| < \infty$ (since it has finite perimeter) and \eqref{eq:increasing of ji's}, there exists $I\in \N$ with $j_I > j$ such that $|C_I|<\varepsilon$. In turn, according to \eqref{eq:nestedness of fs} and \eqref{eq:limiting to ci}, this implies that $|F_{j_I}^I|<\varepsilon$. Now by \eqref{eq:minimality of m}, $F_{j_I}^I=C_{m_I}^{j_I}$, where $m_I$ is minimal such that $|C_{m_I}^{j_I} \cap (\cup_{i<I}C_i )|=0$. Also, since $j_I > j$ and thus $t_{j_I}<t_j$, we have $U \cap \{v > t_j\}\subset U\cap \{v>t_{J_I}\}$, with $U \cap \{v > t_j\}$ partitioned by $\{C_{m}^{j}\}_{m}$ and $U \cap \{v > t_j\}$ partitioned by $\{C_{m'}^{j_I}\}_{m'}$. By this containment and the dichotomy \eqref{eq:dichotomy}, there is $C_{m'}^{j_I}$ such that
\begin{equation}\label{eq:will give cim contra}
|C_{m'}^{j_I} \cap C_{m}^{j}| = |C_m^j|\geq \varepsilon\,,   
\end{equation}
so that in particular, $C^{j_I}_{m'} \supset C^j_m$ (up to Lebesgue null sets). Also, by the dichotomy \eqref{eq:Ci cim dichotomy}, we have that either $|C_{m'}^{j_I} \cap \cup_i C_i|=0$ or $|C_{m'}^{j_I} \cap \cup_i C_i|=|C_{m'}^{j_I}|$. By \eqref{eq:contra with the cis} and the fact that $C^{j_I}_{m'} \supset C^j_m$, it must be that $|C_{m'}^{j_I} \cap \cup_i C_i|=0$. But together, this equality and \eqref{eq:will give cim contra} say that $C_{m'}^{j_I}$ is disjoint (up to Lebesgue null sets) from $\cup_i C_i$ and $|C_{m'}^{j_I}|\geq \varepsilon > |F_{j_I}^I| = |C_{m_I}^{j_I}|$. By the ordering \eqref{eq:mon of Cij}, this inequality implies that $m' < m_I$, contradicting the minimality property \eqref{eq:minimality of m} of $m_I$.

\medskip

\noindent{\it Step 4 ($C_i$ is not essentially disconnected by $\{v^*=0\}$)}: 
Recall that we already verified $\{v^*\leq t_j\}$ does not essentially disconnect $C^j_m$ for each $m$ in \eqref{eq:small guy does not disconnect}. Suppose, for a contradiction, that $\{v^*=0\}$ does essentially disconnect $C_i$ for some $i$. Then, in particular, there exist $A,B \subset C_i$ with $|A||B| > 0$ and
\[
    (\partial^* A \cap \partial^* B \cap C_i^{(1)}) \setminus \{v^* = 0\} = 0\,.
\]
In light of \eqref{eq:nestedness of fs} and \eqref{eq:limiting to ci}, for $j$ large enough, if we choose $A_j := A \cap F_j^i$ and $B_j := B\cap F_j^i$, then $A_j,B_j$ give an essential partition of $F_j^i$, which is $C^j_m$ for some $m$, therefore contradicting \eqref{eq:small guy does not disconnect}.

\medskip
\noindent{\it Step 5 (proof of \eqref{eq:ci ft dichotomy})}: Suppose that $t\in T$ and $F_t\subset \{v^* > t\}$ is not essentially disconnected by $\{v^* = t\}$, so that it is part of the ``essential partition" of $U$ by $\{v^* = t\}$ (as in Step 1). Since $\{C_i\}$ form a Lebesgue partition of $\{v^* > 0 \}\cap U$, we must have $|C_i \cap F_t|>0$ for some $i$, where $C_i$ is an increasing limit of sets $\{F_j^i\}_{j\geq j_i}$ belonging to the essential partition of $U$ by $\{v^* = t_j\}$. For all $j$ large enough such that $t_j < t$, since the sets $\{C_m^j\}_m$ defined in Step 1 partition $\{v^* > t_j\} \cap U$ up to a Lebesgue null set, there must be some $C_{m(j)}^j$ such that $|C_{m(j)}^j \cap F_t|>0$. By \eqref{eq:dichotomy}, $|C_{m(j)}^j \cap F_t| = |F_t|$. By \eqref{eq:Ci cim dichotomy}, there is $C_i$ such that $|C_i \cap F_t| = |F_t|$.

\medskip

{\noindent{\it Step 6 (independence of $\{C_i\}$ on choice of $t_j\searrow 0$ and uniqueness)}: Suppose that $\{C_i\}_i$ is one such partition associated to $t_j \searrow 0$, and $\{D_k\}_k$ is another associated to $s_\ell\searrow 0$, with $C_i = \lim_{j\to\infty} F_j^i$ and $D_k = \lim_{\ell\to\infty} E_\ell^k$. To show that they are $\Lcal^{n+1}$-equivalent, it suffices to show that if $|C_i \cap D_k|>0$, then $|C_i| = |C_i \cap D_k| = |D_k|$. Suppose that $|C_i \cap D_k|>0$; we prove that $|C_i \cap D_k| = |C_i|$, with the other equality following similarly. For each $j$, choose $\ell(j)$ such that $s_{\ell(j)} < t_j$. Then $0<|C_i \cap D_k| = \lim_{j\to\infty}| F_j^i \cap E_{\ell(j)}^k|$. By the argument of \eqref{eq:dichotomy} and the fact that $s_{\ell(j)} < t_j$, for each $j$, exactly one of $|F_j^i \cap E_{\ell(j)}^k|=|F_j^i|$ or $|F_j^i \cap E_{\ell(j)}^k|=0$ holds. But since $0<\lim_{j\to\infty}| F_j^i \cap E_{\ell(j)}^k|$, it must be $|F_j^i \cap E_{\ell(j)}^k|=|F_j^i|$ for all sufficiently large $j$. Thus we indeed have $|C_i \cap D_k| = \lim_{j\to\infty} |F_j^i| = |C_i|$. {The same argument shows that $\{C_i\}_i$ is the unique partition satisfying the properties in the lemma, up to relabelling and null sets.}} 

\medskip

{In the last two steps, we show that $\{v^*=0\}=\{u^*=1\}$ is $\mathcal{C}$-spanning if and only if $\{u^*\geq t\}$ is for all $t\in (1/2,1)$. Note that one direction is trivial, namely that $\{u^*\geq t\}$ is $\mathcal{C}$-spanning if $\{u^*=1\}$ is, because $\{u^*=1\}$ contains $\{u^* \geq t\}$. It remains to prove the opposite implication. We fix $(\gamma,\Psi,T)\in \mathcal{T}(\mathcal{C})$, and aim to verify Definition \ref{def:span borel} for $T$.

\medskip

\noindent{\it Step 7}: We record some measure-theoretic facts needed in the last step. First, we claim that
\begin{align}\label{eq:pae v=0 subset v=0}
    \mathcal{H}^n(T \cap \partial^e \{v=0\} \setminus \{v^* =0\}) = 0\,.
\end{align}
To prove this, note that if $x\in \partial^e \{v=0\}$ and $x$ is a Lebesgue point of $v^*$, then $\{v=0\}$ has positive Lebesgue density at $x$ combined with the Lebesgue property implies that $v^* (x) = 0$. Then \eqref{eq:pae v=0 subset v=0} follows from this fact and the fact that $\mathcal{H}^n$-a.e. $x\in \Omega$ is a Lebesgue point of $v^*$. The second fact, which is an immediate consequence of the Lebesgue points theorem and the coarea formula on $T$, is that there exists $A \subset \mathbb{S}^1$ of full $\mathcal{H}^1$-measure such that
\begin{align}\label{eq:good disks}
    \mathcal{H}^n\big( T[s] \setminus (\{v = 0\}^\one \cup \cup_j \{v >t_j \}^\one )\big) =0 \qquad \forall s\in A\,. 
\end{align}

\medskip

\noindent{\it Step 8 ($\{v^*=0\}$ satisfies Definition \ref{def:span borel} on $T$)}: 
Since $u=1-v$ satisfies Definition \ref{def:gen of spanning}, the equivalent characterization of spanning from \cite[Theorem 3.1]{maggi2023plateau} implies that for each $t_j$, there exists a set $S_j \subset \mathbb{S}^1$ of full $\mathcal{H}^1$-measure such that for $s\in S_j$, the essential partition $\{U_m^j\}_m$ of $T$ induced by $K_j = T[s] \cup \partial^* \{v\leq t_j \} = T[s] \cup \partial^* \{u \geq 1-t_j \}$ satisfies
\begin{align}\label{eq:applying MNR1 thm 3.1}
    \mathcal{H}^n \big((T[s] \cap \{v\leq t_j \}^\zero) \setminus \cup_m \partial^* U_m^j  \big) = 0\,.
\end{align}
Recall from \eqref{eq:cant split} (setting $U=T \setminus T[s]$) that each $(U_m^j)^\one$ is either some $(C_{m_1}^j)^\one \subset \{v\leq t_j \}^\zero$ or $(B_{m_2}^j)^\one \subset \{v\leq t_j \}^\one$. Together with density properties of reduced boundaries, this implies that $\mathcal{H}^n( T[s] \cap \{v\leq t_j \}^\zero \cap \cup_m \partial^* B_m^j)=0$, so that \eqref{eq:applying MNR1 thm 3.1} rewrites as
\begin{align}\label{eq:applying MNR1 thm 3.1 2}
    \mathcal{H}^n \big((T[s] \cap \{v\leq t_j \}^\zero) \setminus \cup_m \partial^* C_m^j  \big) = 0\quad \mbox{if }s\in S_j\,.
\end{align}

Now we define
\begin{align}\notag
    S = A \cap \cap_j S_j\subset \mathbb{S}^1\,. 
\end{align}
Since every set in that intersection has the same $\mathcal{H}^1$-measure as $\mathbb{S}^1$, $\mathcal{H}^1(S) = \mathcal{H}^1(\mathbb{S}^1)$. Therefore, to verify that $\{v^*=0\}$ satisfies Definition \ref{def:span borel} on $T$, it suffices to show that if $s\in S$, then for almost every $x\in T[s]$, there exists a partition $\{T_1,T_2\}$ of $T$ with $x\in\partial^e T_1 \cap \partial^e T_2$ and such that $\{v^*=0\} \cup  T[s]$ essentially disconnects $T$ into $\{T_1,T_2\}$. 

So let us fix $s\in S$ and apply the construction of $\{C_i\}_i$ on $U=T\setminus T[s]$. Since $S \subset A$, \eqref{eq:good disks} implies that $\mathcal{H}^n$-a.e. $x\in T[s]$ belongs to either $\{v=0\}^\one$ or some $\{v >t_{j_0} \}^\one$. Also recall that $\mathcal{H}^n$-a.e. $x\in T$ is a Lebesgue point of $v^*$, and that, by standard facts about Caccioppoli partitions, for each $j$,
\begin{align}\notag
 \mbox{$\mathcal{H}^n$-a.e. $x\in T[s]$ belongs to exactly one of: $(C_m^j)^\one$ for some $i$,} \\ \label{eq:cacc part fact} 
 \mbox{$\{v\leq  t_j\}^\one \cup \{v \leq t_j\}^\half$, or $\partial^* C_m^j \cap \partial^* C_{m'}^j$ for some $m\neq m'$}\,.
\end{align}
Thus to prove the desired claim, it suffices construct satisfactory partitions assuming that $x\in T[s]$ is a Lebesgue point of $v^*$, \eqref{eq:cacc part fact} holds at $x$, and that either $x\in \{v=0\}^\one$ or $x\in \{v >t_{j_0} \}^\one$ for some $j_0$.

\medskip

\noindent{\it The case $x\in \{v=0\}^\one$}: Let $A\subset \mathbb{S}^1$ be an open arc on $\mathbb{S}^1$ with $s$ as one of its boundary points and some other $t\in A_1$ as the other boundary point. Let us set $T_1 = \{v^* = 0 \} \cap \Psi(A \cap B_1^n)$. It is straightforward to check that $T_1$, $T_2$ non-trivially partition $T$ up to Lebesgue null sets and $x\in (T_1)^\half \cap T_2^\half \cap T^\one$ (using $x\in \{v=0\}^\one$). It remains to show that $\partial^e T_1 \cap \partial^e T_2 \cap T^\one$ is $\mathcal{H}^n$-contained in $\{v^* = 0\} \cup T[s]$. By $\partial^e T_1 \cap \partial^e T_2 \cap T^\one = \partial^e T_1 \cap T$ (see \cite[(1.11)]{maggi2023plateau}), followed by \eqref{eq:pae v=0 subset v=0} and the fact that $s,t\in A_1$, so that \eqref{eq:good disks} applies at $s$ and $t$, we have
\begin{align}\notag
    \partial^e T_1 \cap \partial^e T_2 \cap T^\one&= \partial^e T_1 \cap T^\one \\ \notag
    &\subset \big[\partial^e \{v^*=0\} \cap \Psi(A \cap B_1^n) \big] \cup \big[ (T[s]\cup T[t]) \cap \partial^e T_1  \big] \\ \notag
    &\overset{\mathcal{H}^n}{\subset} \{v^* = 0\} \cup T[s] \cup \big[T[t] \cap \partial^e T_1 \cap (\{v = 0\}^\one \cup \cup_j \{v >t_j \}^\one)   \big]\,.
\end{align}
Now each $\{v >t_j \}^\one$ is contained in $\{ v^* >0\}^\one$, so $\{v >t_j \}^\one \subset T_1^\zero$ and thus $\{v >t_j \}^\one \cap \partial^e T_1 = \varnothing$. Continuing the previous string of containments, we find
\begin{align}\notag
    \partial^e T_1 \cap \partial^e T_2 \cap T^\one\overset{\mathcal{H}^n}{\subset} \{v^* = 0\} \cup T[s] \cup \{v = 0\}^\one \,.
\end{align}
Now if $x\in \{v=0\}^\one$ and $x$ is a Lebesgue point of $v$, then $v^*(x)=0$. So we conclude that $\partial^e T_1 \cap \partial^e T_2 \cap T^\one$ is $\mathcal{H}^n$-contained in $\{v^* = 0\} \cup T[s]$, as desired.

\medskip

\noindent{\it The case $x\in \{v>t_{j_0}\}^\one$}: If $x\in \{v>t_{j_0}\}^\one$, then $x\in \{v\leq t_j\}^\zero$ for all $j\geq j_0$, and so \eqref{eq:cacc part fact} and \eqref{eq:applying MNR1 thm 3.1 2} imply that for each $j\geq j_0$, there are $m_1(j)$ and $m_2(j)$ such that $x\in \partial^* C_{m_1(j)}^j\cap \partial^* C_{m_2(j)}^j$ for some $m_1(j)\neq m_2(j)$. Now by \eqref{eq:Ci cim dichotomy} and \eqref{eq:minimality of m}, \eqref{eq:limiting to ci}, for $k=1,2$ there exist $i_k$ such that up to Lebesgue null sets, $C_{m_k(j)}^j \subset C_{i_k}$, with $|C_{i_k}\Delta C_{\overline m_k(j)}^j|\to 0$ for some increasing sequence $\{C_{\overline m_k(j)}^j\}_j$. We claim that
\begin{align}\label{eq:coincidence of things for large j}
&\overline{m}_k(j) = m_k(j)\quad \mbox{for large $j$},\quad \mbox{and} \\ \label{eq:i1'neq i2'}
&i_1 \neq i_2  \,.  
\end{align}
First, for \eqref{eq:coincidence of things for large j}, we apply in order \eqref{eq:dichotomy}, \eqref{eq:Ci cim dichotomy}, \eqref{eq:limiting to ci}, and \eqref{eq:dichotomy} again to write
\begin{align}\notag
   0 &< \liminf_{j\to \infty}|C_{m_k(j)}^j| = \liminf_{j\to \infty}|C_{m_k(j)}^j \cap C_{i_k}| = \liminf_{j\to \infty}|C_{m_k(j)}^j \cap C_{\overline m_k(j)}^j|\\ \notag
   &= \liminf_{j\to \infty}\delta_{m_k(j) \overline m_k(j)}|C_{m_k(j)}^j|\,,
\end{align}
which immediately implies \eqref{eq:coincidence of things for large j}. Next, for \eqref{eq:i1'neq i2'}, we apply in order \eqref{eq:disjointness of C_i}, \eqref{eq:limiting to ci}, \eqref{eq:coincidence of things for large j}, and \eqref{eq:dichotomy} with $m_1(j) \neq m_2(j)$ to write
\begin{align}\notag
    \delta_{i_1 i_2}|C_{i_1}|= |C_{i_1}\cap C_{i_2}| = \lim_{j\to \infty}|C_{\overline m_1(j)}^j \cap C_{\overline m_2(j)}^j|= \lim_{j\to \infty}|C_{ m_1(j)}^j \cap C_{ m_2(j)}^j|=0\,;
\end{align}
since $|C_{i_1}|>0$, this implies that $i_1\neq i_2$.

To conclude, we set $T_1 = C_{i_1}$ and $T_2 = T\setminus T_1$. By step four, we have $\partial^e T_1 \cap \partial^e T_2\cap T^\one \subset \{v^*=0\} \cup T[s]$ up to $\mathcal{H}^n$-null sets; in other words, $\{v^*=0\} \cap T[s]$ essentially disconnects $T$ into $\{T_1,T_2\}$. Furthermore, by $C_{m_k(j)}^j \subset C_{i_k}$, the lower Lebesgue density of $C_{i_k}$ at $x$ is at least $1/2$ (the Lebesgue density of $C_{m_k(j)}^j$ at $x$) for $k=1,2$. But since $i_1\neq i_2$, it must be the case that $|C_{i_1} \cap C_{i_2}|=0$, which implies that the Lebesgue density of $C_{i_1}$ and $C_{i_2}$, and thus $T_1\supset C_{i_1}$ and $T_2\supset C_{i_2}$, at $x$ is $1/2$. Thus $x\in \partial^e T_1 \cap \partial^e T_2$, and we have produced a partition of $T$ according to Definition \ref{def:span borel}.} \end{proof}


We may now use Lemma \ref{lemma connected 2} to define cup competitors and hence our admissible class of functions.

\begin{definition}[Cup competitor]\label{d:cup-comp}
    Let $u\in W^{1,2}_{\rm loc}(\Omega;[0,1])$, let $v=1-u$ and let  $B_r(x)\subset\!\subset \Omega$, and let $\{C_i\}_i$ be the essential partition of $\{v^*>0\}\cap B_r(x)$ given by Lemma \ref{lemma connected 2}. Let $v_i = \mathbf{1}_{C_i} v$ and let $\vphi\in C^\infty(B_r(x);[0,1])$ be the semilinear cutoff constructed in Lemma \ref{lemma semilinearcutoff} satisfying $\vphi\big|_{\partial B_r(x)}=1$ and $\vphi \equiv 0$ on $B_{r/2}(x)$. Let $\psi_i$ denote the {semilinear elliptic} replacement of $v_i\big|_{ B_{r/2}(x)}$ defined in Lemma \ref{lemma semilinearreplacement} {(extended by zero to the entirety of $B_r(x)$)}; we note that $\psi_i$ is continuous and $\psi_i>0$ in $B_\frac{r}{2}(x)$. {Note that the latter property is an artifact of the fact that the boundary data is non-negative.}
    

    Letting $g_i=(v_i +  \min\{\vphi, v-v_i\}) \mathbf{1}_{B_r(x) \setminus B_{r/2}(x)} + \psi_i$, we refer to the function $w_i = 1-g_i$ as a \textbf{cup competitor for $u$ associated to $C_i$ in $B_r(x)$}.
\end{definition}

\begin{remark}\label{remark cupcompetitor}
    Geometrically, cup competitors are modifications of $v=1-u$ in a ball $B_r(x_0)$ centered at a point free boundary point $x_0$ (i.e. $v(x_0)=0$), where, in one of the essentially connected components $C_i$ of the set $\{v>0\} \cap B_r(x_0)$, we use a cutoff $\vphi$ to substitute $v_i = \mathbf{1}_{C_i} v$ by the semilinear replacement $\psi_i$ of $v= v-v_i$ in $B_\frac{r}{2}(x_0)$. For all practical purposes, $\psi_i$ behaves like a harmonic replacement (as we will see in detail in Section \ref{sec freqbound}), in particular $\psi_i>0$ in $B_\frac{r}{2}(x_0)$, implying that for the corresponding cup competitor $w_i$, the function $g_i = 1-w_i$ does not have more essentially connected components than $v$.    
    This description around free boundary points motivates our construction, but cup competitors are perfectly well defined at any point in the domain of $v$, even where $v(x_0)>0$ or when $B_r(x_0)\cap \{v>0\}$ only has one essentially connected component.
\end{remark}


We are now in a position to define our admissible class {and corresponding generalizations of \eqref{new model intro} and \eqref{new model volume constrained}}. Assume that $\wire=\mathbb{R}^{n+1}\setminus \Om$ is compact, $\mathcal{C}$ is a spanning class for $\wire$ satisfying \eqref{eq:spanning class assumption}, and that $F$ and $V$ satisfy \eqref{eq:A1}-\eqref{eq:A3}.

\begin{definition}[Admissible class]\label{d:admissible-class}
    Let
    \[
        \Fcal:= \{u\in W^{1,2}_\loc(\Omega;[0,1]) : \text{$\{u^*\geq t\}$ is $\Ccal$-spanning for all $t\in (1/2,1)$}\}\,.
    \]
    Our \textbf{admissible class} is defined to be
    \[
        \Acal := \{u \in \Fcal : \text{any cup competitor for $u$ lies in $\Fcal$}\}\,.
    \]
\end{definition}
First of all, recalling Remark \ref{r:consistency-spanning-cts}, we observe that $\Acal$ contains the admissible class of functions $u \in C(\Omega;[0,1])$ with weak gradient $\nabla u \in L^2_\loc(\Omega)$ used for \eqref{new model intro} and \eqref{new model volume constrained}- see Lemma \ref{l:gen-adm-class-larger} for the proof that cup competitors for continuous functions satisfy the spanning condition. We are now in a position to introduce the generalizations of \eqref{new model intro} and \eqref{new model volume constrained} over this admissible class. {First, however, we must write the decay at infinity condition from \eqref{new model intro} in a suitable weak sense; when $n\geq 2$ we introduce the space
\begin{equation}\label{eq:alter char of h1dot}
    D^{1,2}_n(\Omega;[0,1]) := \{v: v\in L^{2(n+1)/(n-1)}(\Om;[0,1]),\, \nabla v \in L^2(\Om) \}\,.
\end{equation}
By the Gagliardo-Nirenberg-Sobolev inequality and an extension argument (to account for the compact $\wire$), $D^{1,2}_n(\Om;[0,1])$ is closed under the topology induced by the norm $\|\cdot\|_{L^{2(n+1)/(n-1)}(\Om)}+ \|\nabla \cdot \|_{L^2(\Om)}$. If $n=1$, we cannot use this space since $2$ is the critical Sobolev exponent, so we set
\begin{equation}\label{eq:decay space in the plane}
    D^{1,2}_1(\Om;[0,1]):= \{v:0\leq v \leq 1,\, \nabla v \in L^2,\, \mathcal{L}^2(\{v>t \})<\infty \ \forall t\in (0,1) \}\,.
\end{equation}
Unlike the case when $n\geq 2$, this space is not closed under the norm induced by the $L^2$ norm of the gradient; however, due to our assumption \eqref{eq:sharp existence assumption 2} on $F$ when $n=1$, it will be closed under the convergence one obtains for a minimizing sequence for the generalized formulation of \eqref{new model intro}, which we now state. The generalization of \eqref{new model intro} and \eqref{new model volume constrained} can thus be respectively formulated as
\begin{eqnarray}
\label{new model intro 2}
\inf\Big\{\int_{\Omega} |\nabla u|^2 +F(u)\,:\,u\in D^{1,2}_n(\Omega;[0,1]) \cap \Acal \, \Big\}
\end{eqnarray}
and
\begin{equation}
\inf\left\{\int_{\Omega} |\nabla u|^2 +F(u)\, : \,u\in \Acal, \ \int_{\Omega}V(u)=1 \, \right\}\,. \label{new model volume constrained 2} 
\end{equation}

Let us verify that minimizers of \eqref{new model intro 2} and \eqref{new model volume constrained 2} satisfy the Euler-Lagrange equations \eqref{eq innervariation}-\eqref{eq:differential inequality}.

	\begin{theorem}[Euler-Lagrange equations]\label{thm:EL equations}
		If $\wire=\mathbb{R}^{n+1}\setminus \Om$ is compact, $\mathcal{C}$ is a spanning class for $\wire$ satisfying \eqref{eq:spanning class assumption}, $F$ and $V$ satisfy \eqref{eq:A1}-\eqref{eq:A3}, and $u$ is a minimizer for \eqref{new model intro 2} or \eqref{new model volume constrained 2} then, denoting by $\Phi = F$ in the former case and $\Phi= F-\lambda V$ for suitable $\lambda\in \mathbb{R}$ in the latter, the relations \eqref{eq innervariation}-\eqref{eq:differential inequality} hold. 
        
	\end{theorem}

    The proof of Theorem \ref{thm:EL equations} follows the argument of \cite[Theorem 1.3]{maggi2023hierarchy}, combined with the observation that the variations used to derive \eqref{eq innervariation}-\eqref{eq:differential inequality} lie in the admissible class $\Acal$ whenever the minimizer $u$ does, and if $u$ is continuous then so are the variations. 

	\begin{proof}[Proof of Theorem \ref{thm:EL equations}]
		Beginning with \eqref{eq innervariation}, we first remark that if $\{f_t\}_{-t_0<t<t_0}$ is a smooth one-parameter family of diffeomorphisms with $\{f_t \neq \id\} \cc \Omega$ and $f_0=\id$, then $f_t^{-1}\circ \gamma \in \mathcal{C}$ whenever $\gamma\in \mathcal{C}$ (by the homotopic closedness of $\mathcal{C}$). Since {$\{u^*\geq t\}$ is $\Ccal$-spanning according to Definition \ref{def:span borel} for triples $(\gamma,\Psi,T)$ if and only if $\{(u\circ f_t)^* \geq t\}$ is for triples $(f_t^{-1}\circ\gamma,\Psi,T)$,}
        we deduce that $u\circ f_t$ satisfies the spanning constraint if and only if $u$ satisfies it. {Thus the same applies to any cup competitor, so we remain in the class $\Acal$ when performing inner variations}. Using these variations which preserve the spanning condition {and remain in $\Acal$}, the inner variational equation \eqref{eq innervariation} can be deduced from the formulas in Lemma \ref{lemma volume fixing} by a standard computation, following for example the derivation of the constant mean curvature condition for volume-constrained minimizers of perimeter \cite[Theorem 17.20]{maggi2012sets}.
		
		The idea behind \eqref{EL outer} and \eqref{eq:differential inequality} is to find a way to make outer variations which preserve the spanning constraint, the condition $0\leq u \leq 1$, {and remain in $\Acal$ when $u\in \Acal$}. The computations may be repeated exactly as in \cite[Proof of Theorem 1.3]{maggi2023hierarchy} by taking $\varepsilon=1$ there, so we only give the outline. For \eqref{eq:differential inequality}, we construct admissible variations via the following observation: if $u$ minimizes \eqref{new model intro} or \eqref{new model volume constrained}, $\varphi \in C_c^1(\Om;[0,\infty))$, and $h\in \mathrm{Lip}_c([0,1);[0,\infty))$, then $u+\sigma h(u)\varphi$ satisfies
		\begin{align*}
			{\mbox{$\{(u + \sigma h(u)\varphi)^* \geq t\}=\{u^*\geq t\}$ \qquad for $t\in [1-\gamma,1]$ for small enough $\gamma>0$}}\,,
		\end{align*}
        and
        \begin{align*}
            \mbox{$0\leq u + \sigma h(u) \varphi \leq 1$ \qquad for small enough $\sigma>0$}\,.
        \end{align*}
        {It remains to verify that cup competitors of such variations lie in $\Fcal$. Observe that since $\spt \, h \cc [0,1)$ and $\varphi \in C_c^1(\Omega;[0,\infty))$, there is $t_0\in (0,1)$ and $\sigma_0>0$ such that $\{u=t\}=\{u+\sigma h(u)\varphi = t\}$ for all $t\in (t_0,1)$ and $0<\sigma<\sigma_0$. Since the essentially connected components are limits of superlevel sets $\{u>t\}$ and $\{u+\sigma h(u)\varphi>t\}$ as $t\nearrow 1$, the equality of level sets near $1$ implies that the procedure in Lemma \ref{lemma connected 2} yields the same partition for $u$ and $u+\sigma h(u)\varphi$. By the uniqueness of such partitions, the essentially connected components of $\{1-u>0\}$ and $\{1-u-\sigma h(u)\varphi>0\}$ are the same. Fix such a component $C_i$ and set $v=1-u$, $v^\sigma = 1 - u - \sigma h(u)\varphi$. Then $v_i=\mathbf{1}_{C_i}v$ and $v_i^\sigma=\mathbf{1}_{C_i}v_i^\sigma$ share the same zero set. Recall from Definition \ref{d:cup-comp} that the cup competitors for $u$ and $u+\sigma h(u)\varphi$ are $1-g_i$ and $1-g_i^\sigma$, where $g_i=(v_i +  \min\{\vphi, v-v_i\}) \mathbf{1}_{B_r(x) \setminus B_{r/2}(x)} + \psi_i$ and $g_i^\sigma = (v_i^\sigma +  \min\{\vphi, v-v_i^\sigma\}) \mathbf{1}_{B_r(x) \setminus B_{r/2}(x)} + \psi_i$, with $\psi_i, \psi_i^\sigma >0$. It follows from this formula and the equality of zero sets for $v_i$ and $v_i^\sigma$ that $g_i$ and $g_i^\sigma$ have the same zero set. Since $1-g_i$ has $\mathcal{C}$-spanning $1$-level set, we conclude that $1-g_i^\sigma$ does as well.} 
        Therefore, after fixing the volume constraint if necessary by using the volume fixing variations in Lemma \ref{lemma volume fixing}, {with which we again remain in $\Acal$ by the same reasoning as for inner variations}, we have a one-parameter family $\{u+\sigma h(u)\varphi\}_\sigma$ of admissible outer variations by positive test functions with which to test minimality. The inequality \eqref{eq:differential inequality} is found by testing the minimality of $u$ against $u+\sigma h_k(u)\varphi$, then letting $\sigma\to 0$, and finally sending $h_k\nearrow\mathbf{1}_{[0,1)}$ in the resulting inequality (see \cite[Proof of Theorem 1.3, Steps 1-2]{maggi2023hierarchy}). We remark that this last step of this computation utilizes the fact that
		\begin{equation}\notag
			\Phi'(1)=F'(1)-\lambda V'(1)=0\,,
		\end{equation}
		which follows from our assumption \eqref{eq:A2}.
		
		The argument for \eqref{EL outer} is similar to \eqref{eq:differential inequality} and follows precisely \cite[Proof of Theorem 1.3, Steps 3-7]{maggi2023hierarchy}. The outer variations we wish to use are of the form $u+\sigma h(u)\varphi$ (up to volume constraints, which are handled using Lemma \ref{lemma volume fixing} again), but this time with $\varphi \in C_c^1(\Omega)$, $|\sigma|<\sigma_0$, and $h\in \mathrm{Lip}_c([0,1))$. Now since $h$ is Lipschitz with $\spt \,h \cc [0,1)$, there is $\sigma_0>0$ small enough (depending on $h$ and $\varphi$) such that
		\begin{align*}
			{\mbox{$\{(u + \sigma h(u)\varphi)^*\geq t\}=\{u^*\geq t\}$}}{\quad}\mbox{and}\quad \mbox{$ u + \sigma h(u) \varphi \leq 1$ for $|\sigma|<\sigma_0$.}
		\end{align*}
		{First of all, note that for justifying that cup competitors of these variations remain in $\Fcal$, the very same argument as for \eqref{eq:differential inequality} applies, since again $\spt h \cc [0,1)$.} In addition, since $\sigma h(u)\varphi$ is no longer non-negative $\Lcal^{n+1}$-a.e., an extra argument is necessary to ensure that the variations remain non-negative. This is achieved by showing that
		\begin{equation}\label{eq:nontriviality of u on conn comp}
			\mbox{$u$ is lower-semicontinuous and $\Omega'\subset\Omega$ open, connected $\implies$ $u\equiv 0$ on $\Omega'$ or $u>0$ on $\Omega'$}\,.
		\end{equation}
		\eqref{eq:nontriviality of u on conn comp} then allows one to assume without loss of generality that $u>0$ on $\spt\, \varphi$, so that $u+\sigma h(u)\varphi>0$ on $\spt \varphi$ for small enough $\sigma$. The proof of \eqref{eq:nontriviality of u on conn comp} follows from the fact that $e^{-|\sup \Phi''|r}\fint_{B_r(x_0)} u$ is decreasing for small $r$, which is derived by testing \eqref{eq:differential inequality} with $\{\varphi_k\}$ approximating $[(r^2-|x-x_0|^2)/2]_+$ and using the property $\Phi'(0)=0$ (which follows from \eqref{eq:A2}). With these variations in hand, one tests the minimality of $u$ against $u+\sigma h(u)\varphi $, sends $\sigma \to 0$, approximates $\mathbf{1}_{[0,t]}$ by $h_k^t\overset{k\to \infty}{\to}\mathbf{1}_{[0,t]}$ for each $t\in (0,1)$, and integrates the resulting equality in $t$ {(cf. Lemma \ref{l:almost-subharm-v})}. 
    \end{proof}

{Notice that by definition, a minimizer $u$ of \eqref{new model intro 2} or \eqref{new model volume constrained 2} will have lower energy $\int_\Omega |\nabla u|^2 + F(u)$ than any of its cup competitors. However, when one additionally has the volume constraint, one can only expect almost-minimality of the energy for minimizers $u$.} As consequence of Theorem \ref{thm:EL equations}, we prove two corollaries: an almost-minimality inequality with quadratic error for a suitably modified functional and the corresponding Euler-Lagrange equations \eqref{eq modified iv}-\eqref{eq modified ineq} for $v=1-u$ for minimizers $u$ of our variational problems.

{
\begin{corollary}[Quadratic almost-minimality for Lagrangian functional]\label{corollary qdetachment}
     Let $\wire=\rn \setminus \Om$ be compact and let $\mathcal{C}$ be a spanning class for $\wire$ satisfying \eqref{eq:spanning class assumption}. Suppose that $F$, $V$ satisfy \eqref{eq:A1}-\eqref{eq:A3}, and $u$ is a minimizer for \eqref{new model volume constrained 2}. Then there exist $\tilde{C}>0$, $r_0>0$ and $\eta_0$, depending on $F$, $V$, $\wire$, $\mathcal{C}$, and $u$, such that for all $w\in W^{1,2}(\Om;[0,1])$ with $\{w\geq t\}$ $\mathcal{C}$-spanning $\wire$ for all $t\in (1/2,1)$, $\{u\neq w\}\subset B_{r_0}(x_0)\cap \Om$ for some $x_0\in \Omega$, and $\left|\int_{B_{r_0}(x_0)\cap\Omega} V(u) - V(w) \, dx\right| <\eta_0$
\begin{equation}\label{eq:almost min uniformity}
\int_\Om |\nabla u|^2 + F(u) - \lambda V(u)\,dx \leq \int_\Om |\nabla w|^2 + F(w)- \lambda V({w})\,dx + \tilde{C}\left\|u-w\right\|_{L^2(\Omega)}^2\,.    
\end{equation}
\end{corollary}
\begin{proof}
If $u$ is constant on every connected component of $\Omega$, then \eqref{eq:almost min uniformity} is trivial. Assuming then that $u$ is not constant on some connected component of $\Omega$, we can choose two balls $B_{r_1}(z_1)$ and $B_{r_2}(z_2)$ with $|z_1-z_2|> r_1+r_2$ such that $u$ is not constant on either. Let $\{f_t^1\}_t$, $\{f_t^2\}_t$ be two families of of volume fixing variations according to Lemma \ref{lemma volume fixing} with $A= B_{r_1}(z_1)$ and $A= B_{r_2}(z_2)$ respectively, with constants $\eta_0$, $t_0$, $\beta_0$, and $C_0$ as in that lemma valid for both families. By decreasing $t_0$, let us also assume it is small enough so that the Taylor expansions in Lemma \ref{lemma volume fixing}.i hold for both families with uniform constants. Let $r_0$ be small enough so that $\omega_{n+1}r_0^{n+1}<\eta_0$ and so that the ball $B_{r_0}(x_0)$ is disjoint from at least one of $B_{r_1}(z_1)$ or $B_{r_2}(z_2)$. Then for any $w$ and $B_{r_0}(x_0)$ as in the statement of the lemma, with $B_{r_0}(x_0)$ disjoint from $B_{r_i}(z_i)$ for some $i\in \{1,2\}$, let $\tilde{w}$ be the function which is $w$ on $\Omega \setminus B_{r_i}(z_i)$ and $u\circ f_t^i$ on $B_{r_i}(z_i)$, where $t$ is chosen according to Lemma \ref{lemma volume fixing}.ii so that $\int_\Omega V(\tilde{w})=1$, that is $\int_{B_{r_i}(z_i)}V(u \circ f_t^i) = \int_{B_{r_i}(z_i)}V(u) + \int_{B_{r_0}(x_0)}V(u) - V(w)\, dx$. Then we may test the minimality of $u$ against $\tilde{w}$ and add $-\lambda=-\lambda\int_\Omega V(u) = -\lambda\int_\Omega V(\tilde{w})$ to both sides, yielding
\begin{equation}\label{eq:first test}
   \int_\Omega |\nabla u|^2 + F(u) - \lambda V(u) \, dx \leq \int_\Omega |\nabla \tilde{w}|^2 + F(\tilde{w}) - \lambda V(\tilde{w})\,.
\end{equation}
Let $\eta=\int_{B_{r_0}(x_0)} V(u) - V(w)$, and note that by Lemma \ref{lemma volume fixing}.i, $|t|\leq C_1|\eta|$ for some $C_1$ depending only on the families $\{f_t^i\}$. By Taylor expanding $\int_{B_{r_i}(z_i)}|\nabla \tilde{w}|^2 + F(\tilde{w})$ and using the vanishing inner variation \eqref{eq innervariation} from Theorem \ref{thm:EL equations}, we obtain
\begin{equation}\label{eq:first expansion}
    \int_{B_{r_i}(z_i)}|\nabla \tilde{w}|^2 + F(\tilde{w})- \lambda V(\tilde{w})\, dx = \int_{B_{r_i}(z_i)}|\nabla u|^2 + F(u) - \lambda V(u) + {\rm O}(t^2)\,,
\end{equation}
where, by Lemma \ref{lemma volume fixing}.i and $|t| \leq C_1|\eta|$, we may estimate the error in the previous equality by
\begin{equation}\notag
  |{\rm O}(t^2)| \leq C_2t^2\int_{\Omega}|\nabla u|^2 + F(u) + |\lambda| V(u) \leq C_3\eta^2
\end{equation}
for some universal constant $C_3$. Inserting this estimate into \eqref{eq:first expansion} and then inserting the modified \eqref{eq:first expansion} into \eqref{eq:first test} and using that $w=u$ in $B_{r_i}(z_i)$ while $w=\tilde{w}$ outside of $B_{r_i}(z_i)$, we find that
\begin{equation}\label{eq:better estimate}
    \int_\Omega |\nabla u|^2 + F(u) - \lambda V(u) \leq \int_\Omega |\nabla w|^2 + F(w) - \lambda V(w) + C_3\eta^2\,.
\end{equation}
But by H\"{o}lder's inequality and the Lipschitz regularity of $V$, we may estimate
\begin{equation}\notag
 |\eta| \leq \int_{B_{r_0}(x_0)}|V(u) - V(w)| \leq {\rm Lip}\, V (\omega_{n+1}r_0^{n+1})^{1/2}\|u-w\|_{L^2(\Omega)}\,.   
\end{equation}
Plugging this estimate into \eqref{eq:better estimate} gives \eqref{eq:almost min uniformity}.
\end{proof}
}

\begin{corollary}[Euler-Lagrange equations for $v$]\label{c:EL-v}
	Let $\wire=\rn \setminus \Om$ be compact and let $\mathcal{C}$ be a spanning class for $\wire$ satisfying \eqref{eq:spanning class assumption}. Suppose that $F$, $V$ satisfy \eqref{eq:A1}-\eqref{eq:A3}, and $u$ is a minimizer for \eqref{new model intro 2} or \eqref{new model volume constrained 2}, respectively. Then, setting $\Phi = F$ in the former case and $\Phi= F-\lambda V$ for suitable $\lambda\in \mathbb{R}$ in the latter case, $G(t)=\Phi(1-t)-\Phi(1)$ and $v=1-u$ satisfy \eqref{eq:g assumptions} and \eqref{eq modified iv}-\eqref{eq modified ineq}, respectively, for an appropriate choice of non-negative Radon measure $\mu$ on $\Om$.
\end{corollary}

\begin{proof}
	Let $u$ be a minimizer for \eqref{new model intro 2} or \eqref{new model volume constrained 2}. We first check that $G$ satisfies \eqref{eq:g assumptions}. The fact that $G$ is $C^2$ follows from \eqref{eq:A1}, and trivially $G(0)=0$. Also, due to \eqref{eq:A2},
	\begin{equation}\notag
		G'(0)=-\Phi'(1)= \begin{cases}
			-F'(1)=0 &\quad \mbox{if $\Phi=F$}\\ 
			-F'(1) + \lambda V'(1)=0 &\quad \mbox{if $\Phi=F-\lambda V$}\,.
		\end{cases}
	\end{equation}
    Similarly, we have that $G'(1)=0$ since $F'(0)=V'(0)=0$. Next, \eqref{eq modified iv}-\eqref{eq modified out} for $v=1-u$ follow from substituting $G$ and $v$ into the criticality conditions \eqref{eq innervariation}-\eqref{EL outer} from Theorem \ref{thm:EL equations}, which applies to $u$ since \eqref{eq:A1}-\eqref{eq:A3} are satisfied. The existence of a measure $\mu$ such that \eqref{eq modified ineq} holds follows directly from the differential inequality \eqref{eq:differential inequality} and the identification of monotone linear functionals on $C_c^\infty(\Om)$ with non-negative Radon measures \cite[pg 53]{EG}, as in the proof of Lemma \ref{lemma restriction subsol}.
\end{proof}

\subsection{Lower frequency bound {and regularity} for minimizers: Proof of Theorem \ref{thm:main regularity theorem}.i}\label{sec freqbound} {In this section we obtain a lower frequency bound for minimizers of \eqref{new model intro 2} and \eqref{new model volume constrained 2} and combine it with Theorem \ref{thm:holder reg for crit points} to prove Theorem \ref{thm:main regularity theorem}.i. We are always using the precise representative $v^*$ throughout this section, which we simply denote as $v$ for simplicity of notation.}

We begin with deriving almost-subharmonicity equations associated to the ``component functions"  $v_i$ arising from restricting $v=1-u$ to essentially connected components $C_i$ of balls.

\begin{lemma}\label{eq lemma eq comp}
Let $\wire=\rn \setminus \Om$ be compact and let $\mathcal{C}$ be a spanning class for $\wire$ satisfying \eqref{eq:spanning class assumption}. Suppose that $F$, $V$ satisfy \eqref{eq:A1}-\eqref{eq:A3}, and $u$ is a minimizer for \eqref{new model intro 2} or \eqref{new model volume constrained 2}, respectively. Let $G$ be as in Corollary \ref{c:EL-v} for these respective problems, let $v=1-u$, let $x_0 \in \{v=0\}$, and let $r>0$ be such that $B_r(x_0) \subset \Omega$. Let $C_i$ be the essential connected components of $\{v>0\}\cap B_r(x_0)$ given by Lemma \ref{lemma connected 2}.  The functions $v_i = \mathbf{1}_{C_i}v$ are in $W^{1,2}(B_r(x_0))$ and satisfy $2\Delta v_i = G'(v_i) + \mu_i$  distributionally in $B_r(x_0)$ for some non-negative Radon measures $\mu_i$. Similarly, $w_i = v-v_i$ are in $W^{1,2}(B_r(x_0))$ and also satisfy $2\Delta w_i =  G'(w_i) + \nu_i$ distributionally on $B_r(x_0)$ for some non-negative Radon measures $\nu_i$.
\end{lemma}

\begin{proof}
     Let us first demonstrate that $v_i \in W^{1,2}(B_r(x_0))$. Recall from Lemma \ref{lemma connected 2} that $C_i$ is an increasing limit of sets of finite perimeter $F^i_j \subset \{v > t_{j}\} \cap B_r(x_0)$ for $j \geq j_i$, for some sequence $t_{j} \downarrow 0$ as $j\to\infty$. Consider the associated sequence of functions $v_{i,j}= (v-t_{j}) \mathbf{1}_{F^i_j}$. Since $v|_{F^i_j} \in W^{1,2}$, one may test $v_{i,j}$ against $\mathrm{div} X$ for $X \in C_c^\infty(B_r(x_0);\mathbb{R}^{n+1})$ to explicitly verify that the weak derivative $\nabla v_{i,j}$ is $\nabla v \mathbf{1}_{F^i_j} \in L^2(B_r(x_0))$ (with $L^2$ norm bounded independently of $j$), namely
    \[
        \int_{B_r(x_0)} v_{i,j} \, \mathrm{div} X \, dx = - \int_{F^i_j} \nabla v \cdot X\, dx \qquad \text{for every $X \in C_c^\infty(B_r(x_0);\mathbb{R}^{n+1})$.}
    \]
    Note that the boundary term vanishes due to the fact that $\partial^* F^i_j \cap B_r(x_0) \subset \{v= t_{j}\}$, as guaranteed by Lemma \ref{lemma connected 2}. We may now use the Dominated Convergence Theorem to  deduce that $v_{i,j} \to v_i$ in $L^2$ as $j\to \infty$, which in turn allows us to pass the above identity for the weak derivative to the limit $v$, with domain $C_i$ on the right-hand side. In conclusion, $v_i$ satisfies
     \begin{equation*}\label{eq IBP}
           \int_{B_r(x_0)} v_i \, \mathrm{div} X \, dx = - \int_{C_i} \nabla v \cdot X\, dx \qquad \text{for every $X \in C_c^\infty(B_r(x_0);\mathbb{R}^{n+1})$}
     \end{equation*}
 
    By taking $X = e_j \varphi$ for any $j \in \{1,\cdots , n+1\}$, where $\{e_j\}_{j=1}^{n+1}\subset \R^{n+1}$ is the standard orthonormal basis and $\varphi\in C_c^1(B_r(x_0))$, we deduce that for each $j$ we have
       \begin{equation*}
           \int_{B_r(x_0)} v_i \, \partial_j \varphi \, dx = - \int_{C_i} \partial_j v \varphi\, dx \qquad \text{for every $X \in C_c^\infty(B_r(x_0);\mathbb{R}^{n+1})$}
     \end{equation*}
     implying that $v_i \in W^{1,2}(B_r(x_0))$ and that $\nabla v_i = \nabla v \mathbf{1}_{C_i}$.\\

    Let us now demonstrate that each $v_i$ satisfies $2\Delta v_i = G'(v_i) + \mu_i$ for some non-negative Radon measure $\mu_i$. Recalling the proof of Corollary \ref{c:EL-v}, observe that it suffices to demonstrate that
    \begin{equation}\label{e:EL-ineq-u_i}
        0\leq  \int_{B_r(x_0)} 2\nabla u_i \cdot \nabla \varphi + \Phi'(u)\, \varphi \,dx \qquad \text{for all $\varphi \in C_c^1({B_r(x_0)};[0,\infty))$,}
    \end{equation}
    where $u_i = \mathbf{1}_{C_i} u$. Now we proceed as in \cite[Proof of Theorem 1.3, Steps 1-2]{maggi2023hierarchy}, testing the minimality of $u$ for our variational problem against the 1-parameter family of competitors $u_{i,t} := u + t h(u_i) \varphi$ with $h\in C_c^\infty([0,1);(0,\infty))$ (composed with a suitable diffeomorphism in the volume constrained case) and differentiating in $t$. For a vector field $X\in C_c^\infty({B_r(x_0)};\mathbb{R}^{n+1})$ (chosen suitably in the volume constrained case), this yields (c.f. final displayed equation of \cite[Proof of Theorem 1.3, Step 1]{maggi2023hierarchy})
    \begin{align*}
        0 &\leq \int_{B_r(x_0)}\Big(|\nabla u|^2 + \Phi(u) \Big) \mathrm{div} X - 2 \langle \nabla u, \nabla u \nabla X \rangle \,dx \\
        &\qquad+ \int_{B_r(x_0)} 2 h'(u_i)\varphi \langle \nabla u, \nabla u_i \rangle + 2 h(u_i) \langle \nabla u, \nabla \varphi \rangle \,dx + \int_{B_r(x_0)} \Phi'(u) h(u_i) \varphi\,dx \\
        &= \int_{B_r(x_0)} 2 h'(u_i)\varphi |\nabla u_i|^2 + 2 h(u_i) \langle \nabla u, \nabla \varphi \rangle \,dx + \int_{B_r(x_0)} \Phi'(u) h(u_i) \varphi\,dx\,.
    \end{align*}
    Now we may choose a monotone increasing sequence of non-negative functions $h_k$ with $h_k' \leq 0$ whose pointwise limit is 1 on $[0,1)$ as $k\to \infty$, and use the Monotone Convergence Theorem in the above inequality with this sequence in place of $h$, to obtain \eqref{e:EL-ineq-u_i} as desired. Here we have crucially used the fact that $h_k'(u_i)\varphi |\nabla u_i|^2 \leq 0$.

    It remains to verify that $w_i$ satisfy $2\Delta w_i = G'(w_i) + \nu_i$ for non-negative Radon measures $\nu_i$. Notice that $v = \sum_{i=1}^\infty v_i$ in $W^{1,2}(B_r(x_0))$ and $G(v)= \sum_{i=1}^\infty G(v_i)$ a.e. in $B_r(x_0)$, implying that $\mu = \sum_{i=1}^\infty \mu_i$ as Radon measures. Thus, $w_i \in W^{1,2}(B_r(x_0))$ and $2\Delta w_i = G'(w_i)+\nu_i$ with $\nu_i = \mu-\mu_i$ which is non-negative since $\mu = \sum_{j=1}^\infty \mu_j \geq \mu_i$.\end{proof}

In the next few lemmas we generalize key properties of harmonic functions to solutions of the semilinear equation $2\Delta \psi = G'(\psi)$. We start proving existence of solutions to this PDE and some of its properties in two particular settings.

\begin{lemma}\label{lemma semilinearreplacement}
   Let $r >0$, let $G$ satisfy \eqref{eq:g assumptions}, and let $w \in W^{1,2}(B_r;[0,1])$. Then there exists a function $\psi \in C^2(B_r)\cap W^{1,2}(B_r)$ satisfying $0<\psi<1$ in $B_r$, and weakly solving the boundary value problem
 \begin{equation}\label{eq semreplacement}
    \begin{cases}
        2 \Delta \psi =G'(\psi)\, \mbox{ in $ B_r$}\\
        \psi =w \, \mbox{ on $\partial B_r$}.
    \end{cases}
\end{equation}
Furthermore, these solutions satisfy:
\begin{enumerate}
    \item If $w_1 \leq w_2$ as traces in $\partial B_r$, then corresponding semilinear replacements $\psi_1$ and $\psi_2$ satisfy $\psi_1 \leq \psi_2$ in $B_r,$
    \item there exists $r_0>0$ depending only on the Lipschitz constant of $G'$ such that if $r \in (0, r_0)$, we have that $\Vert \psi \Vert_{L^\infty(B_r)}\leq C \Vert w \Vert_{L^\infty(\partial B_r)}$.
\end{enumerate}
\end{lemma}
\begin{proof}
   Set $K = \Lip(G')>0$, and consider the function $H(t) = G'(t)-Kt$ and the operator $L = 2\Delta-K$. Notice that by construction $H(t)$ is a decreasing function and that $f$ satisfies $2\Delta f = G'(f)$ if and only if $L(f) = H(f)$. Since $G'(1)=0$ and $0 \leq w \leq 1$, the constant function $1$ is a supersolution to \eqref{eq semreplacement}, similarly  $0$ is a subsolution to the same problem. So, we can apply the method of monotone iterations to find $\psi$. More precisely, define recursively $\psi_0 = 0$ and $\psi_{k+1}$ as the unique solution to the linear problem
    \begin{equation}\label{eq lp}
    \begin{cases}
       L \psi_{k+1} = H(\psi_k)\, \mbox{ in $ B_r$}\\
        \psi_{k+1} =w \, \mbox{ on $\partial B_r$}.
    \end{cases}
\end{equation}
This sequence is pointwise monotonically increasing. Indeed, reasoning by induction we see first that the inequality $\psi_1 > 0$ follows from a direct application of the strong maximum principle for solutions of $L$, noting that $H(0)=0$. Assuming $\psi_k \geq \psi_{k-1}$ for $k\geq 1$, we have that if we subtract the equations for $\psi_{k+1}$ and $\psi_k$ and test the difference with $\phi = (\psi_{k+1}-\psi_k)_-$ (bearing in mind the convention $f= f_+ -f_-$) we obtain 
\begin{equation}\label{eq itmon}
    \int_{B_r}  2|\nabla \phi|^2 + K \phi^2 =  \int_{B_r} [H(\psi_{k})-H(\psi_{k-1})] \phi \leq 0.
\end{equation}
If $\xi$ is any supersolution to \eqref{eq itmon}, in particular the constant supersolution $\xi = 1$, a similar argument with $\xi$ in place of $\psi_{k+1}$, noting that on $\partial B_r$ we have $(\xi-\psi_k)_- = (\xi-w)_-=0$, also shows that $ \psi_k \leq \xi$ for any $k \in \N$. So, by an elementary compactness argument exploiting the uniform $L^\infty$ and $W^{1,2}$ boundedness of the sequence, we deduce that $\psi_k \to \psi$ in $W^{1,2}$ with $\psi\in L^\infty$ weakly solving \eqref{eq semreplacement}. The boundedness of $\psi$ and the Lipschitz regularity of $G'$ implies via standard elliptic regularity that $\psi \in C^2(B_r)$. Notice now that since $G'(1)=0$, $\psi$ satisfies
$\Delta (1-\psi) = c(x)(1-\psi)$ with $|c(x)|\leq K$. Implying by the strong maximum principle that $\psi <1$. Similarly, since $G'(0)=0$, we have that
\begin{equation}\label{eq linearization}
    \Delta \psi = c(x)\psi
\end{equation}
with $|c(x)|\leq K$ which implies by the maximum principle that $w>0$ in $B_r$.

On the other hand, let $w_1, w_2 \in W^{1,2}(B_r; [0,1])$ with $w_1 \leq w_2$ as traces in $\partial B_r$ and let $\psi_1$ and $\psi_2$ be their corresponding semilinear replacements. Notice that $\psi_2$ is a supersolution to the problem satisfied by $\psi_1$ since $w_2\geq w_1$, so the monotone sequence obtained iterating from zero with boundary data $w_1$ always lies below $\psi_2$, implying that $\psi_1\leq \psi_2$.

Finally, we can apply standard elliptic estimates to \eqref{eq linearization},e.g. \cite[Theorem 3.7]{gilbarg1977elliptic}, to deduce that
\begin{equation}\label{eq absorbtion}
    \Vert \psi\Vert_{L^\infty(B_r)} \leq C \Vert w \Vert_{L^\infty(\partial B_r)}+Cr^2 \Vert \psi \Vert_{L^\infty( B_r)},
\end{equation}
 which implies $\Vert \psi \Vert_{L^\infty(B_r)}\leq C \Vert w \Vert_{L^\infty(\partial B_r)}$ for $r< r_0$ sufficiently small.
\end{proof}

\begin{lemma}\label{lemma semilinearcutoff}   
     Let $r>0$ and let $G$ satisfy \eqref{eq:g assumptions}. Given  $L \in (0,1]$ there exists a non-negative radial function $\varphi_L \in C^2(\overline{B_r \setminus B_\frac{r}{2}})$, with $\varphi_L>0$ in $B_r\setminus \overline{B_\frac{r}{2}}$, and satisfying the BVP
 \begin{equation}\label{eq semrcutoff}
    \begin{cases}
        2 \Delta \varphi_L =G'(\varphi)\, \mbox{ in $ B_r\setminus B_\frac{r}{2}$}\\
        \varphi_L =0 \, \mbox{ on $\partial B_\frac{r}{2}$},\\
          \varphi_L =L \, \mbox{ on $\partial B_r$}.
    \end{cases}
\end{equation}
Furthermore, these solutions satisfy:
\begin{enumerate}
    \item $\varphi_{L}\leq \varphi_{M}$ if $L\leq M$.
    \item there exists $r_0>0$ depending only on the Lipschitz constant on $G'$ such that if $r \in (0, r_0)$, we have that $\Vert \nabla \varphi_L\Vert_{L^\infty(B_r \setminus B_\frac{r}{2})}\leq \frac{C}{r} L$.
\end{enumerate}
\end{lemma}
\begin{proof}
    As in the proof of Lemma \ref{lemma semilinearreplacement}, we can construct $\varphi_L$ applying the method of monotone iterations, starting from the constant zero function as in Lemma \ref{lemma semilinearreplacement} and using the constant function $1$ as the reference supersolution. If $L, M \in (0,1]$, with $L \leq M$ we note that $\varphi_M$ is a supersolution to the problem satisfied by $\varphi_L$, so at each step of the iteration for $\varphi_L$ the corresponding solution lies below $\varphi_M$, implying that $\varphi_L \leq \varphi_M$. Also notice that since we are starting from the zero solution, at the first step of the iteration we are solving a linear problem with radial right-hand side and radial boundary date, so we can use the maximum principle inductively to guarantee that at each step we obtain a radial function, for this reason we have that the limiting function $\varphi_L$ is necessarily radial. 
    
    On the other hand, using standard elliptic estimates applied to solutions of $2\Delta \varphi = G'(\varphi) = c(x) \varphi$ (see e.g. \cite[Theorem 3.9]{gilbarg1977elliptic}), we deduce that
    \begin{equation*}
        \Vert \nabla \varphi_L\Vert_{L^\infty( B_r\setminus B_\frac{r}{2})}\leq \frac{C}{r}\Vert\varphi_L\Vert_{L^\infty (B_r\setminus B_\frac{r}{2})} \leq \frac{C}{r}L,
    \end{equation*}
    where the last inequality holds for $r<r_0$ as long as $r_0>0$ is small enough by the very same reasoning as in \eqref{eq absorbtion}.
\end{proof}

Our final result concerning the properties of solutions to semilinear PDEs is the following generalization of the Alt-Caffarelli trace inequality, see \cite[Lemma 3.7]{velichkov2023regularity}.

\begin{lemma}\label{lemma ACI}
 Let $r \in (0,1)$, let $G$ satisfy \eqref{eq:g assumptions}, and let $w \in W^{1,2}(B_r)$ satisfying $2\Delta w \geq G'(w)$ weakly in $B_r$ with $0 \leq w \leq 1$. Let $\psi$ be the semilinear replacement of $w$ constructed in Lemma \ref{lemma semilinearreplacement}. Then, 
    \begin{equation}\label{eq trace AC}
     |\{w =0\}\cap B_r|\Big(\fint_{\partial B_{r}} w \,d\mathcal{H}^{n}\Big)^2\leq  Cr^2\int_{B_r}|\nabla (w - \psi)|^2.
\end{equation}
\end{lemma}
\begin{proof}
We essentially reproduce the argument in \cite[Lemma 3.7]{velichkov2023regularity} highlighting the small modifications required in our case. For each \( |z| \leq \frac{1}{2} \), we consider the functions \( w_z \) and \( \psi_z \) defined on \( B_r \) as
\[
w_z(x) := w((r - |x|)z + x) \quad \text{and} \quad \psi_z(x) := \psi((r - |x|)z + x).
\]

Note that both \( w_z \) and \( \psi_z \) still belong to \( W^{1,2}(B_r) \) and that their gradients are controlled from above and below by the gradients of \( w \) and \( \psi \). We call \( S_z \) the set of all \( \xi \in \partial B_1 \) such that the set \( \{ \rho \colon r/8 \leq \rho \leq r, \, w_z(\rho \xi) = 0 \} \) is not empty. For \( \xi \in S_z \), we define
\[
r_\xi = \inf \{\rho \colon r/8 \leq \rho \leq r, w_z(\rho \xi) = 0\}.
\]

For almost all \( \xi \in \partial B_1 \) (and then for almost all \( \xi \in S_z \)), the functions \( \rho \mapsto w_z(\rho \xi) \) are square integrable. For those \( \xi \), one can suppose that the equation
\[
\big( w_z(\rho_2\xi) - \psi_z(\rho_2\xi) \big) - \big( w_z(\rho_1\xi) - \psi_z(\rho_1\xi) \big) = \int_{\rho_1}^{\rho_2} \xi \cdot \nabla (w_z(\rho\xi) - \psi_z(\rho\xi)) \, d\rho
\]
holds for all \( \rho_1, \rho_2 \in [0, r] \). Moreover, we have the estimate
\begin{equation}\label{e:Alt-Caff-psi-L2-bd}
\psi_z(r_\xi \xi) = \int_{r_\xi}^r \xi \cdot \nabla (\psi_z - w_z)(\rho \xi) \, d\rho \leq \sqrt{r - r_\xi} \left( \int_{r_\xi}^r |\nabla (\psi_z - w_z)|^2(\rho \xi) \, d\rho \right)^{1/2}.
\end{equation}

At this point the proof slightly differs from \cite[Lemma 3.7]{velichkov2023regularity}. We condense this change in the following claim. 

\medskip

\noindent {\it Claim:} For every $x\in B_r$,
\begin{equation}\label{eq lbGR}
 \psi(x) \geq  \frac{r - |x|}{C r} \fint_{\partial B_r} w \, d\mathcal{H}^{n}.
\end{equation}

We start noticing that we can apply Lemma \ref{l:almost-subharm-v} to both $\psi$ and $-\psi$ so that $\psi$ satisfies mean value type inequalities in $B_r$. From this property, we can derive exactly as in the proof of the standard Harnack inequality for harmonic functions, see \cite[Theorem 2.5]{gilbarg1977elliptic}, the very same estimate
\begin{equation*}
   \sup_{B_\frac{r}{4}} \psi \leq C  \inf_{B_\frac{r}{4}} \psi,
\end{equation*}
provided that $r$ is bounded by a fixed constant, namely $r <1$. Moreover, the same mean value type inequalities also imply that
\begin{equation}\label{eq full comparison}
  \frac{1}{C}\fint_{\partial B_r} \psi \,d\mathcal{H}^{n} \leq \sup_{B_\frac{r}{4}} \psi \leq C  \inf_{B_\frac{r}{4}} \psi \leq C \fint_{\partial B_r} \psi \,d\mathcal{H}^{n}.
\end{equation}

For $x\in B_r$, let us consider the barrier $b(x) = m[e^{-\alpha\frac{|x|}{r}} -e^{ - \alpha}]$ with $m= \inf_{B_\frac{r}{4}} \psi$ and $\alpha>0$ to be determined. Clearly $b=0\leq \psi$ on $\partial B_r$ and $b \leq m\leq \psi$ on $\partial B_\frac{r}{4}$. Moreover, computing directly in spherical coordinates, on $B_r \setminus B_\frac{r}{4}$ we have
\begin{eqnarray*}
    \Delta b(x) = \frac{\alpha}{r^2} me^{-\alpha\frac{|x|}{r}}\left( \alpha -\frac{r{n}}{|x|}\right)\geq \frac{\alpha}{r^2} me^{-\alpha\frac{|x|}{r}}\left( \alpha -\frac{{n}}{4}\right) \geq  \frac{\alpha}{r^2} b(x),
\end{eqnarray*}
provided $\alpha \geq \frac{{n}}{4}+1$. Taking in addition $\alpha \geq K= {\frac{1}{2}}\max_{t\in [0,1]} |G'(t)|$, we deduce that $ \Delta (b -\psi) \geq  K(b-\psi)$ in $B_r \setminus B_\frac{r}{4}$. The maximum principle thus implies that $\psi(x) \geq b(x)$ in $B_r \setminus B_\frac{r}{4}$, which combined with \eqref{eq full comparison} and a Taylor expansion implies that for $x \in B_r \setminus B_\frac{r}{4}$ we have
\begin{equation*}
    \psi(x) \geq  \frac{r-|x|}{Cr} \fint_{\partial B_r} \psi \, d\mathcal{H}^{n}, 
\end{equation*}
{where $C$ is a new constant, not relabeled, additionally depending on our fixed choice of $\alpha$.} On the other hand, for $x \in B_\frac{r}{4}$,  \eqref{eq full comparison} guarantees $\psi(x) \geq {C^{-2}}\fint_{\partial B_r} \psi \,d\mathcal{H}^{n}$. Recalling that $\psi|_{\partial B_r} = w$, this proves the claim.\\




Applying \eqref{eq lbGR} with $x = (r-r_\xi) z + r_\xi \xi$ and {noticing that $|x| \leq \frac{1}{2}(r-r_\xi) +r_\xi$} yields
$$\psi_z(r_\xi \xi) = \psi \big((r - r_\xi)z + r_\xi \xi \big) \geq \frac{1}{C} \frac{r - r_\xi}{r} \fint_{\partial B_r} w \, d\mathcal{H}^{n} = \frac{1}{C} \frac{r - r_\xi}{r} \fint_{\partial B_r} w_z \, d\mathcal{H}^{n}\,.$$

Combining this with \eqref{e:Alt-Caff-psi-L2-bd}, we have
\[
\frac{r - r_\xi}{r^2} \left( \int_{\partial B_r} w \, d\mathcal{H}^{n} \right)^2 \leq C \int_{r_\xi}^r  |\nabla (\psi_z - w_z)|^2(\rho \xi) \, d\rho.
\]
Integrating over \( \xi \in S_z \subset \partial B_1 \), we obtain the inequality
\[
\int_{S_z} \frac{r - r_\xi}{r^2} d\xi \left(\int_{\partial B_r} w \, d\mathcal{H}^{n} \right)^2 \leq C \int_{\partial B_1} \int_{r_\xi}^r |\nabla (\psi_z - w_z)|^2 (\rho \xi)\, d\rho \, d\xi.
\]

And, by the estimate that \( r/8 \leq r_\xi \leq r \), we have
\begin{eqnarray*}
\frac{1}{r^2} |\{w = 0\} \cap B_r \setminus B_{r/4}(rz)| \left( \int_{\partial B_r} w \, d\mathcal{H}^{n} \right)^2 &\leq& C  \int_{B_r} |\nabla (\psi_z - w_z)|^2 dx\\
&\leq& C \int_{B_r} |\nabla (\psi - w)|^2 dx.
\end{eqnarray*}
\end{proof}

We continue with two elementary non-degeneracy bounds for the frequency of complementary almost subharmonic functions.

\begin{lemma}\label{lemma lowerbound}
   Let $r>0$. There exist universal constants $\kappa_0$ and $C>0$ with the following property: if {$G$ satisfies \eqref{eq:g assumptions},} $\kappa \in (0, \kappa_0]$, and $w_1, w_2$ are two non-negative functions in $W^{1,2}(B_r)$ satisfying
      \begin{equation}\label{eq subsol comp}
          2\Delta w_i \geq G'(w_i),
      \end{equation}
    \begin{equation}\label{eq comp ball}
    w_1\,w_2 = 0  \mbox{ $\mathcal{L}^{n+1}$ a.e. in $B_r$},
    \end{equation}
     \begin{equation}\label{eq nontrivial}
    \int_{\partial B_r} w_1\, d\mathcal{H}^{n} \int_{\partial B_r} w_2\, d\mathcal{H}^{n} >0,
    \end{equation}
    and
      \begin{equation}\label{eq nontrivial'}
 \frac{\int_{ B_r} |\nabla w_1|}{ \int_{\partial B_r} w_1 \, d\mathcal{H}^{n}}\leq \kappa\,,
    \end{equation}
    then
\begin{equation}\label{eq partition sphere}
     \frac{r\int_{ B_r} |\nabla w_2|^2}{ \int_{\partial B_r} w_2^2\, d\mathcal{H}^{n}} \geq C.
\end{equation}
\end{lemma}
\begin{proof}
    By scale invariance of the statement, it suffices to prove it for $r=1$. Arguing by contradiction, we assume the existence of a sequence of non-negative pairs of functions $(w_{1,k}, w_{2,k})$ satisfying \eqref{eq subsol comp}-\eqref{eq nontrivial} such that 
      \begin{equation}\label{eq bound u1}
       \int_{ B_1} |\nabla w_{1,k}|  \leq \frac{1}{k},
    \end{equation}
    \begin{equation}\label{eq bound u2}
       \int_{ B_1} |\nabla w_{2,k}|^2  \leq \frac{1}{k},
    \end{equation}
    with  $\Vert w_{1,k}\Vert_{L^1(\partial B_1)}=1$ and $\Vert w_{2,k}\Vert_{L^2(\partial B_1)}=1$. On the other hand, since both functions are non-negative and satisfy \eqref{eq subsol comp},  we have
       \begin{equation}\label{eq l1 bound}
        \int_{B_1} w_{1,k} \leq   C\int_{\partial B_1} w_{1,k}\, d\mathcal{H}^{n},
    \end{equation}
    \begin{equation}\label{eq l2 bound}
        \int_{B_1} w_{2,k}^2 \leq   C\int_{\partial B_1} w_{2,k}^2\, d\mathcal{H}^{n}.
    \end{equation}
   
    So, combining \eqref{eq bound u1},\eqref{eq bound u2}, \eqref{eq l1 bound}, and \eqref{eq l2 bound} we have that, up to a subsequence, $w_{k,1} \rightharpoonup w_1$ weakly in $W^{1,1}(B_1)$ and $w_{k,2}\rightharpoonup w_2$ weakly in $W^{1,2}(B_1)$ with $\Vert \nabla w_1\Vert_{L^1(B_1)}=0$ and $\Vert \nabla w_2\Vert_{L^2(B_1)}=0$. Additionally, by compactness of the trace operator $T:W^{1,2}(B_1)\to L^2(\partial B_1)$, we have that $\Vert w_2\Vert_{L^2(\partial B_1)}=1$, implying that $w_2=c_2$ is a positive constant. On the other hand, by Poincar\'e's inequality, setting $a_k = \omega_{n+1}^{-1}\int_{B_1} w_{1,k}$, we have that
    \begin{equation}\label{eq poincare}
        \int_{B_1} |w_{1,k} -a_k| \leq C\int_{B_1} |\nabla w_{{1},k}|\to 0.
    \end{equation}
    Combining this last observation with \eqref{eq l1 bound} we deduce that, up to taking a further subsequential limit, $w_{1,k}\to c_1 \in \mathbb{R}$ in $L^1(B_1)$. Let us notice that $c_1 > 0$, otherwise, $a_k \to 0$ as $k\to \infty$; in this latter case  $w_{k,1}\to 0$ strongly in $W^{1,1}(B_1)$, which yields a contradiction based on the fact that the trace operator $T:W^{1,1}(B_1)\to L^1(\partial B_1)$ is continuous whereas $\int_{\partial B_1} w_{1,k}\, d\mathcal{H}^{n}=1$ for all $k\in \mathbb{N}$.
    Combining the previous considerations with Rellich's theorem we have that $w_{i,k}\to c_i>0$ strongly in $L^1(B_1)$ for $i=1,2$ and, therefore, by \eqref{eq comp ball} we have that $c_1 c_2 =0$ $\mathcal{L}^{n+1}$ a.e. in $B_1$, which gives the desired contradiction.
\end{proof}

\begin{lemma}\label{lemma complementary pairs}
     Let $r>0$. There exists a universal constant $\eta_0>0$ with the following property: if $G$ satisfies \eqref{eq:g assumptions} with $G'(1)=0$, $w_1, w_2$ are two non-negative functions in $W^{1,2}(B_r)$ satisfying
      \begin{equation}\label{eq subsol comp 2}
          2\Delta w_i \geq G'(w_i) 
      \end{equation}      
    \begin{equation}\label{eq comp ball 2}
    w_1\,w_2 = 0  \mbox{ $\mathcal{L}^{n+1}$ a.e. in $B_r$},
    \end{equation}
    and
     \begin{equation}\label{eq nontrivial 2}
    \int_{\partial B_r} w_1\, d\mathcal{H}^{n} \int_{\partial B_r} w_2\, d\mathcal{H}^{n} >0\,,
    \end{equation}
   then 
\begin{equation}\label{eq partition sphere 2}
     \frac{r\int_{ B_r} |\nabla w_1|^2}{ \int_{\partial B_r} w_1^2\, d\mathcal{H}^{n}}+ \frac{r\int_{ B_r} |\nabla w_2|^2}{ \int_{\partial B_r} w_2^2\, d\mathcal{H}^{n}} \geq \eta_0.
\end{equation}
\end{lemma}
\begin{proof}
    The proof is a simpler version of the one for Lemma \ref{lemma lowerbound}. Taking $r=1$ (by scale invariance) and arguing by contradiction, we assume the existence of a sequence of pairs of non-negative functions $(w_{1,k}, w_{2,k})$ satisfying \eqref{eq subsol comp 2}-\eqref{eq nontrivial 2} such that 
      \begin{equation}\label{eq bound pair}
       \int_{ B_1} |\nabla w_{1,k}|^2 + \int_{ B_1} |\nabla w_{2,k}|^2  \leq \frac{1}{k},
    \end{equation}
    with  $\Vert w_{1,k}\Vert_{L^1(\partial B_1)}=1$ and $\Vert w_{2,k}\Vert_{L^2(\partial B_1)}=1$. On the other hand, since both functions satisfy \eqref{eq subsol comp 2} and are non-negative, we have
       \begin{equation}\label{eq general bound}
        \int_{B_1} w_{i,k}^2 \leq   C\int_{\partial B_1} w_{i,k}^2 \, d\mathcal{H}^{n},
    \end{equation}
    for $i=1,2$.
    So, combining \eqref{eq bound pair} and \eqref{eq general bound} we have that, up to a subsequence, $w_{i,k} \rightharpoonup w_i$ weakly in $W^{1,2}(B_1)$ with $\Vert \nabla w_i\Vert_{L^2(B_1)}=0$ for $i=1,2$. Additionally, by compactness of the trace operator $T:W^{1,2}(B_1)\to L^2(\partial B_1)$, we have that $\Vert w_i\Vert_{L^2(\partial B_1)}=1$, implying that $w_i=c_i>0$ are positive constants for $i=1,2$. Combining this last observation  with Rellich's theorem we have that $w_{i,k}\to c_i>0$ strongly in $L^2(B_1)$ for $i=1,2$ and, therefore, by \eqref{eq comp ball 2} we have that $c_1 c_2 =0$ $\mathcal{L}^{n+1}$ a.e. in $B_1$, which gives the desired contradiction.
\end{proof}

From the previous two lemmas, we have reduced the problem of obtaining a lower frequency bound in $B_r(x_0)$ to finding a decomposition $v=w_1+w_2$ into non-negative complementary functions $w_1$ and $w_2$ satisfying $\Delta w_i \geq G'(w_i)$ in $B_r(x_0)$ for $i=1,2$, and such that the $L^2$ norms of $w_1$ and $w_2$ are comparable. Loosely speaking, this amounts to ruling out the existence of one essentially connected component of $\{v>0\}\cap B_r(x_0)$ that is much larger than all the others as $r\to 0^+$. In order to rule out this possibility, we exploit the minimality of $u=1-v$ and make use of the cup competitor introduced in Definition \ref{d:cup-comp}. The next lemma provides a basic energy estimate obtained by comparing minimizers with their cup competitors.
\begin{lemma}\label{lemma:cupcomparison}
    Let $u\in W^{1,2}(\Omega;[0,1])$ be a minimizer of  \eqref{new model intro 2} or \eqref{new model volume constrained 2} and set $v=1-u$. Let $x_0 \in \{v=0\}$ and let $B_{r_0}(x_0) \subset \subset \Omega$. Given $r<r_0$, consider a essential connected component $D$ of the essential partition of $\{v>0\}\cap B_r(x_0)$ and consider $w_2 = \mathbf{1}_{D}v$ and $w_1=v-w_2$. Then, there exists a universal constant $\theta$ such that if $r \in (0, \min\{r_0,\theta\})$ we have that
    \begin{eqnarray}\label{eq cont4}
   \fint_{\partial B_\frac{r}{2}}w_2 \,d\mathcal{H}^{n}  \int_{B_\frac{r}{2}} |\nabla w_1| \,dx &\leq&   C\Big(\Vert w_1\Vert_{L^\infty(B_r)} + r^2 \Vert w_2\Vert_{L^\infty(B_\frac{r}{2})}\Big) \int_{ \partial B_\frac{r}{2}}w_1\,dx. 
\end{eqnarray}
\end{lemma}
\begin{proof}
By translation invariance we can assume without loss of generality $x_0 =0$. Set $G(t) = F(1-t)$ if $u$ minimizes \eqref{new model intro 2}, and $G(t) = F(1-t)-\lambda V(1-t)$ where $\lambda$ is as in Corollary \ref{corollary qdetachment} if $u$ minimizes \eqref{new model volume constrained 2}.\\

Let $1-g'$ be a cup competitor for $u$ associated to $D$ in $B_r$ as in Definition \ref{d:cup-comp}, i.e.,  $g':=(w_2+  \min\{\varphi_1, {w_1}\}) \mathbf{1}_{B_r \setminus B_{r/2}} + \psi$ with $\psi=\psi_D$ given by the semilinear replacement constructed in Lemma \ref{lemma semilinearreplacement}, solving
\begin{equation}\label{eq harmonic replacement}
    \begin{cases}
        2 \Delta \psi =G'(\psi)\, \mbox{ in $ B_\frac{r}{2}$}\\
        \psi =w_2 \, \mbox{ on $\partial B_\frac{r}{2}$},
    \end{cases}
\end{equation}
and where the radial cutoff $\varphi_1$ is the one constructed in Lemma \ref{lemma semilinearcutoff} with $\varphi_1=1$ on $\partial B_r$. Let $\varphi$ be the cutoff given by Lemma \ref{lemma semilinearcutoff} with $\varphi = \Vert w_1\Vert_{L^\infty(B_r)} $ on $\partial B_r$. Again, thanks to Lemma \ref{lemma semilinearcutoff}, we have that $\varphi \leq \varphi_1$ and thus
\begin{equation}\label{e:g-new-competitor}
    g :=(w_2+  \min\{\varphi, {w_1}\}) \mathbf{1}_{B_r \setminus B_{r/2}} + \psi \leq g'.
\end{equation}
Since $g \in W^{1,2}_\loc(\Omega;[0,1])$ {and $u\in \Acal$}, \eqref{e:g-new-competitor} tells us that $1-g \in \Fcal$.\\

If $u$ is a minimizer for \eqref{new model intro 2}, then a direct energy comparison yields
\begin{equation}\label{eq:case 1.1}
    \int_{B_r} |\nabla v|^2 \leq \int_{B_r}|\nabla g|^2  + F(1-g)-F(1-v)\,dx \,.
\end{equation}
On the other hand, if $u$ minimizes \eqref{new model volume constrained 2}, we can exploit Corollary \ref{corollary qdetachment} to deduce
\begin{eqnarray}\notag
    \int_{B_r} |\nabla v|^2 &\leq& \int_{B_r}|\nabla g|^2  + F(1-g)-\lambda V(1-g)-(F(1-v)-\lambda V(1-v))\, dx\\ \label{eq:case 1.2}
    &&+ C\Vert v-g\Vert_{L^2(B_r)}^2.
\end{eqnarray}
In any case, we can combine the $C^{1,1}$ estimate \eqref{eq:g estimates} for $G$ with a Taylor expansion to deduce from \eqref{eq:case 1.1} or \eqref{eq:case 1.2} that
\begin{equation}\label{eq minquaderror}
    \int_{B_r} |\nabla v|^2 \leq\int_{B_r}|\nabla g|^2  + G'(g)(g-v)\,dx + C\Vert v-g\Vert_{L^2(B_r)}^2.
\end{equation}

Let $h = \min\{{\varphi}, w_1\}$. Since $w_1$ and $w_2$ have disjoint supports, we can rewrite \eqref{eq minquaderror} as
\begin{eqnarray}\notag
    \int_{B_\frac{r}{2}}|\nabla w_2|^2 + \int_{B_r}|\nabla w_1|^2 &\leq& \int_{B_\frac{r}{2}}|\nabla \psi|^2\,dx + \int_{B_r\setminus B_\frac{r}{2}}|\nabla h|^2dx\\ \label{eq decind}
   && +  \int_{B_r} G'(g)(g-v) \,dx + C\Vert g-v\Vert_{L^2(B_r)}^2.
\end{eqnarray}
Using the elementary identity  
\begin{equation}\label{eq EI}
    |\nabla (f_1-f_2)|^2= |\nabla f_1|^2-|\nabla f_2|^2+2\nabla f_2\cdot \nabla (f_2-f_1),
\end{equation}
we deduce from \eqref{eq decind}
\begin{eqnarray}\notag
    \int_{B_\frac{r}{2}}|\nabla (w_2-\psi)|^2 + \int_{B_\frac{r}{2}}|\nabla w_1|^2 &\leq&   \int_{B_r\setminus B_\frac{r}{2}} |\nabla h|^2-|\nabla w_1|^2 + G'(h)(h-v)\,dx \\\notag
    &&+\int_{B_\frac{r}{2}}2\nabla \psi\cdot \nabla (\psi-w_2)   + G'(\psi)(\psi-v)\,dx\\  \label{eq cont1}
    &&+ C\Vert g-v\Vert_{L^2(B_r)}^2.
\end{eqnarray}

Let us first analyze the first term on the right-hand side of \eqref{eq cont1}. Another application of \eqref{eq EI} together with the observation that the support of $h$ in $B_r$ is contained in $\mathrm{supp}(w_1)\cap(B_r\setminus B_{\frac{r}{2}})$ and integration by parts yields 
\begin{eqnarray}\notag
     &&\int_{B_r\setminus B_\frac{r}{2}}|\nabla h|^2-|\nabla w_1|^2+|\nabla (w_1-h)|^2 + G'(h)(h-v) \, dx   \\
     &&\qquad=  \int_{B_r\setminus B_\frac{r}{2}} 2\nabla h\cdot \nabla (h-w_1) +  G'(h)(h-v)\, dx \\\notag
     &&\qquad=  \int_{B_r\setminus B_\frac{r}{2}} 2\nabla \varphi\cdot \nabla (h-w_1)+ G'(\varphi)(h-v)\, dx\\\notag
     &&\qquad\leq 2 \int_{\partial B_\frac{r}{2}}  |\nabla \varphi| w_1 \,d\mathcal{H}^{n}\\\label{eq anulusterm}
     &&\qquad\leq \frac{C}{r}\Vert w_1\Vert_{L^\infty(B_r)} \int_{\partial B_\frac{r}{2}}  w_1 \,d\mathcal{H}^{n}.
\end{eqnarray}
where we have used {that $\vphi$ satisfies \eqref{eq semrcutoff}, that $h=w_1$ on $\partial B_r$ and $h=0$ on $\partial B_{\tfrac{r}{2}}$, and} that $\Vert \nabla \varphi \Vert_{L^\infty(B_\frac{r}{2})}\leq  \frac{C}{r}\Vert w_1\Vert_{L^\infty(B_r)}$. 

Now let us analyze the second term on the right-hand side of \eqref{eq cont1}. Since $\psi$ satisfies \eqref{eq harmonic replacement}, we can integrate by parts to deduce
\begin{eqnarray}\notag
    \int_{B_\frac{r}{2}}2\nabla \psi\cdot \nabla (\psi-w_2)   + G'(\psi)(\psi-v)\,dx &=&  - \int_{B_\frac{r}{2}} G'(\psi)w_1\,dx\\\label{eq LEHR}
    &\leq & Cr \Vert w_2\Vert_{L^\infty(B_\frac{r}{2})} \int_{ \partial B_\frac{r}{2}}w_1\,d\Hcal^n,
\end{eqnarray}
where in the last line we used the maximum principle for $\psi$ (see Lemma \ref{lemma semilinearreplacement}), as well as the almost-subharmonicity of $w_1$. The latter can be justified by the exact same reasoning as that in Lemma \ref{l:almost-subharm-v}, replacing the use of \eqref{eq:differential inequality} with the analogous PDE for $w_1$ given by Lemma \ref{eq lemma eq comp}.

By combining \eqref{eq cont1}, \eqref{eq anulusterm}, and \eqref{eq LEHR} we deduce
\begin{eqnarray}\notag
    \int_{B_\frac{r}{2}}|\nabla (w_2-\psi)|^2 + \int_{B_\frac{r}{2}}|\nabla w_1|^2 &\leq&   C\Big(\frac{1}{r}\Vert w_1\Vert_{L^\infty(B_r)} + r \Vert w_2\Vert_{L^\infty(B_\frac{r}{2})}\Big) \int_{ \partial B_\frac{r}{2}}w_1\, d\Hcal^n \\ \label{eq cont2}
    &&+ C\Vert g-v\Vert_{L^2(B_r)}^2 -  \int_{B_r\setminus B_\frac{r}{2}}|\nabla (w_1-h)|^2\,,
\end{eqnarray}

We conclude these preliminary estimates by bounding the error term $\Vert g-v\Vert_{L^2(B_r)}^2$. By decomposing $g-v$ and applying the Poincar\'e-trace inequality {(see e.g. \cite[Corollary 4.4.7]{Ziemer})} to each one of the resulting terms, we obtain
\begin{eqnarray}\notag
  \Vert g-v\Vert_{L^2(B_r)}^2 &\leq& 2\int_{B_\frac{r}{2}}  (w_2-\psi)^2 +w_1^2\,dx+\int_{B_r\setminus B_\frac{r}{2}}  (h-w_1)^2 \\\notag
    &\leq& Cr^2\int_{B_\frac{r}{2}}  | \nabla(w_2-\psi)|^2 + |\nabla w_1|^2\, dx   +  Cr^2\int_{B_r\setminus B_\frac{r}{2}}  | \nabla(w_1-h)|^2\\ \label{eq errorterms}
    &&+Cr\int_{\partial B_\frac{r}{2}} w_1^2 \,d\mathcal{H}^{n}\,.
\end{eqnarray}
By taking $\theta$ sufficiently small (and therefore also $r$), we can reabsorb the terms on the right-hand side of \eqref{eq errorterms} into those of \eqref{eq cont2} and deduce
\begin{eqnarray}\label{eq cont3}
    \int_{B_\frac{r}{2}}|\nabla (w_2-\psi)|^2 + \int_{B_\frac{r}{2}}|\nabla w_1|^2 &\leq&   C\Big(\frac{1}{r}\Vert w_1\Vert_{L^\infty(B_r)} + r \Vert w_2\Vert_{L^\infty(B_\frac{r}{2})}\Big) \int_{ \partial B_\frac{r}{2}}w_1\,d\Hcal^n\,. 
\end{eqnarray}
We next apply Cauchy-Schwarz, the generalized Alt-Caffarelli trace inequality of Lemma \ref{lemma ACI} and Young's inequality for products to deduce that
\begin{eqnarray}\notag
\frac{1}{C r} \fint_{\partial B_\frac{r}{2}}w_2 \,d\mathcal{H}^{n}  \int_{B_\frac{r}{2}} |\nabla w_1| \,dx &\leq& \frac{1}{C r} \fint_{\partial B_\frac{r}{2}}w_2 \,d\mathcal{H}^{n} \left[\int_{B_{\frac{r}{2}}} |\nabla w_1|^2\right]^{\frac{1}{2}} |\{w_1 > 0 \}\cap B_{\frac{r}{2}}|^{\frac{1}{2}} \\ \notag
&\leq& \frac{1}{C r^2} \Big(\fint_{\partial B_\frac{r}{2}}w_2 d\mathcal{H}^{n}\Big)^2 |\{w_1 > 0 \}\cap B_{\frac{r}{2}}|+\int_{B_{\frac{r}{2}}} |\nabla w_1|^2\\ \label{eq apAC}
&\leq&  \int_{B_\frac{r}{2}}|\nabla (w_2-\psi)|^2 + \int_{B_\frac{r}{2}}|\nabla w_1|^2\,.
\end{eqnarray}
Finally, by combining \eqref{eq cont3} with \eqref{eq apAC} we obtain \eqref{eq cont4}.
\end{proof}

The final preparatory result that we require for obtaining a lower frequency bound for $v$ is the following Poincar\'e inequality, where the usual average-free condition can be replaced with the presence of a sufficiently large zero set.
\begin{lemma}\label{lemma noncollapsing}
     Given any $\delta \in (0,1)$, there exists $b(\delta)$ such that if {$G$ satisfies \eqref{eq:g assumptions}} and $w$ is a non-negative function in $W^{1,2}(B_\frac{r}{2})$ satisfying $2\Delta w \geq G'(w)$ and $|\{w=0\}\cap B_\frac{r}{2}| \geq |B_\frac{r}{2}|(1-\delta)$, then
      \begin{equation}\label{eq lb small comp}
        \int_{\partial B_\frac{r}{2}} w^2\, d\mathcal{H}^{n} \leq  b(\delta)r \int_{ B_\frac{r}{2}} |\nabla w|^2\,.
      \end{equation}
\end{lemma}

{\begin{remark}
    Note that the constant $b(\delta)$ in Lemma \ref{lemma noncollapsing} will degenerate as $\delta \uparrow 1$.
\end{remark}}
\begin{proof}
   By scale invariance of the statement, we assume $r=1$. Arguing by contradiction, if the desired conclusion fails, we have a sequence of non-negative functions $w_k \in W^{1,2}(B_\frac{1}{2})$  with $\Vert w_k\Vert_{L^2(\partial B_\frac{1}{2})}=1$, satisfying $2\Delta w_k \geq G'(w_k)$
    \begin{equation}\label{e:fat-zero-set}
        |\{w_k = 0 \}\cap B_\frac{1}{2}| \geq \omega_{n+1}(1-\delta)
    \end{equation}
    for some $\delta \in(0,1)$, but such that
    \begin{equation*}
        \int_{B_\frac{1}{2}} |\nabla w_k|^2 \to 0\,.
    \end{equation*}
    Up to extracting a subsequence, we argue as in the proof of Lemma \ref{lemma lowerbound} to deduce that $w_k \rightharpoonup w_\infty$ weakly in $W^{1,2}(B_\frac{1}{2})$ for some non-negative function $w_\infty\in W^{1,2}(B_\frac{1}{2})$ with $\Vert \nabla w_\infty \Vert_{L^2(B_\frac{1}{2})} = 0$. Moreover, the continuity of the trace operator $T:W^{1,2}(B_\frac{1}{2}) \to L^2(\partial B_\frac{1}{2})$ guarantees that $\Vert w_\infty \Vert_{L^2(\partial B_\frac{1}{2})}=1$ and thus $w_\infty \equiv c_0$ for some constant $c_0 > 0$. In fact, Rellich-Kondrachov implies that we may improve the weak $L^2$ of convergence $w_k$ to $w_\infty = c_0$ to strong convergence in $L^2(B_\frac{1}{2})$.
    
    Thus, up to extracting a further subsequence, we have pointwise $\mathcal{L}^{n+1}$-a.e. convergence of $w_k$ to $c_0 > 0$, which is in contradiction with \eqref{e:fat-zero-set}. {Indeed, the latter implies that the intersection of the sets $\{w_k = 0 \}\cap B_{\frac{1}{2}}$ over $k$ must have positive $\Lcal^{n+1}$-measure.}
\end{proof}

\begin{lemma}[Lower frequency bound for minimizers]\label{lower bound frequency}
    Let $G$ satisfy \eqref{eq:g assumptions}, let $u=1-v$ be a minimizer of \eqref{new model intro 2} or \eqref{new model volume constrained 2}, and let $x_0 \in \Omega$ be such that $v(x_0)=0$. Then there exists a universal constant $\eta>0$ such that
    \begin{equation}\label{eq lower bound freq}
     \lim_{r\to 0^+} N_{v,x_0}(r)  \geq \eta\,.
    \end{equation}
\end{lemma}
\begin{proof}
We divide the proof into steps. By translation invariance we assume without loss of generality that $x_0=0$. Let us also assume that $r \in (0, \frac{2}{3} \min\{d, \theta\})$ with $\theta>0$ as in Lemma \ref{lemma:cupcomparison} {and $d = \dist(x_0,\partial\Omega)$}. Also, by renormalizing, we can assume that $\fint_{\partial B_\frac{3r}{2}} v^2 \,d\mathcal{H}^{n}=1$. This latter normalization combined with Lemma \ref{lemma: uc take 2} {and the estimate \eqref{eq lp equivalence} from its proof} implies that for any $\delta\in (0,1]$ we have
\begin{equation}\label{eq comparability norms}
   \frac{1}{C(\delta)}\leq  \rho^\frac{-(n+1)}{p}\Vert v\Vert_{L^p(B_{\rho})},  \Vert v\Vert_{L^p(\partial B_{\rho})} \leq C(\delta)
\end{equation}
for $\rho \in (\delta r, (3-\delta)r)$ and $p \in [1,\infty]$.
\medskip

\noindent{\it Step 1:} We show that either \eqref{eq lower bound freq} is valid or that we can decompose $v = w_1+w_2$, with $w_1, w_2 $ non-negative functions in $W^{1,2}(B_\frac{3r}{2})$ where $w_2= v|_{D}$ for one of the connected components of $B_\frac{3r}{2}\cap \{v>0\}$ constructed in Lemma \ref{lemma connected 2}, satisfying
  \begin{equation}\label{eq nontrivial 3}
    \int_{ \partial B_\frac{r}{2}} w_2\, d\mathcal{H}^{n} \geq \frac{1}{C}\,.
    \end{equation}
   
   Let $v_i=\mathbf{1}_{C_i}$, for $\{C_i\}_{\{i=1\}}^\infty$ be the essential connected components of $B_\frac{3 r}{2}\cap \{v>0\}$  as in Lemma \ref{lemma connected 2}. If for every sequence $r_k \to 0$ as $k\to \infty$ we have $|\{v_i =0\}\cap B_\frac{r_k}{2}| \geq {\frac{1}{2}|B_{\frac{r_k}{2}}|}$ for every {$i\geq 1$}, we deduce from Lemma \ref{lemma noncollapsing} that
\begin{equation*}
     \int_{\partial B_\frac{r_k}{2}} v_i^2 \,d\mathcal{H}^{n} \leq  Cr_k\int_{ B_\frac{r_k}{2}} |\nabla v_i|^2 \qquad \forall \, i = 1,\dots,k\,,
\end{equation*}
which, after adding up over $i$, and taking limit in $k$ yields \eqref{eq lower bound freq}.\\

Thus, we just have to consider the case where $|\{v_i =0\}\cap B_\frac{r}{2}| \leq \frac{1}{2}{|B_{\frac{r}{2}}|}$ for some $i$ and for $r \in (0, r_0)$ for some $r_0>0$ sufficiently small. Set $w_2= v_i$ for this choice of $i$. Notice now that if $\frac{1}{r^{n-1}}\int_{B_{\frac{3r}{2}}}  |\nabla v|^2 \geq \varepsilon$ for some universal $\varepsilon>0$, we immediately obtain \eqref{eq lower bound freq}, in light of our normalization. Suppose instead that $\frac{1}{r^{n-1}}\int_{B_{\frac{3r}{2}}} |\nabla v|^2 <\varepsilon$, for some $\varepsilon>0$ to be determined. By Poincar\'e's inequality, we then obtain
\begin{equation}\label{eq average control}
    \varepsilon > \frac{1}{r^{n-1}}\int_{B_{\frac{r}{2}}} |\nabla v|^2 \geq  \fint_{B_\frac{r}{2}} (v-A)^2\geq \frac{1}{|B_\frac{r}{2}|} \int_{B_\frac{r}{2}\cap \{w_2>0\}} (w_2-A)^2,
\end{equation}
where $A = \fint_{B_\frac{r}{2}} v$ satisfies $\frac{1}{C}\leq A \leq C$ in virtue of \eqref{eq comparability norms}. Since $|\{w >0\}\cap B_\frac{r}{2}| \geq \frac{1}{2}|B_{r/2}|$, we have that for $\varepsilon$ small enough, \eqref{eq average control} implies
\begin{equation*}
 \fint_{B_\frac{r}{2}} w_2^2 \geq \frac{1}{C}\,.
\end{equation*}
Combining this with the almost subharmonicity of $w_2$ (cf. the proof of Lemma \ref{lemma:cupcomparison}), we deduce
\begin{equation}
     C\fint_{\partial B_\frac{r}{2}} w_2 \geq  \Vert w_2\Vert_{L^\infty(\partial B_\frac{r}{2})} \fint_{\partial B_\frac{r}{2}} w_2\geq  \fint_{\partial B_\frac{r}{2}} w_2^2\geq \frac{1}{C}
\end{equation}
where in the first inequality we have used \eqref{eq comparability norms} to guarantee the universal $L^\infty$ bound on $w_2$.

\medskip
\noindent{\it Step 2:} In this step, we show that, for $\kappa_0$ as in Lemma \ref{lemma lowerbound}, {if the lower frequency bound \eqref{eq lower bound freq} does not hold,} there exists $r_1{=r_1(\kappa_0)} \in (0, r_0)$ such that for any $r \in (0, r_1)$  one of the following two properties holds:
\begin{enumerate}
    \item[$O_1)$] 
    \begin{equation}\label{eq upper bound components}
    \int_{ B_\frac{r}{2}} |\nabla w_1| \leq \kappa_0 \int_{\partial B_\frac{r}{2}} w_1\, d\mathcal{H}^{n}\,.
\end{equation}

\item[$O_2)$] 
    \begin{equation}\label{eq comparability components}
  M_0\fint_{\partial B_\frac{3 r}{2}} w_1^2 d\mathcal{H}^{n} \geq \Big(\fint_{\partial B_{\frac{r}{2}}} w_2 \, d\Hcal^n\Big)^2
\end{equation}
for some $M_0{= M_0(\kappa_0)} >0$\,.
\end{enumerate}
{Given this dichotomy, we will proceed to show the validity of \eqref{eq lower bound freq} in Step 3 in either case. Notice that property $O_1$ is an $L^1$-version of an upper frequency bound for $w_1$ and puts us in good shape to apply Lemma \ref{lemma lowerbound}, while $O_2$ (with \eqref{eq comparability norms}) provides $L^p$-comparability of $w_1$ and $w_2$ on $\partial B_{\tfrac{r}{2}}$, which will allow us to make use of the estimate from Lemma \ref{lemma complementary pairs} to obtain the desired lower frequency bound.} 

Our strategy here to prove this dichotomy is to show that if {\eqref{eq nontrivial 3} holds (as a consequence of the failure of \eqref{eq lower bound freq}),} and \eqref{eq comparability components} does not hold, i.e.,
  \begin{equation}\label{eq neagtion comparability}
    \fint_{\partial B_\frac{3r}{2}} w_1^2 \,d\mathcal{H}^{n}\leq \frac{1}{M_0} \Big(\fint_{\partial B_{\frac{r}{2}}} w_2\Big)^2,
\end{equation}
then \eqref{eq upper bound components} necessarily holds.\\

Let us start remarking that by the almost-subharmonicity of $w_1$, we can deduce from the estimate \eqref{eq neagtion comparability} that 
 \begin{equation}\label{eq v2 dominates v1 pt3}
   \fint_{ \partial B_\frac{r}{2}} w_2 \,d\mathcal{H}^{n} \geq \frac{\sqrt{M_0}}{C}\Vert w_1\Vert_{L^\infty(B_r)}\,.
\end{equation}
Moreover, from \eqref{eq comparability norms} and \eqref{eq nontrivial 3}, we have that
\begin{equation}\label{eq domw2}
     \fint_{ \partial B_\frac{r}{2}} w_2 \,d\mathcal{H}^{n} \geq \frac{1}{C}\fint_{ \partial B_\frac{r}{2}} v \,d\mathcal{H}^{n} \geq \frac{1}{C}\Vert w_2\Vert_{L^\infty(B_\frac{r}{2})}\,.
\end{equation}
On the other hand, by Lemma \ref{lemma:cupcomparison} we have that
  \begin{eqnarray*}
   \fint_{\partial B_\frac{r}{2}}w_2 \,d\mathcal{H}^{n}  \int_{B_\frac{r}{2}} |\nabla w_1| \,dx &\leq&   C\Big(\Vert w_1\Vert_{L^\infty(B_r)} + r^2 \Vert w_2\Vert_{L^\infty(B_\frac{r}{2})}\Big) \int_{ \partial B_\frac{r}{2}}w_1\,d\Hcal^n,
\end{eqnarray*}
which, combined with \eqref{eq v2 dominates v1 pt3} and \eqref{eq domw2}, implies
 \begin{eqnarray*}
   \int_{B_\frac{r}{2}} |\nabla w_1| \,dx &\leq&   C\Big(\frac{1}{\sqrt{M_0}} + r^2 \Big) \int_{ \partial B_\frac{r}{2}}w_1\,d\Hcal^n\,,
\end{eqnarray*}
so, by taking $r_1 < {\frac{\kappa_0}{2\sqrt{C}}}$ and $\sqrt{M_0} \geq \frac{C}{2\kappa_0}$, we deduce \eqref{eq upper bound components}.

\medskip
\noindent{\it Step 3:} We finish the proof by showing that either option in the dichotomy proved in Step 2 leads to \eqref{eq lower bound freq}.\\
We again work under the assumption that \eqref{eq lower bound freq} fails, so that we can make use of \eqref{eq nontrivial 3} from Step 1. Before considering each alternative separately, notice that from  \eqref{eq comparability norms}, \eqref{eq nontrivial 3}, and Jensen's inequality
we have that
\begin{equation}\label{eq comparability with v}
     \fint_{\partial B_\frac{r}{2}} w_2^2 \, d\mathcal{H}^n \geq \frac{1}{C}  \Big(\fint_{\partial B_\frac{r}{2}} w_2 d \mathcal{H}^{n}\Big)^2 \,
     \geq \frac{1}{C} \, = \frac{1}{C} \fint_{\partial B_\frac{3r}{2}} v^2 \, d\mathcal{H}^{n}.
\end{equation}

If $O_1$ holds, then we can invoke Lemma \ref{lemma lowerbound}  to deduce that 
\begin{equation}\label{eq option 1}
    \int_{\partial B_\frac{r}{2}} w_2^2 \, d\mathcal{H}^{n} \leq  \tilde{C}r\int_{ B_\frac{r}{2}} |\nabla w_2|^2\,.
\end{equation}

Hence, \eqref{eq option 1} and \eqref{eq comparability with v} altogether imply
\begin{equation}\label{eq lbf}
  \int_{\partial B_\frac{3r}{2}} v^2 \, d\mathcal{H}^{n} \leq  \tilde{C}r\int_{ B_\frac{3r}{2}} |\nabla v|^2\,,
\end{equation}
which is the desired lower frequency bound.

If instead $O_2$ holds, \eqref{eq comparability components} and  \eqref{eq comparability with v} implies
\begin{equation}\label{eq domination 1}
    \int_{\partial B_\frac{3r}{2}} w_1^2 \, d\mathcal{H}^{n} \geq \frac{1}{CM_0} = \frac{1}{CM_0} \int_{\partial B_{\frac{3r}{2}}} v^2 \, d\mathcal{H}^n\,.
\end{equation}

So, Lemma \ref{lemma complementary pairs}, \eqref{eq comparability with v}, and \eqref{eq domination 1} imply
\begin{eqnarray*}\notag
     \eta_0 &\leq& C\frac{r\int_{ B_{\tfrac{3r}{2}}} |\nabla w_1|^2}{ \int_{\partial B_{\tfrac{3r}{2}}} w_1^2\, d\mathcal{H}^{n}}+ \frac{r\int_{ B_{\tfrac{3r}{2}}} |\nabla w_2|^2}{ \int_{\partial B_{\tfrac{3r}{2}}} w_2^2\, d\mathcal{H}^{n}} \\
     &\leq& Cr\int_{ B_{\frac{3r}{2}}} |\nabla v|^2\Big(\frac{M_0}{ \int_{\partial B_{\frac{3r}{2}}} v^2\, d\mathcal{H}^{n}}+ \frac{1}{ \int_{\partial B_\frac{3r}{2}} v^2\, d\mathcal{H}^{n}}\Big) \\ \label{eq option 2}
     &=& C\int_{ B_\frac{3r}{2}} |\nabla v|^2\,.
\end{eqnarray*}
In either case we deduce \eqref{eq lbf} from which we obtain the result by sending $r$ to zero.
\end{proof}

We are now in a position to prove our main regularity result. The following corollary of Lemma \ref{lower bound frequency} is a more precise statement of Theorem 
\ref{thm:main regularity theorem}.(i).

\begin{corollary}[Lipschitz regularity for minimizers]\label{corollary higher order regularity}
	Let $\wire=\rn \setminus \Om$ be compact and let $\mathcal{C}$ be a spanning class for $\wire$ satisfying \eqref{eq:spanning class assumption}. Suppose that $F$, $V$ satisfy \eqref{eq:A1}-\eqref{eq:A3}, and $u$ {is a minimizer for \eqref{new model intro 2} or \eqref{new model volume constrained 2}, respectively}. 
	%
	%
    %
    %
    Then, given $\e_0>0$ there are $C=C(\e_0,u,n)$ and $r_{**}=r_{**}(n,G)$ such that for any $x_0 \in \Omega_{\e_0}$ and $r< \min\{r_{**},\e_0/3\}$, 
		\begin{equation}\label{eq uniform holder}
			r[u]_{\Lip(B_{r/2}(x_0))}\leq C \Big(\frac{1}{r^{n-1}}\int_{B_{r}(x_0)}|\nabla u|^2\Big)^\frac{1}{2}\,,
		\end{equation}
  where $\Omega_{\e_0}=\{x\in \Om : \dist(x,\partial \Om) \geq \e_0\}$. {As a consequence, $u$ is a minimizer for \eqref{new model intro} or \eqref{new model volume constrained}, respectively.}
	
\end{corollary}

\begin{proof}
	
	\noindent{\it To prove \eqref{eq uniform holder}}: {{Recall from Corollary \ref{c:EL-v} that} \eqref{eq:g assumptions} and \eqref{eq modified iv}-\eqref{eq modified ineq} hold for $G$ as defined therein and $v=1-u$, respectively. We would like to prove \eqref{eq uniform holder} by applying Theorem \ref{thm:holder reg for crit points}.(ii) to $v$ on compact subsets of $\Om'_{\e_0}$, which we define to be those $x\in \Om_{\e_0}$ such that $v$ does not vanish identically on the connected component of $\Om$ containing $x$. {In light of Theorem \ref{thm:holder reg for crit points}.i,} this boils down to showing {two things: first, that the local lower frequency bound \eqref{eq:lower freq bound assumption theorem 2.1} holds for $v$; and second,} that there exists $M$ such that the upper frequency bound \eqref{eq:freq bound for thm 3.1} holds uniformly on $\Om_{\e_0/2}'$ (which is defined analagously to $\Om_{\e_0}'$), independently of its subcomponents. {The lower frequency bound is a consequence of Lemma \ref{lower bound frequency}.} Towards verifying the uniform upper frequency bound, notice that by Theorem \ref{thm:holder reg for crit points}.ii, it holds if $\nabla v\in L^2$ and $\mathcal{L}^{n+1}(\{v<t\})<\infty$ for all $t\in (0,1)$. Now if $u$ minimizes \eqref{new model intro 2} or \eqref{new model volume constrained 2}, we have $\nabla u=\nabla (1-v)\in L^2(\Om)$. Furthermore, since $u\to 0$ uniformly as $|x|\to \infty$ in \eqref{new model intro 2}, it is immediate that $\mathcal{L}^{n+1}(\{v<t\})=\mathcal{L}^{n+1}(\{u>1-t\})<\infty$ for all $t\in (0,1)$. Similarly, in \eqref{new model volume constrained 2}, since $\int V(u)=1$ and $V>0$ on $(0,1]$, it must be the case that $\mathcal{L}^{n+1}(\{v<t\})=\mathcal{L}^{n+1}(\{u>1-t\})<\infty$ for all $t\in (0,1)$. So indeed we may apply Theorem \ref{thm:holder reg for crit points}.i to obtain \eqref{eq uniform holder}.

\medskip
    
    \noindent{{\it Minimality of $u$ for \eqref{new model intro} or \eqref{new model volume constrained}}: If $u$ is minimizing for \eqref{new model volume constrained 2}, then \eqref{eq uniform holder} implies that it is continuous and thus admissible for \eqref{new model volume constrained}. Since the admissible class for \eqref{new model volume constrained 2} contains that of \eqref{new model volume constrained} (see Lemma \ref{l:gen-adm-class-larger}), this shows that $u$ is minimizing for \eqref{new model volume constrained}. On the other hand, suppose that $u$ is minimizing for \eqref{new model intro 2}. The continuity estimate \eqref{eq uniform holder} for $u$ implies that it is continuous and decays uniformly to zero at infinity, and is thus admissible for \eqref{new model intro}. It remains to show that it is a minimizer of \eqref{new model intro}.}  This would follow from proving that any admissible $w$ for \eqref{new model intro} with finite energy belongs to the space $D_n^{1,2}(\Om;[0,1])$, {thus also verifying the containment of admissible classes in this case}. When $n\geq 2$, this can be deduced from \eqref{eq:alter char of h1dot} and the fact that $w\in L^{2(n+1)/(n-1)}(B_R^c)$, which is a consequence of $\nabla w\in L^2(B_R^c;[0,1])$ (for $R$ such that $\wire \cc B_R$) and the pointwise decay to zero at infinity of $w$ (see for example \cite[Theorem II.6.1]{Galdi} for a proof of the integrability of $u$ under these two assumptions). {When $n=1$, the uniform decay to zero of $w$ at infinity implies that $\mathcal{L}^2(\{w>t\})<\infty$ for all $t\in (0,1)$, and so $w\in D_1^{1,2}(\Om;[0,1])$.}} 
\end{proof}

\subsection{Existence of Minimizers: Proof of Theorem \ref{thm:main existence theorem}}\label{s:existence}

For Theorem \ref{thm:main existence theorem}, we will need some basic information regarding the auxiliary variational problem
\begin{equation}\label{eq:isop problem def}
\Psi(v_0) = \inf\Big\{\int_{\mathbb{R}^{n+1}}|\nabla u|^2 + F(u)\,dx : u\in W^{1,2}(\R^{n+1};[0,1]), \,\,\int_{\rn}V(u) \, dx = v_0 \Big\}\,.
\end{equation}
 This problem was introduced in \cite{maggirestrepo} with the volume potential $V$ as in \eqref{eq:v def}, and quantitative stability and Alexandrov-type ridigity were established. Here we will only need the existence of {positive} minimizers for \eqref{eq:isop problem def}. The proof is in Appendix \ref{sec: app c} and follows \cite[Theorem A.1]{maggi2023hierarchy}.

\begin{theorem}[Existence of radial isoperimetric functions on $\rn$]\label{thm:existence for isop problem}
If $v_0>0$ and $F$ and $V$ are continuous, non-negative functions such that $F(0)=0=V(0)$ and
\begin{equation}\label{eq:A4 redux}
\lim_{t\to 0^+}\frac{V(t)}{F(t)}=0\,,
\end{equation}
then there exists a {strictly positive,} radial, decreasing minimizer $x\mapsto w(|x|)$ for $\Psi(v_0)$.
\end{theorem}

{The first step towards proving Theeorem \ref{thm:main existence theorem} is the weak closure of the space $\Acal$ of admissible functions. This will allow us to use the direct method to obtain existence of minimizers for \eqref{new model intro 2} or \eqref{new model volume constrained 2}, which, combined with Corollary \ref{corollary higher order regularity}, will yield the desired result.}

\begin{proposition}\label{p:weak-closure-A}
    $\Acal$ is weakly closed in $W^{1,2}(\Omega;[0,1])$.
\end{proposition}

\begin{proof}
Let $\{u_k\}_k$ be a sequence of functions in $\Acal$ converging weakly in $W^{1,2}(\Omega;[0,1])$ to $u$, and set $v_k = 1-u_k$ and $v = 1-u$. {By Theorem \ref{thm:prelim compactness thm}, $\{u^* \geq t\}$ is $\mathcal{C}$-spanning $\wire$ for every $t\in (1/2,1)$, so $u \in \Fcal$.} Let $C_i$ be an essentially connected component of $B_r(x_0) \cap \{v^* >0\}$ with corresponding cup competitor $g_i$ as in Definition \ref{d:cup-comp}. We must show that $\{g_i^* \leq t\}$ is $\mathcal{C}$-spanning for all $t\in (0,1/2)$. The proof is divided into steps.

\medskip

\noindent{\it Step 1:} Here we show that up to taking a subsequence in $k$, there are essentially connected components $C_{i(k)}^k$ of $\{v_k^*>0\}$ such that
\begin{equation}\label{eq:liminf ci equation}
\mathbf{1}_{C_i} \leq \liminf_{k\to \infty} \mathbf{1}_{C_{i(k)}^k}   \qquad \mbox{a.e.} 
\end{equation}
Let $T_k\subset (0,1)$ and $T\subset (0,1)$ be the measure one sets associated to $u_k$ and $u$ respectively, as guaranteed in Lemma \ref{lemma connected 2}. Fix $t_j\searrow 0$ such that $\{t_j\}_j \subset T$. Let $C_i=\lim_{j\to \infty} F_j^i$ as in Lemma \ref{lemma connected 2}, so that $\mathbf{1}_{F_j^i}$ increases almost everywhere to $\mathbf{1}_{C_i}$. It is thus enough to find components $C_{i(k)}^k$ (independent of $j$) such that $\mathbf{1}_{F_j^i} \leq \liminf_{k\to\infty} \mathbf{1}_{C_{i(k)}^k}$ $\Lcal^{n+1}$-a.e.~ for every $j$. 

To do so, first we observe that for every $t_j$, $\sup_k \int_{\{t_{j+1}\leq v_k\leq t_j \}}{|\nabla v_k|^2} < \infty$. Thus by the coarea formula, there is a subsequence (not relabeled) and 
\[
    \{\beta_k^1\}_k \subset (t_2,t_1)\cap \bigcap_k T_k\,, \qquad \beta^1 \in [t_2,t_1]\,,
\]
such that $\beta_k^1\to \beta^1$ and $\sup_k\mathcal{H}^n(\{v_k^* = \beta_k^1\} \cap B_r(x_0))<\infty$. Let $\{A_m^{1,k}\}_m$ denote the essential partition of $B_r(x_0)$ induced by $\{v_k^* = \beta_k^1\}$, so that 
\[
    \sup_k \sum_m \Hcal^n(\partial^* A_m^{1,k} \cap B_r(x_0)) \leq 2\sup_k\mathcal{H}^n(\{v_k^*=\beta_k^1\}\cap B_r(x_0))<\infty\,.
\]
By restricting to a further subsequence, standard compactness for sets of finite perimeter implies the existence of an essential partition $\{A_m^1\}_m$ of $B_r(x_0)$ such that $A_m^{1,k}\to A_m^1$ as $k\to\infty$ in $L^1$ and pointwise a.e.~ for each $m$. By iterating this argument and a diagonalization procedure which restricts to a further subsequence in $k$, again not relabeled, for each $j$ we obtain sequences $\{\beta_k^j\}_k$ with $\beta_k^j \to \beta^j \in [t_{j+1},t_j]$, a partition $\{A_m^j\}_m$ of $B_r(x_0)$, and sequences of essential partitions $\{\{A_m^{j,k}\}_m\}_k$ of $B_r(x_0)$ induced by $\{v_k^* = \beta_k^j\}$ such that for each fixed $m$ and $j$, $A_m^{j,k}\to A_m^j$ as $k\to\infty$ in $L^1$ and pointwise a.e. By Lemma \ref{lemma connected 2} (in particular \eqref{eq:ci ft dichotomy}), we may also identify for each $j$ and $m$ sequences $\{C_{i(j,k,m)}^k\}_k$ of essentially connected components of $\{v_k^* >0\}$ in $B_r(x_0)$ such that $|A_m^{j,k}\setminus C_{i(j,k,m)}^k|=0$ for each $m,j,k$.

We claim that for each $j$, $\cup_m \partial^* A_m^j \cap B_r(x_0) \subset \{ v^* = \beta^j\}$ up to an $\mathcal{H}^n$-null set. To see that this is the case, consider the functions $\mathbf{1}_{A_m^{j,k}}v_k + \beta^j_{k} \mathbf{1}_{(A_m^{j,k})^c}\in W^{1,2}(B_r(x_0))$. Since they are uniformly bounded in $W^{1,2}(B_r(x_0))$ over $k$, their a.e.~ pointwise limit $v_{j,m}:=\mathbf{1}_{A_m^{j}}v + \beta^j \mathbf{1}_{(A_m^{j})^c}$ belongs to $W^{1,2}(B_r(x_0))$ also. Since $W^{1,2}$ functions cannot have jump discontinuities along hypersurfaces, the traces of $v_{j,m}$ coming from inside and outside $A_m^j$ have to be equal. But the trace coming from outside is $\beta^j$, which means that $v^* = \beta^j$ $\mathcal{H}^n$-a.e.~ on $\partial^* A_m^j \cap B_r(x_0)$, proving the claim. 

As a consequence of this claim and the fact that $\beta^j \leq t_j$, we claim further that for each $F_j^i$ there exists $m(j)$ such that $|F_j^i \setminus A_{m(j)}^j|=0$. Indeed, if this were not the case, then there would be some $A_m^j$ such that $0<|F_j^i \cap A_m^j| < |F_j^i|$. But then, we would have $\{v^* = \beta^j \}$ essentially disconnecting $F_j^i$, which is impossible since $(F_j^i)^\one \subset \{v^*> t_j\}$ up to an $\mathcal{H}^n$-null set (cf.~ \eqref{eq:cant split}). 

Putting together all of the previous observations, for each $F^i_j$ we have found a set of finite perimeter $A_{m(j)}^j$ such that $|F^i_j \setminus A_{m(j)}^j|=0$ and $\mathbf{1}_{A^j_{m(j)}} = \lim_{k\to\infty} \mathbf{1}_{A_{m(j)}^{j,k}}$ $\Lcal^{n+1}$-a.e., for a sequence $\{A_{m(j)}^{j,k}\}_k$ such that $|A_{m(j)}^{j,k}\setminus C_{i(j,k,{m(j)})}^k|=0$ for some essentially connected component $C_{i(j,k,{m(j)})}^k$ of $\{v_k^* >0\}$ in $B_r(x_0)$. Thus,
\begin{equation}\label{eq:fij containment}
    \mathbf{1}_{F_j^i} \leq \mathbf{1}_{A_{m(j)}^j} = \lim_{k\to\infty} \mathbf{1}_{A_{m(j)}^{j,k}} \leq \liminf_{k\to\infty} C_{i(j,k,{m(j)})}^k \qquad \mbox{$\Lcal^{n+1}$-a.e.}
\end{equation}
We would therefore be done with proving $\mathbf{1}_{F_j^i} \leq \liminf \mathbf{1}_{C_{i(k)}^k}$ if we could choose a subsequence in $k$ and sets $C^k_{i(k)}$ such that $C^k_{i(j,k,{m(j)})}=C^k_{i(k)}$ for all $k$; in other words, we wish to remove the dependence on $j$. To achieve this, recall that $F_j^i \subset F_{j+1}^i$ for all $j$. Combined with \eqref{eq:fij containment}, we find that
\begin{equation*}
    \mathbf{1}_{F_1^i} \leq \mathbf{1}_{F_2^i} \leq \liminf_{k\to\infty} C_{i(2,k,{m(j)})}^k \quad \mbox{a.e.}\qquad \mbox{and}\qquad \mathbf{1}_{F_1^i} \leq \liminf_{k\to\infty} C_{i(1,k,{m(j)})}^k\quad \mbox{a.e.}
\end{equation*}
But the sets $\{C_i^k\}_i$ are pairwise disjoint up to Lebesgue null sets, which together with the previous pair of inequalities implies the existence of $K(2)$ such that for all $k \geq K(2)$, $C_{i(2,k,{m(j)})}^k=C_{i(1,k,{m(j)})}^k$. Continuing on as such, we may inductively identify $K(j) \geq K(j-1)$ such that for all $k \geq K(j)$, $C_{i(j,k,{m(j)})}^k=C_{i(1,k,{m(j)})}^k$ and
\begin{equation}\label{eq:fij chain}
   \mathbf{1}_{F_1^i} \leq \mathbf{1}_{F_2^i}\leq \cdots \leq \mathbf{1}_{F_j^i} \leq \liminf_{k\to \infty} C_{i(1,k,{m(j)})}^k\qquad \mbox{a.e.}
\end{equation}
We can now choose our final subsequence in $k$ to be $\{K(j)\}_j$ and our sets $C_{i(K(j))}^{K(j)}=C_{i(1,K(j),{m(j)})}^{K(j)}$. So the chain of inequalities \eqref{eq:fij chain} finishes this step. 


\medskip

\noindent{\it Step 2}: Here we conclude that $\{g_i^*\leq t\}$ is $\mathcal{C}$-spanning for all $t\in (0,1/2)$. Fix a ball $B_r(x_0) \subset \Omega$. Let $v_i^k = \mathbf{1}_{C_{i(k)}^k}v_k$, and let $\psi_i^k$ and $\psi_i$ be the respective semilinear replacements from Lemma \ref{lemma semilinearreplacement} for $v_i^k$ and $v_i$ in $B_{r/2}(x_0)$. Let $g_i^k=(v_i^k +  \min\{\vphi, v_k-v_i^k\}) \mathbf{1}_{B_r(x_0) \setminus B_{r/2}(x_0)} + \psi_i^k$ be the cup competitors associated to $v_k$ and $C_{i(k)}^k$. Let us take a further subsequence such that $v_k\to v$ pointwise a.e.~, $v_i^k \toweak w_i$  in $W^{1,2}$ and pointwise a.e.~ for some $w_i \in W^{1,2}(B_{r}(x_0)) $, and $g_i^k \toweak G_i$ in $W^{1,2}$ and pointwise a.e.~ for some $G_i$.

By Step 1 and these observations, we have $v_i \leq \liminf_{k\to\infty} v_i^k =: w_i$ a.e. on $B_{r}(x_0)$. Thus, since $v_i, w_i \in W^{1,2}(B_r(x_0))$ we have $v_i|_{\partial B_{r/2}(x_0)} \leq w_i|_{\partial B_{r/2}(x_0)}$ in the sense of traces on $\partial B_{r/2}(x_0)$, so that by Lemma \ref{lemma semilinearreplacement}, $\psi_i \leq \psi_i'$, where $\psi_i'$ is the semilinear elliptic replacement for $w$ on $B_{r/2}(x_0)$. Furthermore, by compactness of the trace operator $T:W^{1,2}(B_{r/2}(x_0))\to L^2(\partial B_{r/2}(x_0))$, we have $v_i^k|_{\partial B_{r/2}} \to w_i|_{\partial B_{r/2}(x_0)}$ strongly in $L^2(\partial B_{r/2}(x_0))$, which implies that the weak (subsequential) $W^{1,2}$-limit of $\psi_i^k$ is $\psi_i'$ on $B_{r/2}(x_0)$. Therefore,
$$
G_i = w_i + \min\{\varphi,v - w_i\}\mathbf{1}_{B_r(x_0)\setminus B_{r/2}(x_0)}+ \psi_i'\,. 
$$
Extend $G_i$ by $v$ to a function on the entirety of $\Omega$. By Theorem \ref{thm:prelim compactness thm}, $G_i$ satisfies $\{G_i^* \leq t\}$ is $\mathcal{C}$-spanning for every $t\in (0,1/2)$. If we show that $g_i \leq G_i$ a.e.~, then it follows that $g_i$ satisfies this condition as well, so that the proof will be complete. We only need to check the inequality on $B_r(x_0)$ where $g_i \neq G_i$. On $B_{r/2}(x_0)$, we have $g_i = \psi_i \leq \psi_i' = G_i$, so it remains to check on $B_{r}(x_0)\setminus B_{r/2}(x_0)$. On the latter annulus, if $\min\{\varphi(x),v(x) - w_i(x)\}$ is achieved by $\varphi(x)$ at some (Lebesgue) point $x$, we have $G_i(x) = w_i(x) + \varphi(x) \geq v_i(x) + \min\{\varphi(x), v(x)-v_i(x)\} = g_i(x)$. On the other hand, if $\min\{\varphi(x),v(x) - w_i(x)\}$ is achieved by $v(x)-w_i(x)$, then $G_i(x)=w_i(x) + v(x) - w_i(x) = v(x) \geq v_i(x)  + \min\{ \varphi(x), v(x)-v_i(x)\}=g_i(x)$. Thus $G_i \geq g_i$ a.e. on the annulus as well, so we are done.
\end{proof}

Given the weak $W^{1,2}$-closure of $\Acal$, we are now in a position to prove our main existence theorem.

\begin{proof}[Proof of Theorem \ref{thm:main existence theorem}]

The proof is divided into steps. First we obtain limits of minimizing sequences for \eqref{new model intro 2} and \eqref{new model volume constrained 2}. Then in steps two through four, we verify that these limits are admissible and minimizing for the either \eqref{new model intro 2} or \eqref{new model volume constrained 2} (using crucially \eqref{eq:sharp existence assumption 2}) and also \eqref{new model intro} or \eqref{new model volume constrained} respectively (by applying the regularity theory in Section 3). Note that we must distinguish between the cases $n=1$ and $n\geq 2$ when verifying the admissibility for \eqref{new model intro 2} and \eqref{new model intro}.

\medskip

\noindent{\it Step 1 (limits of minimizing sequences)}: Let $\{u_j\}$ be a minimizing sequence for \eqref{new model intro 2} or \eqref{new model volume constrained 2}. By Lemma \ref{lemma:inf not infinity}, which asserts that the infimums are indeed finite, there exists $u\in W^{1,2}_\loc(\Om;[0,1])$ such that (up to a subsequence) $u_j\to u$ strongly in $L^1_\loc$ and, by the lower-semicontinuity of the Dirichlet energy and Fatou's lemma,
\begin{equation}\notag
\int_\Om |\nabla u|^2 + F(u) \,dx\leq \liminf_{j\to \infty}\int_\Om |\nabla u_j|^2 + F(u_j)\,dx\,.
\end{equation}
By Proposition \ref{p:weak-closure-A}, $u \in \Acal$.

\medskip

\noindent{\it Step 2 (admissibility/minimality of $u$ in \eqref{new model intro 2} and \eqref{new model intro} if $n\geq 2$}: In this case, by the lower-semicontinuity of the Dirichlet energy and Fatou's lemma, $u\in D_n^{1,2}(\Om;[0,1]){\cap\Acal}$ and so is admissible for \eqref{new model intro 2}. {Therefore, it is a minimizer for \eqref{new model intro 2}, and so by Corollary \ref{corollary higher order regularity}, it is a minimizer for \eqref{new model intro}.}


\medskip

\noindent{\it Step 3 (admissibility/minimality of $u$ in \eqref{new model intro 2} and \eqref{new model intro} if $n=1)$}: If $n=1$, then by \eqref{eq:sharp existence assumption 2}, there exists $t_k\searrow 0$ such that $F(t_k)>0$. In order to obtain the decay of $u$ at infinity, we will use this to show that, for all $t\in (0,1)$
\begin{equation}\label{eq:sequence bounds}
   \sup_j \mathcal{L}^2(\{u_j>t\}) < \infty\,.
\end{equation}
Assuming the validity of the uniform bound \eqref{eq:sequence bounds}, which depends on $t$ but not $j$, combined with the $L^1_\loc$ convergence of $u_j$ to $u$, we deduce that $u\in D_1^{1,2}(\Om;[0,1]){\cap\Acal}$ and thus is admissible for \eqref{new model intro 2}. 

To prove \eqref{eq:sequence bounds}, for $R$ such that $\wire \cc B_R$, we let $E$ denote a continuous linear extension operator from $W^{1,2}(B_{2R} \setminus B_R;[0,1])$ to $W^{1,2}(B_{2R};[0,1])$. In a slight abuse of notation, for $u_j$ we will let $E u_j$ denote the function on $\mathbb{R}^{n+1}$ which agrees with $u_j$ outside $B_{R}$. It thus suffices to prove \eqref{eq:sequence bounds} for $E u_j$, and in fact for $(E u_j)^*$, which is the radially symmetric decreasing rearrangment (see e.g. \cite[Section 1.4.1]{Grafakos-classical}) of $E u_j$. Note that the uniform energy bound for $u_j$ implies that
\begin{equation}\label{eq:finite energy assumption ex proof}
    \sup_j \int_{\mathbb{R}^{n+1}}|\nabla (E u_j)^*|^2 + F((E u_j)^*) \,dx < \infty\,.
\end{equation}
Let us assume for contradiction that the uniform bound \eqref{eq:sequence bounds} for $(E u_j)^*$ does not hold with some $t_0\in (0,1)$. Then, letting $r_j\to \infty$ be such that $\mathcal{L}^2(\{(E u_j)^*>t_0\})=\pi r_j^2$ (up to extracting a subsequence if necessary), we set $\mathcal{F}(t) = \int_0^t \sqrt{F(s)}\,ds$ and use the identity $2ab \leq a^2 + b^2$ to estimate
\begin{align*}
    \int_{\mathbb{R}^{n+1}}|\nabla (E u_j)^*|^2 + F((E u_j)^*) \,dx &\geq 2\int_{B_{r_j}^c}|\nabla_x \mathcal{F}((E u_j)^*(x))| \, dx \\
    &\geq 4\pi r_j \int_{r_j}^\infty |\partial_r \mathcal{F}((E u_j)^*(r))|\,dr\\
    &=4\pi r_j \Fcal((Eu_j)^*(r_j))\\
    &= 4\pi r_j \mathcal{F}(t_0)\,, 
\end{align*}
where in the next to last line we have used the fundamental theorem of calculus and the assumption that $(E u_j)^*$ decays to $0$ at infinity (which follows from $u_j \in D_1^{1,2}(\Om;[0,1])$), while in the last line we have used the fact that $(Eu_j)^*(r_j)=t_0$, by construction. Note that above we have abused notation slightly, interchanging between $(Eu_j)^*$ as a function of $x$ and as a function of $r=|x|$. Now since $F$ is not the zero function on $(0,t_0)$ by \eqref{eq:sharp existence assumption 2}, $\mathcal{F}(t_0)>0$ and thus the quantity $4\pi r_j \Fcal(t_0)$ diverges to $\infty$ if $r_j\to \infty$, contradicting \eqref{eq:finite energy assumption ex proof}. Therefore, we have shown \eqref{eq:sequence bounds}, and so $u\in D_1^{1,2}(\Om;[0,1]){\cap\Acal}$ and is minimizing for \eqref{new model intro 2}. By {Corollary \ref{corollary higher order regularity}, it is minimizing for \eqref{new model intro}.}

\medskip

\noindent{\it Step four (admissibility/minimality of $u$ in \eqref{new model volume constrained 2} and \eqref{new model volume constrained}}): 
{Since any minimizer of \eqref{new model volume constrained 2} is a minimizer of \eqref{new model volume constrained} (by Corollary \ref{corollary higher order regularity}) and we have already verified that $u \in \Acal$}, to conclude the existence proof for minimizers of \eqref{new model volume constrained} under the additional volume constraint, it remains to show that 
\begin{equation}\label{eq:vol pres in thm proof}
\int_\Om V(u)\,dx = 1\,.
\end{equation}
Assume for contradiction that $\int_\Omega V(u)$ is strictly less than one and set
\begin{equation}\notag
\e := \int_\Om V(u)\,dx\in (0,1)\,.
\end{equation}

In order to prove \eqref{eq:vol pres in thm proof} by contradiction, we make three preliminary claims: first, that
\begin{equation}\label{eq:concentration consequence}
\liminf_{j\to \infty}\int_\Om |\nabla u_j|^2 + F(u_j)\,dx \geq \int_\Om |\nabla u|^2 + F(u)\,dx + \Psi(1-\e)
\end{equation}
(recall the definition of $\Psi$ in \eqref{eq:isop problem def}); second, that $u$ is a minimizer for the problem
\begin{equation}
\Psi_\wire(\e):=\inf\left\{\int_{\Omega} |\nabla v|^2 +F(v): \begin{array}{@{}c@{}} v\in W^{1,2}(\Om;[0,1]), \ \int_{\Omega}V(v)=\e, \\ \text{$\{v^* \geq t\}$ is $\mathcal{C}$-spanning $\wire$ for all $t\in (1/2,1)$}\end{array} \right\} ;\label{eq:vol constrained model in ex proof}
\end{equation}
and third, that
\begin{equation}\label{eq:characterization of infimum}
\mbox{the infimum in \eqref{new model volume constrained 2} is equal to $\int|\nabla u|^2 + F(u)\,dx+ \Psi(1-\e)=\Psi_\wire(\e) + \Psi(1-\e)$}\,.
\end{equation}
The lower bound \eqref{eq:concentration consequence} follows by a standard localization argument which we omit. For the second two claims \eqref{eq:vol constrained model in ex proof}-\eqref{eq:characterization of infimum}, firstly for any $v$ which is admissible for \eqref{eq:vol constrained model in ex proof} (in particular $u$ itself), we may consider the functions
\begin{equation}\notag
v_j(x) = \max\{v(x),w(x-je_1)\}\,,
\end{equation}
where $w$ is a radial, decreasing minimizer for the isoperimetric problem \eqref{eq:isop problem def} with $v_0 =1-\eps$; see Theorem \ref{thm:existence for isop problem}. Observe that $v_j$ 
satisfy the spanning condition and also $\int_\Omega V(v_j)\nearrow 1$, since $w$ is decreasing to zero at infinity. Thus by Lemma \ref{lemma volume fixing}, we may fix volumes so that $v_j$ are admissible for \eqref{new model volume constrained 2}. Combining this with the fact that $\{u_j\}$ is a minimizing sequence for the latter, we obtain the upper bound
\begin{equation}\label{eq:upper bound for new model vol const}
\liminf_{j\to \infty}\int_\Om |\nabla u_j|^2 + F(u_j)\,dx \leq \lim_{j\to \infty}\int_\Om |\nabla v_j|^2 + F(v_j)\,dx = \int_\Om |\nabla v|^2 + F(v)\,dx + \Psi(1-\e)\,.
\end{equation}
By \eqref{eq:concentration consequence} and the fact that \eqref{eq:upper bound for new model vol const} holds for every admissible $v$ in \eqref{eq:vol constrained model in ex proof}, we deduce that
\begin{align*}
\Psi_\wire(\e) + \Psi(1-\e)&\leq   \int_\Om |\nabla u|^2 + F(u)\, dx + \Psi(1-\e) \\
&\leq \liminf_{j\to \infty}\int_\Om |\nabla u_j|^2 + F(u_j)\,dx \\
&\leq \Psi_\wire(\e) + \Psi(1-\e)\,.
\end{align*}
This concludes the arguments for the minimality of $u$ in \eqref{eq:vol constrained model in ex proof} and \eqref{eq:characterization of infimum}. 

\medskip

\noindent{\it To prove \eqref{eq:vol pres in thm proof} by contradiction}: Now that we have demonstrated \eqref{eq:concentration consequence}-\eqref{eq:characterization of infimum}, we are in a position to prove \eqref{eq:vol pres in thm proof}. We introduce the notation
\begin{equation}
\mathcal{E}(u;U) = \int_{U}|\nabla u|^2 + F(u)\,dx\,,\qquad \mathcal{V}(u;U) = \int_{U}V(u)\,dx\,,
\end{equation}
for $U \subset \Omega$. We simply write $\mathcal{E}(u)$ and $\mathcal{V}(u)$ respectively in the case when $U=\Omega$. We introduce the functions
\begin{equation}\notag
v_k(x) = \max \{u(x),\,w(x-ke_1) \}:\Om \to [0,1]
\end{equation}
which have volume strictly less than $1$ {due to the fact that $w>0$ (see Theorem \ref{thm:existence for isop problem})}; more precisely, denoting
\begin{equation}\notag
A_k := \{x\in \Om:0<u(x)<w(x-ke_1) \}\,,\quad B_k := \{x\in \Om:0<w(x-ke_1)\leq u(x) \}{\cup \wire}\,,
\end{equation}
which satisfy $A_k \cap B_k = \emptyset$, $|A_k|+|B_k|>0$ since {$w>0$}, we have
\begin{equation}\notag
\V(v_k) = 1-\V( u ; A_k) - \V(w(\cdot - ke_1);B_k)<1\,.
\end{equation}
Since $u$ is minimal for \eqref{new model volume constrained 2} with potential $\e^{-1}V$, Corollary \ref{corollary higher order regularity}.(iii) applies to $u$, yielding the Lipschitz bound \eqref{eq uniform holder} uniformly on small balls away from $\partial \Om$, so that $u$ decays uniformly to $0$ as $|x|\to \infty$. Combined with the fact that $w$ also decays uniformly to $0$ (it is radially decreasing), we find that
\begin{equation}\notag
0<\max \{\sup\{u(x) : x\in A_k \},\,\sup \{w(x-ke_1):x\in B_k \}\} \leq \beta_k
\end{equation}
for some $\beta_k\to 0$. Therefore, by the assumption \eqref{eq:A4} that $ \lim_{t\to 0}V(t)/F(t)=0$, 
\begin{equation}\label{missing expensive parts}
\frac{\ac(u;A_k)+ \ac(w(\cdot - ke_1);B_k)}{\V(u;A_k)+ \V(w(\cdot - ke_1);B_k)}\geq \frac{\int_{A_k}F(u) + \int_{B_k}F(w(x-ke_1))}{\V(u;A_k)+ \V(w(\cdot - ke_1);B_k)} \geq \inf_{0<t\leq \beta_k}\frac{F(t)}{V(t)}\to \infty\,.
\end{equation}
By applying a volume fixing variation to $v_k$ as given by Lemma \ref{lemma volume fixing}.(ii) that increases the volume to $1$, there is a constant $C_2>0$ (independent of $k$) such that for large $k$, there is $\tilde{v}_k$ with
\begin{equation}\label{eq:tildevk vol fix estimates}
    \mathcal{V}(\tilde{v}_k)=1,\,\, \mathcal{E}(\tilde{v}_k)) \leq C_2\big(1-\mathcal{V}(v_k)\big)+ \mathcal{E}(v_k)=C_2\big(\mathcal{V}(u;A_k)+\mathcal{V}(w(\cdot-ke_1);B_k)\big)+ \mathcal{E}(v_k)\,.
\end{equation}
By \eqref{missing expensive parts}, we may choose some $k'$ large enough so that 
\begin{equation}\label{bad ratio}
\frac{\ac(u;A_{k'})+ \ac(w(\cdot - k'e_1);B_{k'})}{\V(u;A_{k'})+ \V(w(\cdot - k'e_1);B_{k'})}{>C_2}\,,
\end{equation}
Since $\{u=1\}$ is $\mathcal{C}$-spanning, $\{\tilde{v}_{k'}=1\}$ is as well, so it is admissible for \eqref{new model volume constrained 2}. Furthermore, by \eqref{eq:tildevk vol fix estimates}-\eqref{bad ratio} and the minimality of $u$ for $\Psi_\wire(\e)$ we have {
\begin{align}\notag
\ac({\tilde{v}_{k'}}) &\leq C_2\big(\mathcal{V}(u;A_{k'})+\mathcal{V}(w(\cdot-k'e_1);B_{k'})\big)+ \mathcal{E}(v_{k'}) \\ \notag 
&=C_2\big(\mathcal{V}(u;A_{k'})+\mathcal{V}(w(\cdot-k'e_1);B_{k'})\big) + \ac(u;\Omega \setminus A_{k'})+ \ac(w;\Omega \setminus B_{k'})\\ \notag
&<\ac(u;A_{k'})+ \ac(w(\cdot - k'e_1);B_{k'})+\ac(u;\Omega \setminus A_{k'})+ \ac(w;\Omega \setminus B_{k'}) \\ \notag
&= \Psi_\wire(\e) + \Psi(1-\e)\,.
\end{align}
But by the admissibility of $\tilde v_k$ for \eqref{new model volume constrained 2}, this contradicts \eqref{eq:characterization of infimum}. So it must be the case that $\V(u)=1$, which is \eqref{eq:vol pres in thm proof}.}
\end{proof}

\begin{remark}[Optimality of the assumptions in $\mathbb{R}^2$ for \eqref{new model intro}]\label{remark:optimality of ex assumptions}
    If $n=1$ and there exists $t_0>0$ such that $F(t)=0$ for $t\in [0,t_0]$, then there do not exist any minimizers for \eqref{new model intro}, and any minimizing sequence converges to a function which is bounded from below by $t_0$. To see this, take a minimizing sequence $\{u_j\}$ and, for $R$ large enough such that $\Wbf\subset B_R$, consider the functions
    \begin{equation}
        w_j(x) = \begin{cases}
			\max\{u_j(x),t_0\} &\quad x\in B_{R} \cap \Om\\ 
			{\max\{2t_0- t_0\log(|x|)/\log(R),u_j(x)\}} &\quad x\in B_{R^2}\setminus B_{R}\\
   u_j(x)&\quad \mbox{otherwise}\,.
		\end{cases}
    \end{equation}
    Direct computation shows that their energy approaches that of $\max\{u_j,t_0\}$ on $\Om$ as $R\to \infty$, which, since $F=0$ and $u_j$ decays at infinity, has strictly less energy than $u_j$. If $u_j$ converged to a function which took values below $t_0$, this strict inequality would persist in the limit and contradict the minimality of our sequence. So there is no minimizer if $F=0$ on $[0,t_0]$. Note that when $n\geq 2$, there is no way to truncate in a way that simultaneously ensures decay at infinity and that the Dirichlet energy of the tails decay to zero. Indeed, this phenomenon of lack of existence of global solutions to \eqref{new model intro} with $F$ vanishing only occurs when $n=1$.
\end{remark}

\subsection{Proof of Theorem \ref{thm:main regularity theorem}.ii}

Proposition \ref{p:tangents} and Proposition \ref{p: freq gap} allow us to provide the following definitions of the \emph{singular} and \emph{regular} parts of the free boundary $\{u=1\}$, for solutions $u$ of  \eqref{eq innervariation}-\eqref{eq:differential inequality} in terms of $v=1-u$.

\begin{definition}
Let $v\in W^{1,2}_\loc(\Omega)$ be a solution of \eqref{eq modified iv}-\eqref{eq modified ineq}. We define the singular set $\mathcal{S}(v)$ of $\{v=0\}$ as
\[
\mathcal{S}(v) := \left\{x: N_{v,x}(0^+) \geq \tfrac{3}{2}\right\},
\]
and we define the regular set $\Rcal(u)$ of $\{u=1\}$ as
\[
\Rcal(u) := \left\{x : N_{v,x}(0) =1\right\}.
\]

Abusing notation, for $u\in W^{1,2}_\loc(\Omega;[0,1])$ a solution of \eqref{eq innervariation}-\eqref{eq:differential inequality} and $v=1-u$, we in turn define the respective singular and regular sets $\mathcal{S}(u)$, $\Rcal(u)$ of $\{u=1\}$ as
\[
    \Scal(u) := \Scal(v), \qquad \Rcal(u) := \Rcal(v)\,.
\]
\end{definition}

We have the following immediate corollary.

\begin{corollary}\label{corollary decomp fb}
Let $u\in W^{1,2}_\loc(\Omega;[0,1])$ be a solution of \eqref{eq innervariation}-\eqref{eq:differential inequality}. Then $\Omega \cap\{u=1\}$ decomposes as the disjoint union $\mathcal{S}(u) \sqcup \Rcal(u)$, where $\mathcal{S}(u)$ is relatively closed in $\Om \cap \{u=1\}$.
\end{corollary}

\begin{proof}
The decomposition into $\mathcal{R}(u)$ and $\mathcal{S}(u)$ is an immediate consequence of Proposition \ref{p:tangents}, and the fact that $\mathcal{S}(u)$ is relatively closed in $\Om$ follows by the upper semicontinuity of the frequency function.
\end{proof}

We begin our analysis by focusing on $\Rcal(u)$; namely, we will proceed to prove Theorem \ref{thm:main regularity theorem}. This regularity will essentially follow from noticing that points in $\Rcal(u)$ correspond to regular points in the zero set of a function obtained from a suitable reflection of $v$. Thus, a priori, from the regularity of this reflected function, one can get initial regularity via the implicit function theorem for $\Rcal(u)$. This argument is rather standard and we present it here for the sake of completeness, we remark that this reflection argument can be traced back, at least in spirit the argument to \cite{evans1940minimal} where surfaces of minimal capacity were realized as zero sets of multivalued harmonic functions. In our case, we follow the arguments in \cite{TT}. The only difference lies in the analyticity conclusion which is a direct consequence of \cite[Theorem 4]{koch2015partial}. Notice that this latter result is rather surprising, since it guarantees that regular level sets of solutions to semilinear PDEs are analytic regardless of the regularity of the non-linearity.

Let us re-state Theorem \ref{thm:main regularity theorem} here for convenience.

\begin{theorem}[Regularity of $\Rcal(u)$]\label{t:reg-set}
If $u\in W^{1,2}_\loc(\Omega;[0,1])$ is a solution of \eqref{eq innervariation}-\eqref{eq:differential inequality} with {$\Phi\in C^2$ and $\Phi'(1)=0$}, then $\Rcal(u)$ is locally an $n$-dimensional analytic submanifold.
\end{theorem}

\begin{proof}[Proof of Theorem \ref{t:reg-set}]
Let $x_0\in \Rcal(u)$. It suffices to prove that $\mathcal{R}(u)\cap B_{\rho_0}(x_0)$ has the desired structure for some $0<\rho_0<\dist(x_0,\mathcal{S}(u))$, bearing in mind that $\dist(x_0,\mathcal{S}(u))>0$, since $\mathcal{S}(u)$ is relatively closed in $\Omega$. We proceed in steps as follows.

Let $v=1-u$ and {let $G(v):= \Phi(1-v)$}. First, we observe that, in virtue of Proposition \ref{lemma density components}, there exists $0<\rho_0<\dist (x_0,\mathcal{S}(u))$ such that $\{v>0\} \cap B_{\rho_0}(x_0)$ has exactly two connected components.      

 We will now proceed to show that the zero set of $\tilde v := v\mathbf{1}_{\overline{B}^+} - v\mathbf{1}_{B^-}$ is analytic in $B_{\rho_0}(x_0)$. Firstly, we may apply Lemma \ref{l:reflection} to conclude that $\tilde v$  is a weak solution of
\[
\Delta \tilde v = \frac{1}{2}\tilde H(\tilde v) \mbox{  in  $B_{\rho_0}(x_0)$}\,,
\]
for $\tilde H$ as in \eqref{eq odd reflection}. Since $G'(0)=0$, and $G \in C^{2}$, we have that $\tilde H$ is $C^1$. Thus, in virtue of the regularity of $G$ and standard elliptic regularity theory, we deduce from the previous step that $\tilde{v}\in C^{2}_\loc(B_{\rho_0}(x_0))$.  In particular, $\nabla \tilde{v}(x)$ exists in the classical sense at any $x\in B_{\rho}(x_0)$. Let us notice that at any $y \in \{v=0\}\cap B_{\rho_0}(x_0)$, we have $N_{v,y}(0^+) = N_{\tilde v, y}(0^+)=1$. Now for any such $y$, consider a subsequential limit $w$ of the rescalings $\tilde v_{y,r}(x)=\frac{\tilde v (y+rx)}{H_{\tilde v,y}(r)^{1/2}}$. Once again exploiting Lemma \ref{l:reflection} together with Lemma \ref{lemma strong convergence blow-ups} and Lemma \ref{lemma prop blowups} (cf. the proof of Lemma \ref{l:harmonic-polyn}), we deduce that $w$ is a homogeneous harmonic polynomial of degree $N_{w,0}(0^+) = N_{\tilde v, y}(0^+)$. Now, if $\nabla \tilde v(y)=0$, the subsequential convergence of $\tilde v_{y,r}$ to $w$ guarantees that $N_{w,0}(0^+) > 1$, yielding a contradiction. Thus, $\nabla\tilde v$ doesn't vanish anywhere on $\{v=0\}\cap B_{\rho_0}(x_0)$. Finally, we deduce from \cite[Theorem 4]{koch2015partial} that $\{\tilde v =0\}\cap B_{\rho_0}(x_0)$ is analytic.
\end{proof}

We continue our analysis by providing a dimension bound on $\mathcal{S}(u)$ \`{a} la Federer. The argument is standard and appears in the literature in numerous places (for instance \cite[Theorem 4.6]{TT}, \cite{DL-survey-JDG}), but we provide a proof here nevertheless, for purpose of clarity, since it is short and elementary.  We start by combining Lemma \ref{lemma prop blowups} and Proposition \ref{p:tangents} to deduce that when $n=1$, $\mathcal{S}(u)$ consists of isolated points.

\begin{theorem}\label{t:2d-isolated-sing}
Let $n=1$ and let $v\in W^{1,2}_\loc(\Omega)$ be a solution of \eqref{eq modified iv}-\eqref{eq modified ineq}. Then $\mathcal{S}(v)$ consists of isolated points.
\end{theorem}

\begin{proof}
We argue by contradiction. Suppose that there exists a sequence $\{x_k\}\subset \mathcal{S}(v)$ with an accumulation in the interior of $\Omega$. Then up to extracting a subsequence, $x_k \to x_0\in \mathcal{S}(v)$. Let $r_k := 2|x_k - x_0|$. Applying Lemma \ref{lemma prop blowups} to the sequence $v_{x_0,r_k}$, we obtain a limiting radially $N_{v,x_0}(0^+)$-homogeneous function $\bar v\in W^{1,2}\cap { \text{Lip}(\bar B_1)}$, which, up to rotation, has the structure \eqref{e:structure-2d-tangents} for some integer $N \geq 2$. However, observe that the points $y_k = \frac{x_k-x_0}{r_k}$ satisfy $|y_k| = \frac{1}{2}$ and again by upper semicontinuity of the frequency, $y_k \to y_0$ with $N_{\bar v, y_0}(0) > 1$. However, this contradicts the classification in Lemma \ref{lemma:planar tangent classification} established for $\bar v$; indeed, it is easy to explicitly check that $N_{\bar v,y}(0^+) = 1$ for any $y\neq 0$.
\end{proof}

\begin{corollary}\label{c:sing-set-dim}
Let $v\in W^{1,2}_\loc(\Omega)$ be a solution of \eqref{eq modified iv}-\eqref{eq modified ineq}. Then
\[
\dim_\Hcal(\mathcal{S}(v)) \leq n-1.
\]
\end{corollary}

\begin{proof}
We will argue by induction on $n$, following Federer's dimension reduction argument in this setting. Observe that the $n=1$ case is automatically covered by Theorem \ref{t:2d-isolated-sing}, which already provides a sharper statement. Now suppose that $n\geq 2$ and that we have established the the dimension estimate in $\R^n$, but suppose for a contradiction that it is false in $\R^{n+1}$. Then there exists $v$ satisfying \eqref{eq modified iv}-\eqref{eq modified ineq}, an exponent $\alpha>0$ and a compact subset $K\subset \mathcal{S}(v)$ such that
\[
\Hcal^{n-1+\alpha}(K) > 0.
\]
Recall the notion of $(n-1+\alpha)$-dimensional Hausdorff content $\Hcal^{n-1+\alpha}_\infty$ (see e.g. \cite{LSimon_GMT}), which has the same negligible sets as $\Hcal^{n-1+\alpha}$, but unlike the Hausdorff measure itself, is upper semicontinuous with respect to Hausdorff convergence of compact sets).

In particular, {$\Hcal^{n-1+\alpha}_\infty(K) > 0$ and so, since for $\Hcal^{n-1+\alpha}$-a.e. point the upper $\Hcal^{n-1+\alpha}_\infty$-density is strictly positive (again, see \cite{LSimon_GMT}),} there exists $x_0\in K$ and $\eta > 0$ such that
\[
 \liminf_{r\downarrow 0} \Hcal^{n-1+\alpha}_\infty(B_1\cap K_{x_0,r}) {=} \liminf_{r\downarrow 0} \frac{\Hcal^{n-1+\alpha}_\infty(B_r(x_0)\cap K)}{r^{n-1+\alpha}} \geq \eta.
\]
where $K_{x_0,r} \subset \Scal(v_{x_0,r})$ denotes the rescaling $(K- x_0) r^{-1}$, with $v_{x_0,r}$ as defined in \eqref{e:rescaling}. Therefore, there exists a subsequence $r_k \downarrow 0$ and a compact set $K_\infty$ such that $K_{x_0,r_k} \to K_\infty$ in Hausdorff distance, and
\begin{equation}\label{e:density-pt}
\Hcal^{n-1+\alpha}_\infty(B_1\cap K_\infty) \geq \eta.
\end{equation}
In particular, we argue as above to deduce that there must exist a point $y_0\in K_\infty\cap B_1 \setminus \{0\}$ with
\[
\liminf_{r\downarrow 0} \frac{\Hcal^{n-1+\alpha}_\infty(B_r(y_0)\cap K_\infty)}{r^{n-1+\alpha}} > 0
\]
Furthermore, letting $\bar v$ denote a tangent function of $v$ at $x_0$ along the sequence $\{r_k\}$; the conclusions of Lemma \ref{lemma prop blowups} imply that $K_\infty\cap B_1 \subset \Scal(\bar v)$. 

Repeating the above steps, we may now apply Lemma \ref{lemma prop blowups} to take a tangent function $\bar v_\infty$ to $\bar v$ at $y_0$, along some sequence $\rho_k \downarrow 0$, so that we additionally have
\[
\Hcal^{n-1+\alpha}_\infty(B_1\cap \Scal(\bar v_\infty)) > 0.
\]
Since $y_0\neq 0$ and $\bar v$ is radially homogeneous, this implies that $\bar v_\infty$ is translation-invariant along some line through the origin. In other words, up to rotation, $\bar v_\infty(x_1,\dots, x_{n+1}) = \bar w_\infty(x_1,\dots, x_{n})$, with
\[
\Hcal^{n-2+\alpha}_\infty(B_1\cap \Scal(\bar w_\infty)) > 0
\]
However, by our inductive hypothesis, we must have $\dim_\Hcal(\Scal (\bar w_\infty)) \leq n-2$, which yields the desired contradiction.
\end{proof}

\begin{lemma}\label{l:2d-finite-components}
    Let $n=1$ and let $u$ be a solution of \eqref{eq innervariation}-\eqref{eq:differential inequality}. Suppose that $x_0\in \{v=0\}$. Then there exists $r_0 > 0$ (depending on $x_0$) such that $\{v>0\}\cap B_{r_0}(x_0)$ has finitely many connected components.
\end{lemma}

\begin{proof}
We may without loss of generality assume that $x_0=0$. We divide the proof into steps as follows.

    \textit{Step 1.} We first demonstrate that $\{v=0\}$ has finite length inside any annulus centered at the origin contained in $B_{r_0}$, for any $r_0>0$ small enough. In light of Theorem \ref{t:reg-set} and Theorem \ref{t:2d-isolated-sing}, there exists $r_0>0$ such that $\mathcal{S}(u)\cap B_{r_0} = \{0\}$, and $\Rcal(u)\cap B_{r_0}$ consists of analytic curves (possibly infinitely many) terminating at the origin. Let $0<r<s\leq \frac{r_0}{2}$ and let $\vphi\in C_c^\infty(B_{r_0};[0,\infty))$ be such that $\mathbf{1}_{B_s \setminus \overline{B}_r} \leq \vphi \leq \mathbf{1}_{B_{2s} \setminus \overline{B}_{r/2}}$. Let $\{D_i\}_{i\in\N}$ denote the connected components of $B_{r_0}$ and let $v_i = v|_{D_i}$, extended by zero to $B_{r_0}$. Then each $v_i$ is Lipschitz, Lemma \ref{l:reflection} and an analogous computation to \eqref{eq int by parts} together guarantee that, since $2\Delta v_i = G'(v_i)$ in $D_i$, we have
    \begin{align*}
        2\sum_i \int_{\partial\{v_i=0\}} |\nabla v_i| \varphi  d\Hcal^n &= - {2}\sum_i\int_{\{v_i>0\}} \nabla v_i\cdot \nabla \varphi  -\sum_i \int_{\{v_i>0\}} G'(v_i) \varphi\,.
    \end{align*}
    Since $N_{v,x}(0^+) = 1$ for each $x\in B_{2s}\setminus \overline{B}_{r/2}$, the same argument as in Step 2 of Theorem \ref{t:reg-set} guarantees that $|\nabla v|$ does not vanish anywhere on $\{v=0\}\cap (B_{s}\setminus \overline{B}_{r})$, and thus
    \[
        \Hcal^n(\{v=0\}\cap (B_{s}\setminus \overline{B}_{r})) \leq C(r,s)\,. 
    \]
    \textit{Step 2.} Let us now conclude that there exists $r_1 \leq \tfrac{r_0}{2}$ such that for any $0<r<r_1$, under the additional assumption that $\{v=0\}$ has transverse intersection with $\partial B_{r}$, then $\{v=0\}$ consists of finitely many curves in $B_r$. From this, the conclusion will follow, in light of the transversality of smooth parametric families of maps to a given smooth submanifold (which follows from Sard's Theorem). Indeed, the latter together with the regularity of $\{v=0\}$, tells us that for almost-every $\rho \in (0,r_1)$, $\{v=0\}$ is transverse to $\partial B_\rho$. 
    
    Fix $r_1$ arbitrarily, to be determined later. First of all, observe that the conclusion of Step 1 guarantees that $\{v=0\}\cap (\overline{B}_{r_1}\setminus B_{r})$ consists of countably many disjoint curves $\gamma_i: [0,1] \to \overline{B}_{r_1}\setminus B_{r}$, $i\in\N$, and at most finitely many of them have $\gamma_i(0)\in \partial B_{r_1}$ and $\gamma_i(1)\in \partial B_{r}$ (or vice versa).
    
    In addition, we claim that only finitely many of them can have both $\gamma_i(0)$ and $\gamma_i(1)$ lying on $\partial B_{r}$. Indeed, if there are infinitely many, then the transversality assumption combined with an additional application of the conclusion of Step 1 implies that there must exist a closed embedded curve $\Ccal\subset \{v=0\}$ contained in the interior of $B_{r_1}$. This in turn produces a connected component $U$ of $\{v>0\}$ contained strictly in the interior of $B_{r_1}$. We claim that for $r_1$ sufficiently small (depending implicitly on $x_0$ which we have taken to be the origin), this is not possible. This follows the reasoning of \cite[Proposition 6.2]{CTV-three-phase}, which we repeat here for convenience. First of all, consider the rescaling $v_{r_1} \equiv v_{0,r_1}$ as in \eqref{e:rescaling}. In light of Lemma \ref{lemma strong convergence blow-ups}, we have the identity 
    \[
        \Delta v_{r_1} = \frac{r_1^2}{2H(r_1)^{1/2}}G'\big(v_{r_1} H(r_1)^{1/2}\big)\,,
    \]
    inside the rescaled component $\tilde U := r_1^{-1}U$. Testing this against $v_{0,r_1}$ (which can be done since $v$ has zero boundary data in $U$) and integrating by parts, we obtain the Poincar\'{e} inequality
    \[
        \int_{\tilde U} |\nabla v_{r_1}|^2 \leq \frac{kr_1^2}{2} \int_{\tilde U} v_{r_1}^2\,,
    \]
    where $k= \sup_{[0,1]}|G''|$. Choosing $r_1$ sufficiently small such that $\tfrac{kr_1^2}{2} < \lambda_1(B_1)$, where $\lambda_1(B_1)$ denotes the lowest Dirichlet eigenvalue of the unit ball (which is an explicitly computable constant), we arrive at a contradiction.
    
    Observe that this argument further tells us that we cannot have any connected components of $\{v>0\}$ in $B_r$, and thus, again combining with the transverse intersection assumption, we deduce that the only possibility is that $\{v=0\}\cap B_r$ consists of a finite number of curves with either both endpoints on $\partial B_r$, or with one endpoint on $\partial B_r$ and one endpoint at the origin.
\end{proof}

We finish this section with the proof of our main theorem. Our proof of the uniqueness of blow-ups at singular points in the planar case is a well know argument (see, e.g, \cite{TT}) which exploits the expansion of solutions to elliptic equations around critical points in the plane  \cite[Theorem 1]{hartman1953local}.
\begin{proof}[\textit{Proof of Theorem \ref{thm:main regularity theorem}.ii}]
The conclusions of Part \textbf{(ii)} when $n\geq 2$, together with the regularity of $\Rcal(u)$ when $n=1$, follow immediately from Corollary \ref{corollary decomp fb}, Theorem \ref{t:reg-set} and Corollary \ref{c:sing-set-dim}. It merely remains to characterize the behavior of $\{u=1\}$ at points in $\Scal(u)$ when $n=1$.
Letting $v=1-u$, from Theorem \ref{t:2d-isolated-sing} we know that $\mathcal{S}(u)$ is discrete. Thus, for $x_0 \in \mathcal{S}(u)$, in virtue of Lemma \ref{l:2d-finite-components}, there exists $r_0>0$ such that  $\{v>0\}\cap B_{r_0}(x_0)$ has a finite number $\ell$ of connected components. Assuming without loss of generality that $x_0=0$, let us consider the function $w(\rho,\theta) = v(\rho^2,2\theta)$ written in polar coordinates $(\rho, \theta)$. Notice that $\{w>0\}\cap B_{r_0}$ has $2\ell$ connected components $\{C_{i}\}_{i=1}^{2\ell}$ labelled so that $\partial C_{i}\cap \partial C_{i+1} \cap B_{r_0}\neq \varnothing$ for $i=1,\cdots, 2\ell-1$ and $\partial C_{2\ell}\cap \partial C_{1} \cap B_{r_0}\neq \varnothing$. Consider now the function $z = \sum_{i=1}^{2\ell} (-1)^i w|_{C_i}$. We claim that
\begin{equation}\label{eq z bvp}
\Delta z(x) = 2|x|^2 \tilde{H}(z(x)) \quad \mbox{$x \in B_{r_0}$,}
\end{equation}
with $\tilde{H}$ given by \eqref{eq odd reflection}, and that \eqref{eq z bvp} implies the desired conclusion.

Assuming for a moment the validity of the claim \eqref{eq z bvp}, since $f(x)=2|x|^2\frac{\tilde{H}(z(x))}{z(x)}$ is continuous, \eqref{eq z bvp} falls under the hypotheses of \cite[Theorem 1]{hartman1953local} with this choice of $f$, and $d=e=0$ (see also \cite{hartman1955local} for a ``modern'' formulation of the result), implying that $z$ admits a unique asymptotic expansion in polar coordinates of the form
\begin{equation}\label{eq HW}
z(\rho, \theta) = c_1 \rho^{L}\sin(L\theta)+c_2 \rho^{L}\cos(L\theta)+o\Big(\rho^L\Big),
\end{equation}
as $\rho \to 0^+$ for some $c_1,c_2\in \R$ and $L\in \N$. Notice that this combined with Lemma \ref{lemma:planar tangent classification} and Lemma \ref{lemma prop blowups} implies that the tangent function $\bar v$ of $v$ at $0$ is unique and that $c_1 = \frac{1}{\sqrt{\pi}}$ and $L ={2\ell} = 2 N_{v,0}(0^+)$ in the expansion \eqref{eq HW}, as desired.

We finish the argument by proving \eqref{eq z bvp}.  Let us observe first that when $z>0$, 
\begin{eqnarray*}
\Delta z(\rho,\theta)&=&  \partial_{\rho \rho}z +\frac{\partial_\rho z}{\rho} +\frac{1}{\rho^2}\partial_{\theta \theta} z\\
&=& 4\rho^2\Big( \partial_{\rho \rho}v(\rho^2,2\theta)+{\frac{1}{\rho^2}}\partial_{\rho }v(\rho^2,2\theta)+\frac{1}{\rho^4}\partial_{\theta \theta} v(\rho^2,2\theta)\Big)\\
&=& 2\rho^2 G'(v)(\rho^2,2\theta),
\end{eqnarray*}
similarly we have that if $z<0$, $\Delta z = -2\rho^2G'(-v)(\rho^2,2\theta)$. Lastly, since for each connected component $C_i$ of $w > 0$, $\partial C_i \cap (B_{r_0}\setminus \{0\})$ is a union of regular curves in virtue of Theorem \ref{t:reg-set} and the normal derivatives of $z$ on each side of $\partial C_i $ match for $i=1,\cdots ,2\ell$, we have that $\Delta z(x) = 2|x|^2 \tilde{H}(z(x))$ holds in $B_{r_0}\setminus \{0\}$ but since $z$ is continuous up to the origin, we conclude that actually \eqref{eq z bvp} holds.
\end{proof}

\appendix



\section{Variational estimates}
Here we collect some basic variational estimates relating to minimizers of \eqref{new model intro 2} and \eqref{new model volume constrained 2}, mostly contained in \cite{maggi2023hierarchy}. We begin with the following lemma, quoted from \cite[Lemma 4.5]{maggi2023hierarchy}, giving the inner variation formulae for the energy and volume. 

\begin{lemma}\label{lemma volume fixing}
{\bf (i):} If $F$, $V$ are $C^1$, $A\subset\R^{n+1}$ is open, $X\in C^\infty_c(A;\R^{n+1})$, and $f_t(x)=x+t\,X(x)$, then there are positive constants $t_0$ and $C_0$ depending on $X$ only, such that, for every $|t|<t_0$, $f_t:A\to A$ is a diffeomorphism, and for every $w\in W^{1,2}(A;[0,1])$ we have
\begin{eqnarray}\nonumber
&&\Big|\int_A |\nabla (w \circ f_t)|^2 + F(w \circ f_t) - \int_A |\nabla w|^2 + F(w)
\\\nonumber
&&\hspace{1cm}-t\,\int_A\big[|\nabla w|^2+F(w)\big]\,\Div\,X-2(\nabla w)\cdot\nabla X[\nabla w]\Big|\le C_0t^2 \int_A |\nabla w|^2 + F(w)\,,
\\\label{pushing}
&&\Big|\int_A V(w\circ f_t)-\int_A V(w)-t\,\int_A\,V(w)\,\Div\,X\Big|\le C_0t^2 \int_A V(w)\,,
\end{eqnarray}
{\bf (ii):} If $F$, $V$ are $C^1$ and $V$ satisfies \eqref{eq:A3}, $A\subset\R^{n+1}$ is open, $u\in L^1(A;[0,1])$ and $u$ is not constant on $A$, then there are positive constants $\eta_0$, $t_0$, $\beta_0$, and $C_0$ and a one parameter family of diffeomorphisms $\{f_t\}_{|t|<t_0}$, all depending on $A$ and $u$, such that $f_0=\id$, $\{f_t\ne\id\}\cc A$, {$f_t(x) = x + tX(x)$ for some $X \in C_c^\infty(A;\mathbb{R}^{n+1})$,} and for every $w\in W^{1,2}(A;[0,1])$ with $\|u-w\|_{L^1(A)}\leq \beta_0$ and $|\eta|<\eta_0$, there is $t=t(\eta)\in (-t_0,t_0)$ such that $w_t=w\circ f_t$ satisfies
\[
\int_A V(w_t)=\int_A V(w)+\eta\,,\qquad \Big| \int_A |\nabla w_t|^2 + F(w_t) - |\nabla w|^2 - F(w)\Big|\le C_0|\eta|\int_A |\nabla w|^2 + F(w)\,\,.
\]
\end{lemma}

\begin{proof}[Outline of Proof]
The first item follows from the area formula and does not depend on the form of $V$. The second item is the volume-fixing variations argument for perimeter (\cite[Lemma 29.13, Theorem 29.14]{maggi2012sets}) adapted to the Allen-Cahn setting. The only required property of $V$ is that the non-constancy of $u$ implies that $V(u)$ is non-constant also; see \cite[Proof of Lemma 4.5.(ii)]{maggi2023hierarchy}. But this is guaranteed by our assumption \eqref{eq:A3} that $V$ is strictly increasing.
\end{proof}

\begin{corollary}\label{corollary:almost min appendix}
If $F,V$ are $C^1$, $\wire=\mathbb{R}^{n+1}\setminus \Om$ is compact, $\mathcal{C}$ is a spanning class for $\Om$, and $u$ is a minimizer for \eqref{new model volume constrained 2}, then there exists positive $\tilde{r}$ and $\tilde{C}$, both depending on $u$, such that for all $w\in W^{1,2}(\Om;[0,1])$ with $\{w\geq t\}$ $\mathcal{C}$-spanning $\wire$ for all $t\in (1/2,1)$ and $\{u\neq w\}\subset B_{\tilde{r}}(x_0)\cap \Om$ for some $x_0\in \Omega$,
\begin{equation}\label{eq:almost min corollary}
\int_\Om |\nabla u|^2 + F(u)\,dx \leq \int_\Om |\nabla w|^2 + F(w)\,dx + \tilde{C}\left|1- \int_\Om V(w)\right|\,.
\end{equation}
\end{corollary}

\begin{proof}
First, note that $u$ cannot be constant on $\Om$ since $\int_\Om V(u)=1$ and $\Om$ is unbounded. Let $A_i\subset \Om$ for $i=1,2$ be such that $\dist(A_1,A_2)>0$ and $u$ is non-constant on each $A_i$. Then Lemma \ref{lemma volume fixing}.(ii) applies to $u$ and $A_i$, so we may choose $\tilde{r}$ small enough so that if $w\in W^{1,2}(\Om;[0,1])$ and $\{u\neq w\}\subset B_{\tilde{r}}(x_0)\cap \Om$ for any $x_0$, then $|1-\int V(w)|< \eta_0$ and $B_{\tilde{r}}(x_0)$ is disjoint from at least one $A_i$. By fixing the volume of $w$ on this $A_i$ via $w \circ f_{t(\eta)}$ as in the previous lemma, we may modify it so that the modification has volume $1$. In addition, we claim that this modification preserves the spanning constraint in \eqref{new model volume constrained}. Indeed, if $B$ is Borel, then $f_t^{-1}(B^\one) = (f_t^{-1}(B))^\one$ and $f_t^{-1}(B^\zero) = (f_t^{-1}(B))^\zero$ (both immediate consequences of the area formula) imply that
\begin{equation}\label{eq:preservation of essential boundaries}
\pa^e (f_t^{-1}(B)) = \big((f_t^{-1}(B))^\one \cup (f_t^{-1}(B))^\zero\big)^c=\big(f_t^{-1}(B^\one) \cup f_t^{-1}(B^\zero)\big)^c = f_t^{-1}(\pa^e B)\,;
\end{equation}
also, due to the closure of $\mathcal{C}$ under homotopy, 
\begin{equation}\label{eq:preservation of gamma}
f_t^{-1}\circ \gamma \in \mathcal{C} \quad \forall |t|<t_0\,.
\end{equation}
By \eqref{eq:preservation of essential boundaries}-\eqref{eq:preservation of gamma}, the verification of \eqref{eq:c spanning for functions} for $w\circ f_t$ via Definition \ref{def:span borel} with a triple $(\gamma,\Psi,T)$ reduces to the validity of the same condition for $w$ on $(f_t^{-1}\circ \gamma,f_t^{-1}\circ \Psi, f_t^{-1}\circ T)\in \mathcal{T}(\mathcal{C})$.
Testing the minimality of $u$ against this modification and using the estimates from Lemma \ref{lemma volume fixing}.(ii) concludes the argument.
\end{proof}

\section{Preliminaries for Theorem \ref{thm:main existence theorem}}\label{sec: app c}


The following important lemma will allow us to work with the spanning condition of Definition \ref{def:span closed} in place of that of Definition \ref{def:span borel} for continuous functions.

\begin{lemma}[Spanning for continuous functions]\label{lemma:equiv for continuous functions}
If $\wire=\mathbb{R}^{n+1}\setminus \Om$ is compact, $\mathcal{C}$ is a spanning class for $\wire$, $\delta\in (1/2,1]$, $u\in (W^{1,2}_\loc\cap C^0)(\Om;[0,1])$, and $\{u^* \geq t\}=\{u\geq t\}$ is $\mathcal{C}$-spanning $\wire$ for every $t\in (1/2,\delta)$, then $\{u\geq \delta\}$ is $\mathcal{C}$-spanning $\wire$.
\end{lemma}

\begin{proof}
By Remark \ref{remark:consistency of spanning defs}, it suffices to show that for every $\gamma\in \mathcal{C}$, $\{u\geq \delta\} \cap \gamma \neq \varnothing$. Pick a sequence $\{t_j\}\subset (1/2,\delta)$ such that $t_j \nearrow \delta$. Since $\{u\geq t_j\}$ is closed for every $t$, Remark \ref{remark:consistency of spanning defs} implies that $\{u\geq t_j\}$ is $\mathcal{C}$-spanning in the sense of Definition \ref{def:span closed}, that is there exists $x_j\in \gamma_j$ such that $u(x_j)\geq t_j$. By the compactness of $\gamma$, there must therefore be $x\in \gamma$ such that $x_j \to x$. Thus by the continuity of $u$, 
\[
u(x) = \lim_{j\to\infty} u(x_j) \geq \lim_{j\to\infty} t_j = \delta,
\]
and so $\{u \geq \delta\} \cap \gamma \neq \varnothing$.
\end{proof}

We additionally require the following lemma, which guarantees that our admissible class $\Acal$ of Definition \ref{d:admissible-class} contains all continuous functions with gradients in $L^2_\loc$.

{\begin{lemma}[Generalized admissible class contains original one]\label{l:gen-adm-class-larger}
    Suppose that $\wire=\mathbb{R}^{n+1}\setminus \Om$ is compact and that $\mathcal{C}$ is a spanning class for $\wire$. Then we have the containment
    \[
        \{u\in C(\Omega;[0,1]): \nabla u \in L^2(\Omega)\} \subset \Acal\,.
    \]
\end{lemma}

\begin{proof}
    Let $B_r(x) \cc \Omega$, consider a cup competitor $g_i=1-w_i$ associated to a component $C_i$ for $u$ in $B_r(x)$, and let $\gamma\in \Ccal$. Note that when $u$ is continuous, the essential partition $\{C_i\}_i$ of $\{v^*>0\}\cap B_r(x)$ for $v=1-u$ simply consists of the connected components of $B_r(x)\setminus \{u=1\}$. Moreover, by construction, $g_i$ is also continuous in this case. In light of Lemma \ref{lemma:equiv for continuous functions} (see Remark \ref{r:consistency-spanning-cts}), we may apply \cite[Lemma 10]{DGM17} to $K= \{u=1\}$ in $B_r(x)$ and $\gamma$, which tells us that either $\gamma \cap (\{u=1\}\setminus B_r(x)) \neq \emptyset$, or $\gamma \cap \overline{B_r(x)}$ is homeomorphic to a closed interval with endpoints belonging to two distinct connected components of $\overline{B_r(x)}\setminus \{u=1\}$. If neither of these connected components contains $C_i$, then either $\gamma \cap B_{r/2}(x) = \emptyset$ in which case 
    \[
        \gamma\cap (\{g_i=1\}\cap \overline{B_r(x)}) = \gamma \cap (\{u=1\}\cap \overline{B_r(x)}) \neq \emptyset\,,
    \]
    or $\gamma \cap B_{r/2}(x) \neq \emptyset$ in which case $\gamma$ intersects $\partial B_{r/2}(x)\setminus \overline{C_i}$, where $w_i=0$ and thus $g_i=1$ by construction.

    It thus remains to treat the case when one of the distinct connected components of $\overline{B_r(x)}\setminus \{u=1\}$ containing an endpoint of $\gamma$ also contains $C_i$. In this case, recalling the definition of $w_i$, see Definition \ref{d:cup-comp}, we simply observe that since the semilinear replacement of $(v-v_i) \big|_{\partial B_{r/2}(x)}$ is positive in $\overline{B_{r/2}(x)}$ and the radial cutoff $\vphi$ is positive, $w_i$ preserve the property of disconnecting points of $C_i\cap \partial B_r(x)$ from $\partial B_r(x)\setminus \overline{C_i}$.
\end{proof}}

We need to know that spanning is preserved for $L^1_\loc$-limits of sequences with uniform Dirichlet energy bounds; this is guaranteed by the following theorem, which was originally proved in \cite{maggi2023hierarchy}.

\begin{theorem}[Compactness]\label{thm:prelim compactness thm}
If $\wire\subset \mathbb{R}^{n+1}$ is compact, $\mathcal{C}$ is a spanning class for $\wire$, \{$\delta_j\}_j\subset (1/2,1]$, $\delta_j\to \delta_0\in (1/2,1]$, and $\{u_j\}\subset W^{1,2}_\loc(\Omega;[0,1])$ are such that $u_j \to u$ in $L^1_\loc(\Om)$ for some $u$, $\{u_j^* \geq t\}$ is $\mathcal{C}$-spanning for each $j$ and $t\in [1/2,\delta_j)$, and
\begin{equation}\label{eq:uniform dirichlet bound}
\sup_j \int_{\Om} |\nabla u_j|^2 \,dx < \infty\,,
\end{equation}
then $\{u^* \geq t\}$ is $\mathcal{C}$-spanning for every $t\in [1/2,\delta_0)$.
\end{theorem}

\begin{proof}[Outline of Proof] Fix a triple $(\gamma,\Phi,T)\in \mathcal{T}(\mathcal{C})$ for which we must verify that Definition \ref{def:span borel} holds for every $\{u^* \geq t\}$ with $t\in [1/2,\delta_0)$. We modify our function so as to allow for the application of \cite[Theorem 3.2]{maggi2023hierarchy}. Let $w\in W^{1,2}_\loc(\Om;[0,1])$ be such that
\begin{equation}\label{eq:w props}
\mbox{$w=0$ on $\cl T$, $w=1$ on $\Om \setminus B_R(0)$ for some large $R$}
\end{equation}
and consider the functions
\begin{equation}\label{eq:max def}
v_j = \max\{u_j,w\}\,,\qquad v = \max\{u,w\}\,.
\end{equation}
Note that
\begin{align}\label{eq:loc L1 conv 2}
&\sup_j \int_\Om |\nabla v_j|^2 \, dx < \infty\,,
\\
\label{eq:loc L1 conv}
&\mbox{ $v_j \overset{L^2_\loc}{\longrightarrow} v$, and $\{v_j^* \geq t\}$ is $\mathcal{C}$-spanning $\wire$ for every $t\in [1/2,\delta_j)$ }\,.
\end{align}
Since $v=u$ on $\cl T$, the super-level sets of $v$ satisfy the spanning condition on this fixed tube $T$ if and only if those of $u$ do as well. So it suffices to explain why $\{v^*\geq t\}$ satisfies Definition \ref{def:span borel} for this $(\gamma,\Phi,T)$. This can be done by following the compactness result \cite[Theorem 3.2]{maggi2023hierarchy}, which gives conditions under which the spanning condition is preserved under limits of functions. Our assumptions \eqref{eq:loc L1 conv 2}-\eqref{eq:loc L1 conv} on $v_j$ the same as in \cite[Theorem 3.2]{maggi2023hierarchy} up to the facts that there, the uniform bound
\begin{equation}\label{eq:allencahn bound}
\sup_j  \int_\Om \e|\nabla v_j|^2 + \frac{W(v_j)}{\e}\,dx<\infty\,,\qquad \e>0\,,
\end{equation}
where $W$ is a double-well potential with $W(1)=0=W(0)$ is assumed instead of \eqref{eq:loc L1 conv 2}, and the functions $v_j$ are assumed to belong to $L^2$ rather than the $L^2_\loc$. By \eqref{eq:w props}-\eqref{eq:max def}, the functions $v_j$ satisfy \eqref{eq:allencahn bound}, and the class $L^2_\loc$ is enough to repeat \cite[Proof of Theorem 3.2]{maggi2023hierarchy} verbatim. (The spanning condition on  a single tube $(\gamma,\Phi,T)$ is local in nature, in that it does not depend on the values of $v$ outside $T$, so this last claim should be heuristically clear without referencing the details of \cite[Proof of Theorem 3.2]{maggi2023hierarchy}.)
\end{proof}
\begin{lemma}[Non-triviality of \eqref{new model intro 2}-\eqref{new model volume constrained 2}]\label{lemma:inf not infinity}
If $F$ and $V$ are continuous with $F(0)=0=V(0)$ and $V(t)>0$ for $t\in (0,1]$, $\wire= \mathbb{R}^{n+1}\setminus \Om$ is compact, and $\mathcal{C}$ is a spanning class for $\wire$ satisfying \eqref{eq:spanning class assumption}, then \eqref{new model intro 2} and \eqref{new model volume constrained 2} have finite infimums.
\end{lemma}

\begin{proof}
We will proceed to explicitly construct an admissible function $u$ for both \eqref{new model intro 2} and \eqref{new model volume constrained 2} with finite energy. For some $\delta>0$ fixed, let us consider the $\wire_\delta$ to be those points in $\Om$ such that $\dist(x,\wire)\leq \delta$. We claim that:
\begin{equation}\label{eq:gamma diameter property}
    \mbox{if $\gamma \in \mathcal{C}$ and (the image of) $\gamma$ is contained in $\wire_\delta^c$, then $\diam \gamma > \delta/2$}.
\end{equation}
Indeed, if this were not the case, and there were some $\gamma$ with image contained in $\wire_\delta^c$ with diameter no more than $\delta/2$, then, choosing some ball $\wire_\delta^c\supset B_{\delta/2}(x)\supset \gamma$, we would have $\dist(B_{\delta/2}(x),\wire)\geq \delta/2$. But then $\gamma \subset B_{\delta/2}(x)\subset \Omega$, and thus $\gamma$ is homotopic to a point, contradicting \eqref{eq:spanning class assumption}. {Let $R>0$ be such that $\wire \subset B_R$. Now, defining the grid $G_{\delta/2}=\cup_{z\in \mathbb{Z}^{n+1}}\tfrac{\delta}{2\sqrt{n+1}}\big(z+\partial( [0,1]^{n+1})\big)$ of diameter $\tfrac{\delta}{2}$, we claim that
\begin{equation}\label{eq:grid c spanning set}
    \mbox{$\big[(\Omega \cap \wire_\delta) \cup \partial B_R \cup ( B_R \cap G_{\delta/2})\big]\cap \gamma \neq \varnothing$ for all $\gamma \in \mathcal{C}$}\,.
\end{equation}
To see that this is the case, first notice that if $\dist (\gamma,\wire)\leq \delta$, then clearly the intersection is non-empty. On the other hand, if $\gamma \subset \wire_\delta^c$, then it must intersect $\partial B_R \cup ( B_R \cap G_{\delta/2})$, because otherwise it would be contained in a single cube and contradict \eqref{eq:gamma diameter property} or be contained in $\overline{B}_R^c$ and again be homotopic to a point, contradicting \eqref{eq:spanning class assumption}. Finally, for $\varepsilon>0$ to be determined, we define
\begin{equation}\notag
    u(x) = \max\{1-\dist(x, \wire_\delta \cup \partial B_R \cup ( B_R \cap G_{\delta/2}))/\varepsilon, 0\}\,.
\end{equation}
Since $u$ is Lipschitz with compact support and $\{u=1\}$ contains the $\mathcal{C}$-spanning set from \eqref{eq:grid c spanning set}, $u$ is admissible in \eqref{new model intro 2}. Furthermore, for \eqref{new model volume constrained 2}, note that $\int_\Om V(u)$ is continuous and increasing in $\varepsilon$, so the intermediate value theorem yields some $\varepsilon$ such that $u$ satisfies the volume constraint. Finally, clearly $u$ has finite energy, since it has compact support and is Lipschitz.}\end{proof}

\begin{proof}[Proof of Theorem \ref{thm:existence for isop problem}]
The argument is the same as the one in \cite[Theorem A.1]{maggi2023hierarchy} and depends on \eqref{eq:A4 redux}. Let $\{w_j\}_j$ be a minimizing sequence for $\Psi(v)$. By the P\'{o}lya-Szeg\"{o} inequality, we may as well assume that $w_j(x)=g_j(|x|)$ are radially decreasing. Due to the uniform Dirichlet and $L^\infty$ bounds on $w_j$, there exists $w\in L^1_\loc(\rn;[0,1])$ with finite Dirichlet energy such that $w_j \to w$ in $L^1_\loc$, $w(x)=g(|x|)$, and
\begin{equation}\notag
\int_{\mathbb{R}^{n+1}}|\nabla w|^2 + F(w)\,dx \leq \liminf_{j\to \infty}\int_{\mathbb{R}^{n+1}}|\nabla w_j|^2 + F(w_j)\,dx\,.
\end{equation}
To show that $w$ is a minimizer for \eqref{eq:isop problem def}, we only need to show that $\int_{\rn}V(w)\,dx = v$, which would follow from showing that
\begin{equation}\label{eq:tight}
\lim_{R\to \infty}\sup_j\int_{B_R^c}V(w_j)\,dx = 0\,.
\end{equation}
Since $g_j\to g$ a.e.~ on $(0,\infty)$ and $g$ is radially decreasing, it follows that $\lim_{R\to \infty}\sup_j g_j(R)=0$. Now by \eqref{eq:A4 redux}, we estimate
\begin{equation}\label{eq:no mass escaping isop}
0\leq  \lim_{R\to \infty} \sup_j \int_{B_R^c}V(g_j(|x|))\,dx \leq \lim_{R\to \infty}
\begin{cases}
	\sup_j \displaystyle\frac{V(g_j(R))}{F(g_j(R))}\displaystyle\int_{B_R^c} F(g_j(|x|))\,dx & g_j(R) \neq 0 \\ 
	0 & g_j(R)=0\,;
\end{cases}
\end{equation}
note that $g_j(R)\neq 0$ implies that $F(g_j(R))\neq 0$ by \eqref{eq:A4 redux} and \eqref{eq:A3}. By using $\lim_{R\to \infty}\sup_j g_j(R)=0$ in \eqref{eq:no mass escaping isop}, we find \eqref{eq:tight}. The fact $w>0$ follows from the Euler-Lagrange equations as in \eqref{eq:nontriviality of u on conn comp}.
\end{proof}

\bibliographystyle{alpha}
\bibliography{ref}

\end{document}